\documentclass[11pt, oneside, reqno]{amsart}

\newcommand{\hkra}{\hookrightarrow}

\newcommand{\ra}{\rightarrow}

\newcommand{\sub}{\subseteq}

\newcommand{\s}[2]{\sum\limits_{#1}{#2} }

\newcommand{\inv}{^{-1}}

\DeclareMathOperator\Spec{Spec}

\newcommand{\cl}{\colon}

\newcommand{\xra}{\xrightarrow}

\newcommand{\eva}{\normalfont\text{ev}}

\newcommand{\lspan}[1]{\langle {#1}\rangle}

\newcommand{\chml}{\check{H}}

\newcommand{\fcolim}{\varinjlim\limits}

\newcommand{\eval}{\eva}

\newcommand{\cupr}{\smile}

\newcommand{\dul}{^\vee}

\newcommand{\vfc}[1]{[#1]^{\virt}}

\newcommand{\fcl}[1]{[#1]}

\newcommand{\reg}{\normalfont\text{reg}}

\newcommand{\GL}{\normalfont\text{GL}}
\newcommand{\PGL}{\normalfont\text{PGL}}
\newcommand{\PU}{\normalfont\text{PU}}

\newcommand{\coker}{\normalfont\text{coker}}
\newcommand{\pr}{\normalfont\text{pr}}

\newcommand{\Mbar}{\overline{\mathcal M}}

\newcommand{\vdim}{\text{vdim}}
\newcommand{\delbar}{\bar\partial}

\newcommand{\Aut}{\normalfont\text{Aut}}

\newcommand{\bC}{\mathbb{C}}
\newcommand{\bD}{\mathbb{D}}

\newcommand{\bG}{\mathbb{G}}
\newcommand{\bH}{\mathbb{H}}

\newcommand{\bK}{\mathbb{K}}

\newcommand{\bP}{\mathbb{P}}
\newcommand{\bQ}{\mathbb{Q}}
\newcommand{\bR}{\mathbb{R}}

\newcommand{\bZ}{\mathbb{Z}}

\newcommand{\cB}{\mathcal{B}}
\newcommand{\cC}{\mathcal{C}}

\newcommand{\cE}{\mathcal{E}}
\newcommand{\cF}{\mathcal{F}}
\newcommand{\cG}{\mathcal{G}}
\newcommand{\cH}{\mathcal{H}}
\newcommand{\cI}{\mathcal{I}}
\newcommand{\cJ}{\mathcal{J}}
\newcommand{\cK}{\mathcal{K}}
\newcommand{\cL}{\mathcal{L}}
\newcommand{\cM}{\mathcal{M}}
\newcommand{\cN}{\mathcal{N}}
\newcommand{\cO}{\mathcal{O}}

\newcommand{\cS}{\mathcal{S}}
\newcommand{\cT}{\mathcal{T}}
\newcommand{\cU}{\mathcal{U}}
\newcommand{\cV}{\mathcal{V}}

\newcommand{\fA}{\mathfrak{A}}
\newcommand{\fB}{\mathfrak{B}}
\newcommand{\fC}{\mathfrak{C}}

\newcommand{\fF}{\mathfrak{F}}

\newcommand{\fL}{\mathfrak{L}}
\newcommand{\fM}{\mathfrak{M}}

\newcommand{\ff}{\mathfrak{f}}
\newcommand{\fg}{\mathfrak{g}}

\newcommand{\fl}{\mathfrak{l}}

\newcommand{\fs}{\mathfrak{s}}

\newcommand{\fu}{\mathfrak{u}}

\newcommand{\virt}{\normalfont\text{vir}}
\newcommand{\rank}{\normalfont\text{rank}}

\newcommand{\pcd}{\normalfont\text{PD}}

\usepackage[utf8]{inputenc}
\usepackage[english]{babel}
\usepackage{amsmath}
\usepackage{amssymb}
\usepackage{amsfonts}
\usepackage{amsthm}
\usepackage{enumerate}
\usepackage[dvipsnames]{xcolor}
\usepackage{tikz-cd}
\usepackage{subfiles}
\usepackage[scr]{rsfso}
\usepackage{extarrows}
\usepackage{hyperref}
\usepackage{appendix}
\usepackage{bm}
\hypersetup{
    colorlinks=true,
    linkcolor=Blue,
    filecolor=red,      
    urlcolor=teal,
    citecolor=teal,
}
\newtheorem{theorem}{Theorem}[section]
\newtheorem{lemma}[theorem]{Lemma}
\newtheorem{corollary}[theorem]{Corollary}
\newtheorem{proposition}[theorem]{Proposition}

\newtheorem{construction}[theorem]{Construction}

\theoremstyle{definition}
\newtheorem{definition}[theorem]{Definition}
\newtheorem{discussion}[theorem]{Discussion}
\theoremstyle{remark}
\newtheorem{remark}[theorem]{Remark}

\newtheorem{example}[theorem]{Example}
\newtheorem*{notation*}{Notation}

\numberwithin{equation}{subsection}

\makeatletter
\@addtoreset{equation}{section}
\@addtoreset{theorem}{section}
\makeatother

\DeclareMathOperator{\Pic}{Pic}
\DeclareMathOperator{\ind}{ind}

\newcommand{\Addresses}{{
  \bigskip
  \footnotesize

  \textsc{Amanda Hirschi, Sorbonne Université}\par\nopagebreak
  \textit{E-mail address}: \texttt{hirschi@imj-prg.fr}
  
\smallskip
  
    \textsc{Mohan Swaminathan, Stanford University}\par\nopagebreak
    \textit{E-mail address}: \texttt{mohans@stanford.edu}
}}

\begin{document}

\title[Global charts and a product formula in GW theory]{Global Kuranishi charts and a product formula in symplectic Gromov--Witten theory}
\author{Amanda Hirschi, Mohan Swaminathan}

\begin{abstract}
    We construct global Kuranishi charts for the moduli spaces of stable pseudoholomorphic maps to a closed symplectic manifold in all genera. This is used to prove a product formula for symplectic Gromov--Witten invariants. As a consequence we obtain a K\"unneth formula for quantum cohomology. 
\end{abstract}

\maketitle
\tableofcontents
\section{Introduction}

\subsection{Background}
Let $(X,\omega)$ be a closed symplectic manifold. Let $\cJ_\tau(X,\omega)$ be the space of $\omega$-tame almost complex structures and let $J\in\cJ_\tau(X,\omega)$. The fundamental object of study in Gromov--Witten theory, as introduced in \cite{KM94}, is the moduli space $\Mbar_{g,n}(X,A;J)$ of stable $J$-holomorphic maps representing a class $A\in H_2(X,\bZ)$. To construct invariants using this moduli space, one must contend with two technical issues: lack of transversality and smoothness of gluing. 

The first of these issues occurs when we have non-regular $J$-holomorphic curves $u\cl (C,j)\to X$, i.e., those for which the linearization of the Cauchy--Riemann operator $\delbar_{J,j}$ at $u$ is not surjective. In general, this issue persists even if we choose a generic $\omega$-tame $J$, due to the presence of multiply covered curves.

To explain the second issue, note that we need gluing theorems to show that the moduli space is a manifold (or orbifold) near nodal holomorphic curves, even when we have transversality. In general, determining whether two local charts for the moduli space constructed via gluing are smoothly compatible is very subtle.

When $(X,\omega)$ is semi-positive and $J$ is generic, one can control the codimension of the non-regular curves. In this situation, \cite{RT95,RT97} gave a construction of Gromov--Witten invariants using pseudocycles. For general symplectic manifolds, the aforementioned issues require the use of \emph{virtual techniques}. The last decades have seen the development of several virtual frameworks in symplectic geometry, see \cite{LT98,FO99,CM,HWZ17,MW17b,P16,IP19} for example. Most of these constructions begin by representing the moduli space locally using Kuranishi charts. They then employ delicate local-to-global arguments to extract from this a global invariant, called the virtual fundamental class, a generalization of the fundamental class for moduli spaces containing non-regular curves.

The recent work \cite{AMS21} showed how to construct a global Kuranishi chart for $\Mbar_{0,n}(X,A;J)$ by building on ideas from \cite{siebert-98}. This eliminates the need for patching together local information and simplifies the definition of Gromov--Witten invariants considerably. A \emph{global Kuranishi chart} $\cK$ for a compact space $Z$ consists of a finite-rank vector bundle $\cE$ over a manifold $\cT$ both equipped with an almost free action by a compact Lie group $G$ and an equivariant section $\fs \cl \cT\to \cE$, so that $\fs\inv(0)/G\cong Z$. 
We refer to $\cT,\cE$ and $\fs$ as the \emph{thickening}, \emph{obstruction bundle} and \emph{obstruction section} respectively.
An \emph{orientation} on $\cK$ consists of orientations on $\cT$ and $\cE$ which are both $G$-invariant and an orientation on $\fg = \text{Lie}(G)$. The \emph{virtual fundamental class} $\vfc{Z}_\cK \in \chml^{\vdim(Z)}(Z;\bQ)\dul$ associated to an oriented global Kuranishi chart $\cK$ for $Z$ is defined to be the linear map 
\begin{equation*} \chml^{\vdim(Z)}(Z;\bQ) \xra{\fs^*\tau_{\cE/G}\cup(\cdot)} H_c^{\dim(\cT/G)}(\cT/G;\bQ)\xra{[\cT/G]} \bQ
\end{equation*}
where $\vdim(Z) := \dim(\cT)-\dim(G)-\rank(\cE)$ is the \emph{virtual dimension} of $Z$. Here $\tau_{\cE/G}$ is the Thom class of $\cE/G$ and $[\cT/G]$ is the fundamental class of $\cT/G$ in Borel--Moore homology; see \textsection\ref{subsec:gkc-vfc} for more details. The virtual fundamental class of $Z$ can be thought of as the Euler class localized to the zero locus of $\fs$. As in \cite{P16}, we do not need $\cT$ to be a smooth manifold.

\subsection{Main results} In this article, we construct global Kuranishi charts for moduli spaces of stable pseudo-holomorphic maps, generalizing the work of \cite{AMS21} to all genera. 
 
\begin{theorem}[Theorem \ref{global-kuranishi-existence}]\label{intro-main-thm} 
For integers $g,n\geq 0$ and $A\in H_2(X,\bZ)$, the moduli space of stable maps $\Mbar_{g,n}(X,A;J)$ admits an oriented global Kuranishi chart $\cK = (G,\cT,\cE,\fs)$ of the correct virtual dimension, depending on certain auxiliary data, but unique in the following sense.
	\begin{enumerate}[\normalfont(i)]
		\item Different choices of auxiliary data result in global Kuranishi charts which are related by certain equivalence moves that do not affect the virtual fundamental class.
		\item Given any other $J'\in\cJ_\tau(X,\omega)$, there exist auxiliary data such that the associated global Kuranishi charts are cobordant.
	\end{enumerate}
In particular, $\Mbar_{g,n}(X,A;J)$ admits a well-defined virtual fundamental class.
\end{theorem}

We have the natural map
\begin{align*}
	\text{ev}\times\text{st}\cl \Mbar_{g,n}(X,A;J)&\xrightarrow{} X^n\times\Mbar_{g,n}
\end{align*}
given by evaluation at the marked points and stabilization of the domain, where we follow the convention that $\Mbar_{g,n}$ is a point for $2g-2+n\le 0$. 
Using this map, we define the \emph{Gromov--Witten class} of $(X,\omega)$ to be 
\begin{equation}\label{gw-class-de}  \text{GW}^{(X,\omega)}_{A,g,n}:=(\text{ev}\times\text{st})_*[\Mbar_{g,n}(X,A;J)]^\text{vir}\in H_{\vdim}(X^n\times\Mbar_{g,n},\bQ).\end{equation}
This class contains the same information as the \emph{Gromov--Witten homomorphism}
$$\text{GW}^{(X,\omega)}_{A,g,n}\cl H^*(X^n;\bQ)\to H^*(\Mbar_{g,n};\bQ)$$
defined by the formula
 $$ \text{GW}^{(X,\omega)}_{A,g,n}(\alpha) = \pcd(\normalfont\text{st}_*(\eva^*\alpha \cap \vfc{\Mbar_{g,n}(X,A;J)})),$$
 where $\pcd$ denotes the Poincar\'e duality isomorphism with $\bQ$-coefficients for orbifolds. The Kontsevich-Manin axioms in \cite{KM94} are defined in terms of this homomorphism. They are satisfied by the invariants defined here, as was shown in \cite{Hir23} by the first author. Specializing to $g = 0$ and $n = 3$, we obtain the \emph{small quantum cohomology ring} $QH^*(X,\omega)$ over the universal Novikov ring $\Lambda_0$ as in \cite[Chapter 11]{MS12}.\par
The global Kuranishi charts of Theorem \ref{intro-main-thm} can be shown to be equivalent in genus $0$ to the global Kuranishi charts constructed by \cite{AMS21} and thus, the associated Gromov--Witten invariants agree; see Remark \ref{AMS-comparison} for more details.

\begin{remark}[Smoothness]\label{rel-smooth-remark} The thickening $\cT$ we construct is not smooth, but it admits a topological submersion $\pi$ to a smooth manifold $\cB$ and carries a \emph{canonical} structure of a \emph{rel--$C^\infty$ manifold} over $\cB$. This relies on the results of \cite{Swa21}, recalled in \textsection\ref{subsec:rel-smooth-review} and explained further in Appendix \ref{rel-smooth-addendum}. 
	In particular, we can use smoothing theory as in \cite[\textsection 4.2]{AMS21} to upgrade our construction to a smooth global Kuranishi chart for $\Mbar_{g,n}(X,A,J)$. This allows for the definition of a Morava K-theory valued virtual fundamental class as in \cite{AMS21}. Also, as explained in \cite{BX22}, our global chart can be used as an input for the construction of $\bZ$-valued Gromov--Witten type invariants in all genera.
\end{remark}

As an application of the construction, we prove a product formula for the Gromov--Witten invariants of a product symplectic manifold. This result was conjectured in \cite{KM94} and shown in \cite{KM96,Beh99} in the algebraic setting. In the symplectic category, \cite{RT95} prove the product formula for semi-positive symplectic manifolds in genus $0$. See \cite{Hy12} for related work.

\begin{theorem}[Product formula, Theorem \ref{gw-product}]\label{gw-product-formula} Let $(X_i,\omega_i)$ for $i=0,1$ be closed symplectic manifolds and set $(X,\omega) = (X_0,\omega_0)\times (X_1,\omega_1)$. Let $A_i\in H_2(X_i;\bZ)$ and $g,n\geq 0$ be such that $2g-2+n > 0$ and $(g,n)$ is neither $(1,1)$ nor $(2,0)$. Then, for any $\alpha_i\in H^*(X_i^n;\bQ)$, we have
 \begin{align*}
 	\s{({\pr_i})_*A = A_i}{\normalfont\text{GW}^{(X,\omega)}_{g,n,A}}(\alpha_0\times\alpha_1) = \normalfont\text{GW}^{(X_0,\omega_0)}_{g,n,A_0}(\alpha_0)\cupr\normalfont\text{GW}^{(X_1,\omega_1)}_{g,n,A_1}(\alpha_1),
 \end{align*}
 where the sum on the left side ranges over all $A\in H_2(X,\bZ)$ which map to $A_i$ under the natural projection $\pr_i:X_0\times X_1\to X_i$ for $i=0,1$.
 \end{theorem}

As a consequence of this product formula, we deduce a K\"unneth formula for quantum cohomology.
 
 \begin{corollary}[Corollary \ref{quantum-kunneth}]\label{quantum} The K\"unneth map induces an isomorphism
 \begin{align*}
 	QH^*(X_0)\otimes_{\Lambda_0}QH^*(X_1)\to QH^*(X_0\times X_1)
 \end{align*}
 	of $\Lambda_0$-algebras, where $\Lambda_0$ is the universal Novikov ring.
 \end{corollary}

Theorem \ref{gw-product-formula} is used crucially in \cite{HW24}.

During the completion of this work, we were informed by the authors of \cite{AMS21} about an independent construction of global Kuranishi charts for moduli spaces of stable maps of positive genus. This has now appeared in \cite{AMS23}.

\subsection{Overview of the paper}
In \textsection\ref{sec:prelims}, we recall several preliminary results required for the global Kuranishi chart construction. In particular, in \textsection\ref{subsec:AMS-review}, we review the construction of \cite{AMS21} and discuss the difficulties in generalizing it to positive genus. In \textsection\ref{subsec:rel-smooth-review}, we define rel--$C^\infty$ manifolds and review the results of \cite{Swa21}; a refinement, necessary for proving the relative smoothness of the obstruction section, is stated and shown in Appendix \ref{rel-smooth-addendum}. In \textsection\ref{subsec:gkc-vfc}, we discuss rel--$C^\infty$ global Kuranishi charts, the notions of equivalence and cobordism for them and the associated virtual fundamental classes. The impatient reader is advised to read \textsection\ref{overview} first and to refer to \textsection\ref{sec:prelims} as needed.

In \textsection\ref{overview}, we define the auxiliary data required for the construction of the global Kuranishi chart, which is described precisely in Construction \ref{high-level-description-defined}. Suitable auxiliary data are constructed in \textsection\ref{achieving-transversality-proof}. The existence statement in Theorem \ref{intro-main-thm} is proved in \textsection\ref{transversality-implies-smoothness-proof}, while the uniqueness statements are established in \textsection\ref{proof-of-equivalence-and-cobordism}.

In \textsection\ref{products}, we establish the product formula, Theorem \ref{gw-product}. Appendices \ref{line-bundles-on-families-of-curves} and \ref{rel-smooth-addendum} contain technical results necessary for showing the relative smoothness of the obstruction section.
 
\subsection*{Acknowledgements} A.H. thanks her advisor Ailsa Keating for helpful discussions, her encouragement and feedback. She is grateful to Soham Chanda, Andrin Hirschi, Noah Porcelli and Luya Wang for interesting conversations. M.S. thanks Shaoyun Bai, Deeparaj Bhat, Yash Deshmukh and John Pardon for useful comments. 
Both authors thank Mohammed Abouzaid, Mark McLean and Ivan Smith for useful discussions. They are also grateful to the anonymous referee for several insightful comments and suggestions which improved the exposition considerably. Part of this work was done while both authors were in residence at SLMath (formerly known as MSRI) during the Fall 2022 semester, supported by NSF Grant No. DMS-1928930. 
A.H. was supported in her graduate studies by the EPSRC and is currently supported by ERC Grant No. 864919. 

\subsection*{Conventions}
The symbol $\otimes$ denotes tensor product over $\bC$ unless indicated otherwise. The notation for complex projective space $\bC\bP^N$ is abbreviated to $\bP^N$. Smooth manifolds are always assumed to be Hausdorff.
\section{Preliminaries}\label{sec:prelims}

In this section, we collect some relevant background material. In \textsection\ref{subsec:AMS-review}, we describe Abouzaid--McLean--Smith's construction of global Kuranishi charts for genus $0$ Gromov--Witten moduli spaces. In \textsection\ref{subsec:rel-smooth-review}, we discuss some results about canonical rel--$C^\infty$ structures on moduli spaces and obstruction bundles.

\subsection{Abouzaid--McLean--Smith's genus $0$ construction}\label{subsec:AMS-review}

To put Theorem \ref{intro-main-thm} in context, we briefly sketch the genus $0$ construction from \cite[\textsection 6]{AMS21} without proof. We then explain the difficulties in directly adapting it to the general genus case. For simplicity, we restrict to the case without marked points. 

Fix a closed symplectic manifold $(X,\omega)$, a homology class $A\in H_2(X,\bZ)$ and $J\in\cJ_\tau(X,\omega)$. We will describe how to construct a global Kuranishi chart (in the sense of \cite[Definition 4.1]{AMS21}) for the stable map moduli space $\Mbar_0(X,A;J)$.

\begin{definition}[Polarization]\label{polarization-defined}
    A \emph{polarization} on $X$ taming $J$ is a Hermitian line bundle $\cO_X(1)\to X$ equipped with a Hermitian connection $\nabla$ with curvature form $-2\pi \text{i}\Omega$ where $\Omega$ is a symplectic form on $X$ taming $J$.
\end{definition}

\begin{remark}
    In \cite{AMS21}, the polarization $\cO_X(1)$ is denoted by $L_\Omega$.
\end{remark}

Given any compact subset of $\cJ_\tau(X,\omega)$, there exists a polarization on $X$ taming all the almost complex structures in this set. We give a proof of this fact in Lemma \ref{Suitable Line Bundle on Target} below as this is used in the proof of Theorem \ref{intro-main-thm} as well.

\begin{lemma}[Polarization: existence]\label{Suitable Line Bundle on Target}
	There is a complex line bundle $\cO_X(1)\to X$ with a Hermitian metric $\langle\cdot,\cdot\rangle$ and a Hermitian connection $\nabla$ whose curvature is given by $-2\pi{\normalfont\text{i}}\Omega$ where $\Omega$ is a symplectic form taming $J$. 
 
    In fact, given any compact subset $K\subset\cJ_\tau(X,\omega)$ containing $J$, we can choose $\cO_X(1)$ to be such that $\Omega$ tames each $J'\in K$.
\end{lemma}
\begin{proof}
    By approximating $[\omega]\in H^2(X,\bR)$ by an element of $H^2(X,\bQ)$ and multiplying by a large positive integer to clear denominators, we first choose a symplectic form $\Omega$ taming (each almost complex structure in) $K$ such that $[\Omega]$ has an integral lift $h\in H^2(X,\bZ)$. This is possible as being symplectic and taming $K$ are both open properties of closed $2$-forms. Now, let $\cL$ be a complex line bundle on $X$ with $c_1(\cL) = h$. Choose any Hermitian metric $\langle\cdot,\cdot\rangle$ on $\cL$ and a compatible Hermitian connection $\nabla'$ on $\cL$ and write the curvature as $-2\pi \text{i} \Omega'$. Since $h$ is a common integral lift of $[\Omega]$ and $[\Omega']$, we can find a smooth (real) $1$-form $\beta$ such that $\Omega' = \Omega + d\beta$. The connection $\nabla = \nabla' + 2\pi \text{i} \beta$ now is also Hermitian for $\langle\cdot,\cdot\rangle$ and has curvature given by $-2\pi \text{i}\Omega$. Thus, we may take $\cO_X(1)$ to be the line bundle $\cL$ equipped with the metric $\langle\cdot,\cdot\rangle$ and compatible Hermitian connection $\nabla$.
\end{proof}

Choose a polarization $\cO_X(1)\to X$ taming $J$. Let $-2\pi\text{i}\Omega$ be the curvature form of the connection $\nabla$ on $\cO_X(1)$ and write $d := \langle[\Omega],A\rangle$.

\begin{definition}\label{AMS-framed-curve}
    A \emph{framed genus $0$ curve in $X$} is a tuple $(C,u,F)$ where
    \begin{enumerate}
        \item $C$ is a nodal genus $0$ curve, 
        \item $u\cl C\to X$ is a smooth map with $u_*[C] = A$ with $\int u^*\Omega\ge 0$ (resp. $>0$) on each irreducible (resp. unstable irreducible) component of $C$,
        \item $F = (f_0,\ldots,f_d)$ is a complex basis of the space of global holomorphic sections of the line bundle $u^*\cO_X(1)$, equipped the holomorphic structure given by $(u^*\nabla)^{0,1}$, such that the $(d+1)\times(d+1)$ Hermitian matrix
        \begin{align*}
            \cH(C,u,F) = \left(\int_C\langle f_i,f_j\rangle\, u^*\Omega\right)_{0\le i,j\le d}
        \end{align*}
        is positive definite (i.e., has strictly positive eigenvalues).
    \end{enumerate}
    Note that conditions (1) and (2) imply that $u^*\cO_X(1)$ has vanishing $H^1$ and thus, we have $\dim_\bC H^0(C,u^*\cO_X(1)) = d+1$ by Riemann--Roch.
    
    An \emph{equivalence} of framed genus $0$ curves $(C,u,F)$ and $(C',u',F')$ in $X$ is a biholomorphic map $\varphi\cl C\to C'$ such that $u'\circ\varphi = u$ and $\varphi^*F' = F$.
\end{definition}

Any $[u\cl C\to X]\in\Mbar_0(X,A;J)$ satisfies (1) and (2) of Definition \ref{AMS-framed-curve}. For any choice of basis $F = (f_0,\ldots,f_d)$ of $H^0(C,u^*\cO_X(1))$, condition (3) is a consequence of the $\Omega$-tameness of $J$ and the $J$-holomorphicity of $u$.

\begin{definition}\label{AMS-base}
    Let $\Mbar_0^*(\bP^d,d)$ be the moduli space  of degree $d$ genus $0$ stable holomorphic maps to $\bP^d$ whose image is not contained in any hyperplane. Let
    \begin{center}
    \begin{tikzcd}
        \cC \arrow[r] \arrow[d] & \bP^d \\
        \Mbar_0^*(\bP^d,d) & 
    \end{tikzcd}    
    \end{center}
    be the universal family of stable maps over this moduli space, i.e., its fibre over a point $[v\cl C\to \bP^d]$ of $\Mbar_0^*(\bP^d,d)$ is the curve $C$ mapping to $\bP^d$ via the map $v$. 
\end{definition}

\cite[Lemma 6.4]{AMS21} shows that $\Mbar_0^*(\bP^d,d)$ and $\cC$ are smooth quasi-projective varieties. There is a natural forgetful map from equivalence classes of framed genus $0$ curves in $X$ to points of $\Mbar_0^*(\bP^d,d)$ described as follows. Given $(C,u,F)$, we map it to $\phi_F = [f_0:\cdots:f_d]\cl C\to\bP^d$. This yields a well-defined holomorphic embedding $\iota_F\cl C\hookrightarrow\cC$ which identifies $C$ with the fibre over $[\phi_F\cl C\to\bP^d]$ of the universal curve $\cC$.\\\\
Next, choose
\begin{enumerate}[(i)]
    \item a relatively ample line bundle $\cL$ on $\cC\to\Mbar_0^*(\bP^d,d)$, equipped with a Hermitian metric so that the natural $U(d+1)$ action on $\cC$ induced by the linear action $U(d+1)$ action on $\bP^d$ lifts to a unitary action on $\cL$,
    \item a $\bC$-linear connection on ${T^*_\cC}^{0,1}$, which is invariant under the $U(d+1)$ action,
    \item a $\bC$-linear connection
    and tangent bundle $T_X$, viewed as a complex vector bundle using the almost complex structure $J$ and
    \item a sufficiently large positive integer $k\gg 1$.
\end{enumerate}

\begin{definition}
    The \emph{thickening} $\cT$ is the moduli space of tuples $(C,u,F,\eta)$ where
    \begin{enumerate}
        \item $(C,u,F)$ is a \emph{framed curve of genus $0$ in $X$} and
        \item $\eta\in H^0(C,\iota_F^*({T^*_\cC}^{0,1})\otimes u^*T_X\otimes \cL^{\otimes k})\otimes\overline{H^0(C,\cL^{\otimes k})}$,
    \end{enumerate}
    satisfy the equation
    \begin{align*}
        \delbar_Ju + \langle\eta\rangle\circ d\tilde\iota_F = 0.    
    \end{align*}
    Here, the holomorphic structures on $\iota_F^*({T^*_\cC}^{0,1})$ and $u^*T_X$ come from the $\bC$-linear connections fixed above, $\tilde\iota_F$ is the pullback of $\iota_F\cl C\hookrightarrow\cC$ to the normalization $\tilde C\to C$ and $\langle\cdot\rangle$ is induced by the Hermitian inner product on $\cL^{\otimes k}$.
    
    The \emph{obstruction bundle} $\cE\to\cT$ is the vector bundle with fibre
    \begin{align*}
        \left(H^0(C,\iota_F^*({T^*_\cC}^{0,1})\otimes u^*T_X\otimes \cL^{\otimes k})\otimes\overline{H^0(C,\cL^{\otimes k})}\right)\oplus\cH_{d+1}
    \end{align*}
    over $(C,u,F,\eta)\in\cT$, with $\cH_{d+1}$ being the space of $(d+1)\times(d+1)$ Hermitian matrices. The \emph{obstruction section} $\fs$ is the section of $\cE$ over $\cT$ defined by
    \begin{align*}
        (C,u,F,\eta)\mapsto(\eta,\log\cH(C,u,F)),
    \end{align*}
    where $\log$ denotes the inverse of the exponentiation map from Hermitian matrices to positive definite Hermitian matrices. The group $G = U(d+1)$ has a natural action on $\cT$, which lifts to $\cE$ so that $\fs$ becomes a $G$-equivariant section.
\end{definition}

\begin{theorem}[{\cite[Proposition 6.1]{AMS21}}]
    When $k\gg 1$, there is an open neighborhood of $\fs^{-1}(0)\subset\cT$ over which the forgetful map $\cT\to\Mbar_0^*(\bP^d,d)$ is a topological submersion and $\cE$ is a vector bundle. Moreover, $\fs$ is a continuous section of $\cE\to\cT$ and the map $\fs^{-1}(0)/G\to\Mbar_0(X,A;J)$ given by $[C,u,F,0]\mapsto[u\cl C\to X]$ is a homeomorphism.
\end{theorem}

\begin{remark}
    The bijection between $\fs^{-1}(0)/G$ and $\Mbar_0(X,A;J)$ is due to the fact that $\fs^{-1}(0)$ consists of framed genus $0$ curves $(C,u,F)$ in $X$ such that $u$ is $J$-holomorphic and the basis $F$ is unitary, i.e., $\cH(C,u,F)$ is the identity matrix.
\end{remark}

\begin{remark}
    The fact that $\cT\to\Mbar_0^*(\bP^d,d)$ is a topological submersion is crucially used in \cite{AMS21} to construct a $G$-invariant smooth structure on $\cT$. In a bit more detail, one first shows that $\cT\to\Mbar_0^*(\bP^d,d)$ has a well-defined vertical tangent bundle after which it becomes possible to apply $G$-equivariant smoothing theory to construct the desired smooth structure.
\end{remark}

\begin{discussion}
Let us now consider how to extend this construction to the case of stable maps of any genus. Let $u\cl C\to X$ be a $J$-holomorphic stable map of positive genus.

First, the holomorphic line bundle $u^*\cO_X(1)$ may have non-vanishing $H^1$, i.e., the dimension of $H^0(C,u^*\cO_X(1))$ may jump. To resolve this issue, we replace $u^*\cO_X(1)$ by $(\omega_C\otimes(u^*\cO_X(1))^{\otimes 3})^{\otimes p}$ for some $p\gg 1$, where $\omega_C$ denotes the dualizing line bundle of $C$. This ensures that the relevant $H^1$ vanishes.

Second, even after replacing $u^*\cO_X(1)$ as above, the analogue of the forgetful map is not a topological submersion. To resolve this, we need to consider a version of Definition \ref{AMS-framed-curve} where, in place of $u^*\cO_X(1)$, we allow any holomorphic line bundle $L$ on $C$ which is topologically isomorphic to $(\omega_C\otimes(u^*\cO_X(1))^{\otimes 3})^{\otimes p}$ and then take a complex basis of $H^0(C,L)$.
  
\end{discussion}

\subsection{Rel--$C^\infty$ manifolds and representability}\label{subsec:rel-smooth-review}

We recall the definition of the category of rel--$C^\infty$ manifolds and summarize the main results from \cite{Swa21}. These will play an important role in our proof of Theorem \ref{global-kuranishi-existence}, which is a more precise formulation of Theorem \ref{intro-main-thm} in the introduction.

\subsubsection{Generalities}  For a continuous map of spaces $p\cl Y\to S$, we abbreviate $(Y,S,p\cl Y\to S)$ to simply $Y/S$ when $p$ is clear from the context.

\begin{definition}[Rel--$C^\infty$ maps]
    Given spaces $S$ and $T$ and integers $m,n\ge 0$, suppose $U\subset\bR^n\times S$ and $V\subset\bR^m\times T$ are open subsets. A \emph{rel--$C^\infty$ map} $\Phi = (\tilde\varphi,\varphi)\cl U/S\to V/T$ consists of a continuous map $\varphi\cl S\to T$ and a continuous map $\tilde\varphi\cl U\to V$ having the form
    \begin{align*}
        \tilde\varphi(x,s) = (F(x,s),\varphi(s)),
    \end{align*}
    where $F\cl U\to\bR^m$ is such that all its partial derivatives $D^\alpha_xF\cl U\to\bR^m$ with respect to $x\in\bR^n$ are defined and continuous as functions on $U$.
\end{definition}

\begin{definition}[Rel--$C^\infty$ manifolds]
    For a continuous map of spaces $p\cl Y\to S$ and an integer $n\ge 0$, a \emph{chart} of relative dimension $n$ is an open subset $U\subset Y$ together with an open embedding $\varphi\cl U\to \bR^n\times S$ which is compatible with the projection maps to $S$. Two charts $(U,\varphi)$ and $(V,\psi)$ of relative dimension $n$ for $Y/S$ are said to be \emph{rel--$C^\infty$ compatible} if
    \begin{align*}
        (\psi\circ\varphi^{-1},\text{id}_S)\cl \varphi(U\cap V)/S\to\psi(U\cap V)/S
    \end{align*}
    is a \emph{rel--$C^\infty$ diffeomorphism}, i.e., a rel--$C^\infty$ map with a rel--$C^\infty$ inverse.
    
    A \emph{rel--$C^\infty$ structure} on $Y/S$ is a maximal atlas of rel--$C^\infty$ compatible charts for $Y/S$. In this situation, we refer to $Y/S$ with this atlas as a \emph{rel--$C^\infty$ manifold} or we say that $Y$ has the structure of a rel--$C^\infty$ manifold over $S$.
\end{definition}

\begin{definition}[Rel--$C^\infty$ category]
    Given rel--$C^\infty$ manifolds $Y/S$ and $Z/T$, a rel--$C^\infty$ map $\Phi=(\tilde\varphi,\varphi)\cl Y/S\to Z/T$ is a commutative diagram
    \begin{center}
    \begin{tikzcd}
        Y \arrow[d] \arrow[r,"\tilde\varphi"] & Z \arrow[d] \\
        S \arrow[r,"\varphi"] & T
    \end{tikzcd}    
    \end{center}
    such that $\varphi$ and $\tilde\varphi$ are continuous and, in any local charts belonging to the maximal atlases of $Y/S$ and $Z/T$, the pair $(\tilde\varphi,\varphi)$ induces a rel--$C^\infty$ map between open subsets of $\bR^n\times S$ and $\bR^m\times T$ for some integers $m,n\ge 0$.

    The \emph{category of rel--$C^\infty$ manifolds}, denoted by $(C^\infty/\cdot)$, is the category whose objects are rel--$C^\infty$ manifolds and whose morphisms are rel--$C^\infty$ maps. Any smooth manifold $M$ can be considered as a rel--$C^\infty$ manifold $M/\text{pt}$, where $\text{pt}$ denotes the space with one point.
\end{definition}

\begin{definition}[Rel--$C^\infty$ vector bundles]
    Let $Y/S$ be a rel--$C^\infty$ manifold, and let $E$ be a topological $\bK$-vector bundle of rank $r$ on the space $Y$, where $\bK$ denotes $\bR$ or $\bC$. Two local trivializations of $E$ are said to be \emph{rel--$C^\infty$ compatible} if their transition function is a rel--$C^\infty$ function valued in $\text{GL}(r,\bK)$.
    
    A \emph{rel--$C^\infty$ vector bundle} on $Y/S$ is a topological vector bundle $E$ on $Y$ along with a maximal atlas of rel--$C^\infty$ compatible trivializations. In this situation, $E/S$ is naturally a rel--$C^\infty$ manifold and thus, there is a well-defined notion of rel--$C^\infty$ sections of a rel--$C^\infty$ vector bundle.
\end{definition}

\begin{example}
    Any rel--$C^\infty$ manifold $Y/S$ has a well-defined \emph{vertical tangent bundle} denoted by $T_{Y/S}$. This is a rel--$C^\infty$ real vector bundle on $Y/S$.
\end{example}

\begin{definition}[Rel--$C^\infty$ submersion]
    A morphism $f\cl Y'/S\to Y/S$ in $(C^\infty/\cdot)$, where the underlying map $S\to S$ is the identity, is called a rel--$C^\infty$  submersion if the differential $df\cl T_{Y'/S}\to f^*T_{Y/S}$ is surjective.
\end{definition}

Rel--$C^\infty$ submersions are given by coordinate projections in suitable local coordinates just like $C^\infty$ submersions. This follows from the inverse function theorem with parameters, \cite[Lemma 5.10]{Swa21}. Similarly, we have a well-behaved notion of \emph{vertical transversality} for rel--$C^\infty$ maps, which is the analogue of usual transversality formulated using the vertical tangent bundle. The following observation was not stated explicitly in \cite{Swa21}.

\begin{lemma}\label{fibre-product}
    Suppose $Y'/S\to Y/S$ is a rel--$C^\infty$ submersion. Then, $Y'/Y$ is naturally a rel--$C^\infty$ manifold given by the categorical fibre product
    \begin{align*}
        Y'/Y = (Y'/S)\times_{(Y/S)}(Y/Y)
    \end{align*}
    in the category $(C^\infty/\cdot)$ of rel--$C^\infty$ manifolds.
\end{lemma}
\begin{proof}
    This statement is local on $Y'$ and $Y$. Thus we may reduce to the case of $Y'=\bR^{m+n}\times S$ and $Y = \bR^n\times S$ with $p$ given by a coordinate projection. Now, the statement can be checked directly.
\end{proof}

\subsubsection{Moduli spaces}
Holomorphic curves in almost complex manifolds can be used to define moduli functors on the category of rel--$C^\infty$ manifolds.

Let $\fB$ be a complex manifold and $g,m\ge 0$ be integers. Let $\pi\cl\fC\to \fB$ be a flat analytic family of prestable curves (i.e., projective, connected curves with at worst nodal singularities) of arithmetic genus $g$ with $m$ marked points, given by analytic sections $p_1,\ldots,p_m\cl\fB\to \fC$. The fibre of this family over any point $s\in \fB$ will be written as $(C_s,p_1(s),\ldots,p_m(s))$.

\begin{definition}[Moduli functors]\label{rel-smooth-moduli-functors}
    Let $(X,J)$ be a smooth almost complex manifold and let $A\in H_2(X,\bZ)$. Define the functor
    \begin{align*}
        \fM^\text{reg}(\pi,X)_A\cl (C^\infty/\cdot)^\text{op}\to(\text{Sets})
    \end{align*}
    as follows. Given a rel--$C^\infty$ manifold $Z/T$, we define $\fM^\text{reg}(\pi,X)_A(Z/T)$ to be the set of all commutative diagrams
    \begin{center}
    \begin{tikzcd}
        \fC \arrow[d,"\pi"] & \arrow[l] \fC_T \arrow[d] & \arrow[l] \fC_Z \arrow[r,"f"] \arrow[d] & X \\
        \fB & \arrow[l,"\varphi"] T & \arrow[l] Z &
    \end{tikzcd}
    \end{center}
    where 
    \begin{itemize}
        \item $\varphi\cl T\to \fB$ is continuous,
        \item the two squares are pullbacks (i.e., $\fC_T = \fC\times_\fB T$ and $\fC_Z = \fC_T\times_T Z$),
        \item  the map $f\cl\fC_Z/\fC_T\to X/\text{pt}$ is rel--$C^\infty$ (with the rel--$C^\infty$ structure on $\fC_Z/\fC_T$ being pulled back from $Z/T$) whose restriction to each fibre of $\fC_Z\to Z$ is an unobstructed $J$-holomorphic map in class $A$.
    \end{itemize}  
    Explicitly, the last condition says that for $z\in Z$ with image $s\in \fB$, the restriction $f_z\cl C_s\to X$ of the map $f$ to $\fC_Z|_z = C_s$ is $J$-holomorphic, $(f_z)_*[C_s] = A$ and the linearization
    \begin{align*}
        D(\delbar_J)_{f_z}\cl\Omega^0(C_s,f_z^*T_X)\to\Omega^{0,1}(\tilde C_s,\tilde f_z^*T_X)
    \end{align*}
    is surjective. Here, $\tilde C_s\to C_s$ denotes the normalization of $C_s$ and $\tilde f_z$ is the pullback of $f_z$ to $\tilde C_s$. To complete the definition of $\fM^\text{reg}(\pi,X)_A$, note that such diagrams can be pulled back along any morphism $Z'/T'\to Z/T$ in $(C^\infty/\cdot)$. 
    
    Additionally, define
    \begin{align*}
        \Mbar^\text{reg}(\pi,X)_A\cl (C^\infty/\cdot)^\text{op}\to(\text{Sets})
    \end{align*}
    to be the subfunctor of $\fM^\text{reg}(\pi,X)_A$ consisting of commutative diagrams as above which satisfy the additional condition that for any $z\in Z$ with image $s\in \fB$, the number of biholomorphic automorphisms of $C_s$ which preserve $p_1(s),\ldots,p_m(s)$ and $f_z$ is finite, i.e., $(C_s,p_1(s),\ldots,p_m(s),f_z)$ is a stable map.
\end{definition}

Recall that for any category $\mathscr{C}$, we say that a functor $\mathscr{F}\cl\mathscr{C}^\text{op}\to(\text{Sets})$ \emph{represented} by an object $M$ of $\mathscr{C}$ if there is a natural isomorphism
\begin{align*}
    \eta\cl\text{Hom}_{\mathscr{C}}(-,M)\to\mathscr{F}
\end{align*}
of functors. By the Yoneda lemma, $\eta$ is specified completely by $\eta(\text{id}_M)\in\mathscr{F}(M)$. The pair $(M,\eta)$, if it exists, is unique up to unique isomorphism.

The main result of \cite{Swa21}, recalled below, is that the functors defined above are representable. The representing objects are the moduli spaces, endowed with rel--$C^\infty$ structures.

\begin{theorem}[Rel--$C^\infty$ structures on moduli spaces]\label{rel-smooth-representability}
    The functors $\fM^{\normalfont\text{reg}}(\pi,X)_A$ and $\Mbar^{\normalfont\text{reg}}(\pi,X)_A$ are represented by rel--$C^\infty$ manifolds over $\fB$, whose underlying topological spaces are Hausdorff. Moreover, the inclusion
    \begin{align*}
        \Mbar^{\normalfont\text{reg}}(\pi,X)_A\subset\fM^{\normalfont\text{reg}}(\pi,X)_A
    \end{align*}
    is an open embedding.
\end{theorem}
\begin{proof}
    \cite[Theorems 3.5, 3.6 and Corollary 3.7]{Swa21} show that $\fM^{\normalfont\text{reg}}(\pi,X)_A$ and $\Mbar^{\normalfont\text{reg}}(\pi,X)_A$ are represented by Hausdorff\footnote{In fact, the cited statements show the \textit{a priori} weaker property of the rel--$C^\infty$ manifolds being separated over $\fB$. As the base $\fB$ is Hausdorff, the two notions agree.} rel--$C^\infty$ manifolds over $\fB$ and that the inclusion of the latter in the former is an open embedding.
\end{proof}

In view of Theorem \ref{rel-smooth-representability}, we will conflate the functors $\fM^{\normalfont\text{reg}}(\pi,X)_A$ and $\Mbar^{\normalfont\text{reg}}(\pi,X)_A$ with the rel--$C^\infty$ manifolds over $\fB$ that represent them.

\begin{remark}
    Points of $\fM^{\normalfont\text{reg}}(\pi,X)_A$ are given by pairs $(s,f)$, where $s\in \fB$ and $f\cl C_s\to X$ is an unobstructed $J$-holomorphic map in class $A$, as can be seen by taking $Z = T = \text{pt}$ in Definition \ref{rel-smooth-moduli-functors}. Among these, points of $\Mbar^\text{reg}(\pi,X)_A$ are the ones for which $(C_s,p_1(s),\ldots,p_m(s),f)$ is a stable map.
\end{remark}

The proof of Theorem \ref{rel-smooth-representability} given in \cite{Swa21} uses polyfold theory to show the existence of a representing object. The key elements of this argument are reviewed in Appendix \ref{rel-smooth-addendum}, where we also establish a useful refinement of Theorem \ref{rel-smooth-representability}.\par
 We emphasize that the rel--$C^\infty$ manifolds whose existence is asserted in Theorem \ref{rel-smooth-representability} are determined up to unique isomorphism by the corresponding moduli functors. Thus, any alternate proof method (e.g., using only usual gluing analysis) would necessarily produce a canonically isomorphic object.

\subsubsection{Obstruction bundles} We next describe a generalization of Theorem \ref{rel-smooth-representability}, which will allow us to describe the obstruction bundle in the rel--$C^\infty$ set-up. 

\begin{definition}[Obstruction bundle functor]\label{obstruction-bundle-functor}
    We continue with the same setup as in Definition \ref{rel-smooth-moduli-functors}. Additionally, also assume that we have a smooth $\bC$-vector bundle $p\cl F\to X$, equipped with a $\bC$-linear connection $\nabla$. Define
    \begin{align*}
        \fM^\text{reg}_F(\pi,X)_A\cl (C^\infty/\cdot)^\text{op}\to(\text{Sets})
    \end{align*}
    to be the subfunctor of $\fM^\text{reg}(\pi,X)_A$ consisting of diagrams as in Definition \ref{rel-smooth-moduli-functors} with the additional property that, for any $z\in Z$ with image $s\in \fB$, the restriction $f_z\cl C_s\to X$ of the map $f$ to $\fC_Z|_z = C_s$ satisfies
    \begin{align*}
        H^1(C_s,f_z^*F) = 0.    
    \end{align*}
    Here the vector bundle $f_z^*F$ is endowed with the holomorphic structure given by the $\bC$-linear Cauchy--Riemann operator $(f_z^*\nabla)^{0,1}$.

    Define the functor
    \begin{align*}
        \fF^\text{reg}_F(\pi,X)_A\cl (C^\infty/\cdot)\to(\text{Sets})
    \end{align*}
    as follows. Given a rel--$C^\infty$ manifold $Z/T$, we define $\fF^\text{reg}_F(\pi,X)_A(Z/T)$ to be the set of all commutative diagrams
    \begin{center}
    \begin{tikzcd}
        \fC \arrow[d,"\pi"] & \arrow[l] \fC_T \arrow[d] & \arrow[l] \fC_Z \arrow[r,"\tilde f"] \arrow[d] & F \\
        \fB & \arrow[l,"\varphi"] T & \arrow[l] Z &
    \end{tikzcd}
    \end{center}
    such that composing with the projection $p\cl F\to X$ yields a diagram
    \begin{center}
    \begin{tikzcd}
        \fC \arrow[d,"\pi"] & \arrow[l] \fC_T \arrow[d] & \arrow[l] \fC_Z \arrow[r,"f"] \arrow[d] & X \\
        \fB & \arrow[l,"\varphi"] T & \arrow[l] Z &
    \end{tikzcd}
    \end{center}
    belonging to $\fM^\text{reg}_F(\pi,X)_A(Z/T)$ and for, any $z\in Z$ with image $s\in \fB$, the section $\sigma_z$ of $f_z^*F$ corresponding to the map $\tilde f_z\cl C_s\to F$ is holomorphic with respect to the $\bC$-linear Cauchy--Riemann operator $(f_z^*\nabla)^{0,1}$, i.e., $\sigma_z\in H^0(C_s,f_z^*F)$. Here, $\tilde f_z\cl C_s\to F$ is the restriction of the $\tilde f$ to $\fC_Z|_z = C_s$ and $f_z = p\circ\tilde f_z\cl C_s\to X$.
\end{definition}

\begin{theorem}[Rel--$C^\infty$ obstruction bundle on moduli spaces]\label{rel-smooth-obs-bundle-rep}
    The functors $\fM^{\normalfont\text{reg}}_F(\pi,X)_A$ and $\fF^{\normalfont\text{reg}}_F(\pi,X)_A$ are represented by rel--$C^\infty$ manifolds over $\fB$, whose underlying topological spaces are Hausdorff. Moreover, the inclusion
    \begin{align*}
        \fM^{\normalfont\text{reg}}_F(\pi,X)_A\subset\fM^{\normalfont\text{reg}}(\pi,X)_A
    \end{align*}
    is an open embedding and the map
    \begin{align*}
        p_*\cl\fF^{\normalfont\text{reg}}_F(\pi,X)_A\to\fM^{\normalfont\text{reg}}_F(\pi,X)_A
    \end{align*}
    induced by the projection $p\cl F\to X$ defines a rel--$C^\infty$ $\bC$-vector bundle of rank $N(1-g)+\langle c_1(F),A\rangle$, where $N$ is the rank of the $\bC$-vector bundle $F$.
\end{theorem}
\begin{proof}
    This follows from \cite[Corollary 5.15 and Theorem 5.18]{Swa21}. To see that $\fM^{\normalfont\text{reg}}_F(\pi,X)_A$ and $\fF^{\normalfont\text{reg}}_F(\pi,X)_A$ are Hausdorff, we use the same argument as in the proof of Theorem \ref{rel-smooth-representability}.
\end{proof}

\begin{remark}
    A point $(s,f\cl C_s\to X)\in\fM^\text{reg}(\pi,X)_A$ lies in the open subset $\fM^\text{reg}_F(\pi,X)_A$ precisely when $H^1(C_s,f^*F) = 0$.
    
    The fibre of $p_*$ over a point $(s,f\cl C_s\to X)\in\fM^{\normalfont\text{reg}}_F(\pi,X)_A$ is the $\bC$-vector space $H^0(C_s,f^*F)$. Since we have $H^1(C_s,f^*F) = 0$, the Riemann--Roch formula determines the dimension of $H^0(C_s,f^*F)$.
\end{remark}

\subsection{Global Kuranishi chart and virtual fundamental class}\label{subsec:gkc-vfc}

We define global Kuranishi charts and the relations of equivalence and cobordism between them. Our definitions are slightly different from the ones in \cite[\textsection 4.1]{AMS21} due to the fact that we work in the rel--$C^\infty$ category. We then explain how an oriented global Kuranishi chart determines a virtual fundamental class.

\begin{definition}[Rel--$C^\infty$ global Kuranishi charts]\label{rel-smooth-global-chart-abstract-def}
    A \emph{rel--$C^\infty$ global Kuranishi chart} $\cK=(G,\cT/\cB,\cE,\fs)$ consists of
    \begin{enumerate}[(i)]
        \item a rel--$C^\infty$ manifold $\cT\to\cB$, called the \emph{thickening}, where the \emph{base space} $\cB$ is a smooth manifold;
        \item  a rel--$C^\infty$ vector bundle $\cE$ on $\cT/\cB$, the \emph{obstruction bundle}, and a rel--$C^\infty$ section $\fs$ of $\cE$, the \emph{obstruction section};
        \item a compact Lie group $G$, called the \emph{symmetry group}, which acts on $\cT/\cB$ and $\cE$ with finite stabilizers so that $\fs$ is $G$-equivariant. The action map
        \begin{align*}
            (G\times\cT)/(G\times\cB)\to\cT/\cB
        \end{align*}
        and its analogue for $\cE$ are both required to be rel--$C^\infty$ maps while the action of $G$ on $\cB$ is required to be smooth.
    \end{enumerate}
    If, in addition, we are given a Hausdorff space $Z$ and a homeomorphism 
    \begin{align*}
        \fs^{-1}(0)/G\xrightarrow{\simeq}Z,   
    \end{align*}
    then we say $\cK$ is a \emph{rel--$C^\infty$ global Kuranishi chart for} $Z$. The \emph{virtual dimension} of $Z$ with respect to $\cK$ is defined to be
    \begin{align*}
        \vdim_\cK\,Z = \dim\cT-\dim G-\rank\,\cE,
    \end{align*}
    where all dimensions are understood to be real dimensions.
    
    The rel--$C^\infty$ global Kuranishi chart $\cK$ is \emph{oriented} if we are provided with the data of orientations on $\cT$ and $\cE$ which are preserved by the $G$-action and an orientation on $\fg = \text{Lie}(G)$. 
    
    We say that $\cK$ is \emph{stably complex} if we are given the data of a $G$-invariant almost complex structure on $\cB$ and a $G$-invariant stably complex lift of the virtual vector bundle $T_{\cT/\cB} - (\cE\oplus\underline{\fg})$. Here, $T_{\cT/\cB}$ denotes the vertical tangent bundle of $\cT/\cB$ and $\underline{\fg}$ is the trivial bundle with fibre $\fg = \text{Lie}(G)$. 
\end{definition}

\begin{definition}[Rel--$C^\infty$ equivalence]\label{equivalence-of-charts-defined}
    Let $\cK = (G,\cT/\cB,\cE,\fs)$ be a rel--$C^\infty$ global Kuranishi chart as in Definition \ref{rel-smooth-global-chart-abstract-def}. Consider the following moves applied to $\cK$.
    \begin{enumerate}[(i)]
        \item\label{germ-move} (Germ equivalence) Given a $G$-invariant open neighborhood $U\subset\cT$ of $\fs^{-1}(0)$, replace $\cK$ by $(G,U/\cB,\cE|_U,\fs|_U)$.
        \item\label{stabilization-move} (Stabilization) Given a rel--$C^\infty$ vector bundle $p\cl W\to\cT/\cB$ carrying a compatible rel--$C^\infty$ action of $G$, replace $\cK$ by $(G,W/\cB,p^*\cE\oplus p^*W,p^*\fs\oplus\Delta_W)$. Here, $\Delta_W$ denotes the tautological diagonal section of $p^*W\to W$.
        \item\label{group-move} (Group enlargement) Given a compact Lie group $G'$ and a rel--$C^\infty$ principal $G'$-bundle $q\cl P\to\cT/\cB$ carrying a compatible rel--$C^\infty$ action of $G$, replace $\cK$ by $(G\times G',P/\cB,q^*\cE,q^*\fs)$.
        \item\label{base-move} (Base modification) Given a smooth manifold $\cB'$ (equipped with a smooth submersion to $\cB$) and a rel--$C^\infty$ submersion $\cT/\cB\to\cB'/\cB$ (covering the identity map of $\cB$), replace $\cK$ by $(G,\cT/\cB',\cE,\fs)$.
    \end{enumerate}
    We say that two rel--$C^\infty$ global Kuranishi charts $\cK,\cK'$ are \emph{rel--$C^\infty$ equivalent} if there exists a finite sequence of rel--$C^\infty$ global Kuranishi charts $\cK = \cK_0,\ldots,\cK_N = \cK'$ such that for each $0\le i<N$, the chart $\cK_i$ is obtained from $\cK_{i+1}$ (or $\cK_{i+1}$ is obtained from $\cK_i$) by applying one of the moves \eqref{germ-move}--\eqref{base-move} above. There is an obvious refinement of this notion of equivalence when $\cK$, $\cK'$ are stably complex (resp. oriented) by allowing only $G$-equivariant stably complex (resp. oriented) $W$ in (Stabilization) and pseudo-holomorphic (resp. oriented) submersions $\cB'\to\cB$ in (Base modification).
\end{definition}

\begin{remark}
    The move (Base modification) is not present in \cite{AMS21} since the definition of global Kuranishi charts therein does not make explicit reference to the base space of the thickening.
\end{remark}

\begin{definition}[Rel--$C^\infty$ cobordism]\label{cobordism-of-charts-defined}
    Let $\cK_0 = (G,\cT_0/\cB,\cE_0,\fs_0)$ and $\cK_1=(G,\cT_1/\cB,\cE_1,\fs_1)$ be rel--$C^\infty$ global Kuranishi charts having the same symmetry group $G$ and base space $\cB$. We say that $\cK_0$ and $\cK_1$ are \emph{rel--$C^\infty$ cobordant} if there exists $\cK_{01} = (G,\cT_{01}/\cB,\cE_{01},\fs_{01})$ with the following properties.
    \begin{enumerate}[(i)]
        \item $\cT_{01}\to\cB$ is a rel--$C^\infty$ manifold-with-boundary with $$\partial(\cT_{01}/\cB) = (\cT_0/\cB)\sqcup(\cT_1/\cB).$$
        \item $\cE_{01}\to\cT_{01}/\cB$ is a rel--$C^\infty$ vector bundle which restricts on the boundary to $\cE_0\sqcup\cE_1$.
        \item $\fs_{01}$ is a rel--$C^\infty$ section of $\cE_{01}$ with compact zero locus and it restricts on the boundary to $\fs_0\sqcup\fs_1$.
        \item There is a rel--$C^\infty$ $G$-action on $\cE_{01}\to\cT_{01}/\cB$ which makes $\fs_{01}$ a $G$-equivariant section and is compatible with the given actions on the boundary.
    \end{enumerate}
    There is an obvious refinement of this notion of cobordism when $\cK_0,\cK_1$ are stably complex (resp. oriented) by requiring the cobordism to carry compatible stable complex structures (resp. orientations).
\end{definition}

Consider an oriented rel--$C^\infty$ global Kuranishi chart $\cK = (G,\cT/\cB,\cE,\fs)$ for a compact Hausdorff space $Z$. An elementary argument (Lemma \ref{local-linear-action} and the paragraph preceding it) shows that the action of $G$ on $\cT$ is locally linear, i.e., for any $x\in \cT$, there are local coordinates on $\cT$ centred at $x$ in which the stabilizer $G_x$ acts linearly. From this, it follows that $\cT/G$ is a $\bQ$-homology manifold and carries a fundamental class $[\cT/G]$ in Borel--Moore homology with $\bQ$-coefficients. Recall that Borel--Moore homology with $\bQ$-coefficients is dual to compactly supported cohomology with $\bQ$-coefficients. 

Using local linearity of the $G$-action and a Mayer--Vietoris argument, it also follows that the natural map $(\cE\times EG)/G\to\cE/G$ induces an isomorphism on cohomology with $\bQ$-coefficients. Let $\tau_{\cE/G}$ be the image of the $G$-equivariant Thom class of $\cE$ under this isomorphism.

\begin{definition}[Virtual fundamental class] Let $\cK = (G,\cT/\cB,\cE,\fs)$ be an oriented rel--$C^\infty$ global Kuranishi chart for a compact Hausdorff space $Z$.
The \emph{virtual fundamental class} $\vfc{Z}_\cK\in \check{H}^d(Z,\bQ)^\vee$ of $Z$ with respect to $\cK$ is defined to be
\begin{equation}\label{vfc-de}
    \chml^{\vdim(Z)}(Z;\bQ) \xra{\fs^*\tau_{\cE/G}\cup(\cdot)} H_c^{\dim(\cT/G)}(\cT/G;\bQ)\xra{[\cT/G]} \bQ.
\end{equation}
The first map in \eqref{vfc-de} uses the continuity property 
$$\chml^{*}(Z;\bQ)\cong \fcolim H^*(U;\bQ)$$
of \v{C}ech cohomology, where the direct limit ranges over open neighborhoods $U$ of $\fs\inv(0)/G$ in $\cT/G$ while the second map in \eqref{vfc-de} uses the fact that Borel--Moore homology is dual to compactly supported cohomology.
\end{definition}

\begin{lemma}\label{vfc-invariant}
     Given an oriented global Kuranishi chart $\cK$ for a compact Hausdorff space $Z$, the virtual fundamental class $\vfc{Z}_\cK$ is unchanged if we replace $\cK$ by an equivalent oriented global Kuranishi chart.
\end{lemma}

\begin{proof} 

We consider each move in Definition \ref{equivalence-of-charts-defined} in turn. Invariance of the virtual fundamental class under (Base modification) is a tautology while invariance under (Germ equivalence) follows from excision. Invariance under (Group enlargement) follows from the fact that only the quotients $\cT/G$ and $\cE/G$ figure in the definition of the virtual fundamental class. 

Finally, fix a global Kuranishi chart $(G,\cT,\cE,\fs)$ for $Z$ and let $p \cl W\to \cT$ be an oriented rel--$C^\infty$ $G$-vector bundle. Let $\Delta_W$ be the tautological diagonal section of $p^*W\to W$ and set $\cE' := p^*\cE \oplus p^*W$ and $\fs':= p^*\fs \oplus \Delta_W$. Using the identity $(\fs')^*\tau_{\cE'/G} = p^*\fs^*\tau_{\cE/G}\cup\tau_{W/G}$ and the commutativity of the diagram
\begin{center}
\begin{tikzcd}
    H^{\dim(\cT/G)}_c(\cT/G;\bQ) \arrow[rr,"\tau_{W/G}\cup p^*(\cdot)"] \arrow[drr,"{[\cT/G]}"] & &   H^{\dim(W/G)}_c(W/G;\bQ) \arrow[d,"{[W/G]}"] \\
      & & \bQ
\end{tikzcd}
\end{center}
we get invariance under (Stabilization) as well.
\end{proof}

\begin{remark}\label{rem:vfc-alternate-description} 
The definition of the virtual fundamental class with $\bQ$-coefficients given in \cite[Equation (5.1)]{AMS21} coincides with \eqref{vfc-de} since the Poincar\'e duality isomorphism intertwines the cup product with the cap product, e.g., see \cite[Theorem V.10.1]{Bre12}.
\end{remark}

\section{Main construction}\label{overview}

Let $(X,\omega)$ be a closed symplectic manifold, $A\in H_2(X,\bZ)$ and let $g,n\ge 0$ be integers. Given any $J\in\cJ_\tau(X,\omega)$, we will construct a rel--$C^\infty$ global Kuranishi chart in the sense of Definition \ref{rel-smooth-global-chart-abstract-def} for the Gromov--Witten moduli space $\Mbar_{g,n}(X,A;J)$ using the choice of an auxiliary datum as specified in Definition \ref{aux-choices-defined}. The construction will be independent of this choice in a sense made precise in Theorem \ref{global-kuranishi-existence} below.

\subsection{Curves in projective space}We introduce certain moduli spaces of curves in projective spaces which will be important for our construction, some of which will play the role of $\cB$ in Definition \ref{rel-smooth-global-chart-abstract-def} in the global Kuranishi charts for $\Mbar_{g,n}(X,A;J)$. We also introduce \emph{framed stable maps}, which are analogous to the framed curves of \cite{AMS21}.

\begin{definition}[Algebraic base spaces]\label{base-space-de} 
    Given integers $N,m\ge 1$ and $\ell\ge 0$ we have the moduli stack $\Mbar_{g,\ell}(\bP^N,m)$ of stable holomorphic maps to $\bP^N$ of degree $m$ and genus $g$ with $\ell$ marked points. When $\ell = 0$, we typically omit it from the notation. Define
    \begin{align*}
        \Mbar_g^*(\bP^N,m)\subset\Mbar_g(\bP^N,m)
    \end{align*}
    to consist of the stable maps $\iota\cl C\to\bP^N$ such that
    \begin{enumerate}[(i)]
        \item the map $\iota: C\hookrightarrow\bP^N$ is an embedding and
        \item the restriction $H^0(\bP^N,\cO_{\bP^N}(1))\to H^0(C,\cO_C(1))$ is an isomorphism and $H^1(C,\cO_C(1)) = 0$, where we are identifying $C$ with its image under $\iota$.\footnote{Observe that we need $N = m-g$ for this condition to be non-vacuous.}
    \end{enumerate}
    Denote the universal family on it  as follows.
    \begin{center}
    \begin{tikzcd}
        \cC_g \arrow[r] \arrow[d,"\pi"] & \bP^N \\
        \Mbar_g^*(\bP^N,m) &
    \end{tikzcd}    
    \end{center} 
    Define $\Mbar_{g,\ell}^*(\bP^N,m)\subset\Mbar_{g,\ell}(\bP^N,m)$ to be the inverse image of $\Mbar_g^*(\bP^N,m)$ under the map $\Mbar_{g,\ell}(\bP^N,m)\to\Mbar_g(\bP^N,m)$ given by first forgetting the marked points and then stabilizing the resulting map. Given an integer $p\ge 1$, define
    \begin{align*}
        \Mbar^{*,p}_{g,\ell}(\bP^N,m)\subset\Mbar^*_{g,\ell}(\bP^N,m)
    \end{align*}
    to consist of the stable maps $(C,x_1,\ldots,x_\ell,\iota\cl C\to\bP^N)$ such that 
    \begin{enumerate}[(i)]
        \item the map $\iota\cl C\hookrightarrow\bP^N$ is an embedding and
        \item the line bundles $(\omega_C(x_1 + \cdots + x_\ell))^{\otimes p}$ and $\iota^*\cO_{\bP^N}(1)$ have the same degree on each irreducible component of $C$, where $\omega_C$ denotes the dualizing line bundle of the curve $C$.\footnote{Observe that we need $m = p(2g-2+\ell)$ for this condition to be non-vacuous.}
    \end{enumerate}
    Note that these two conditions imply that $\omega_C(x_1 + \cdots + x_\ell)$ has positive degree on each component, i.e., $(C,x_1,\ldots,x_\ell)$ is a stable pointed curve.
    
    Finally, define $\Mbar_{g,[\ell]}^{*,p}(\bP^N,m)$ to be the quotient of $\Mbar_{g,\ell}^{*,p}(\bP^N,m)$ under the action of the permutation group $S_\ell$ on the marked points.
\end{definition}

The group $\PGL(N+1,\bC) = \text{Aut}(\bP^N)$ acts on all the moduli spaces introduced in Definition \ref{base-space-de}. The next two statements will show that these moduli spaces are second countable complex manifolds. 

\begin{lemma}\label{alg-base-spaces-are-smooth-qproj}
    For integers $N,m,p\ge 1$ and $\ell\ge 0$, the moduli spaces
    \begin{align*}
        \Mbar_{g,\ell}^*(\bP^N,m)\quad\text{and}\quad \normalfont\Mbar_{g,\ell}^{*,p}(\bP^N,m)    
    \end{align*}
    are smooth quasi-projective schemes over $\bC$ of complex dimension
     \begin{align*}
         (N-3)(1-g)+m(N+1)+\ell.
     \end{align*}
   In particular, these moduli spaces are second countable complex manifolds.
\end{lemma}
\begin{proof}
    Throughout this argument, `open' is used to mean `Zariski open'.\par
   We first assume $\ell = 0$. In this case, $\Mbar^*_g(\bP^N,m)$ can be identified with a subset of the Hilbert scheme $\text{Hilb}_{P}(\bP^N)$ parametrizing closed subschemes of $\bP^N$ with Hilbert polynomial $P(t) = mt-g+1$. By \cite[\href{https://stacks.math.columbia.edu/tag/0E6U}{Tag 0E6U}]{stacks-project}, there is a maximal open subscheme $U\subset\text{Hilb}_P(\bP^N)$ parametrizing prestable curves, which are of genus $g$ and degree $m$ due to the definition of $P$. The curves in $U$ on which the restriction of $\cO_{\bP^N}(1)$ has vanishing $H^1$ constitute an open subscheme $U'\subset U$ by semicontinuity theorem \cite[Theorem III.12.8]{Har77}. For curves in $U'$, the assignment $[C\subset\bP^N]\mapsto H^0(C,\cO_C(1))$ defines a vector bundle by the theorem on cohomology and base change \cite[Theorem III.12.11]{Har77}. Thus, the curves in $U'$ for which the restriction map $H^0(\bP^N,\cO_{\bP^N}(1))\to H^0(C,\cO_C(1))$ is an isomorphism is an open subscheme. Thus, we have realized
    \begin{align*}
        \Mbar_g^*(\bP^N,m)\subset\text{Hilb}_P(\bP^N)
    \end{align*}
    as an open subscheme. The Hilbert scheme (for any fixed Hilbert polynomial) is known to be a projective scheme over $\bC$ and, thus its open subschemes are quasi-projective. To show that $\Mbar_g^*(\bP^N,m)$ is smooth of the expected dimension at any given curve $C\subset\bP^N$, it is enough to argue that $H^1(C,T_{\bP^N}|_C) = 0$. This vanishing statement follows by restricting the Euler exact sequence
    \begin{align*}
        0\to\cO_{\bP^N}\to\cO_{\bP^N}(1)^{N+1}\to T_{\bP^N}\to 0
    \end{align*}
    to $C$ and using $H^1(C,\cO_C(1)) = 0$. 
    
    As $\Mbar_{g,\ell}^*(\bP^N,m)$ is the inverse image of $\Mbar_g^*(\bP^N,m)$ under the forgetful map
    \begin{align*}
        \Mbar_{g,\ell}(\bP^N,m)\to\Mbar_g(\bP^N,m),
    \end{align*}
    we conclude that it is a quasi-projective scheme over $\bC$ using $\ell$ repeated applications of \cite[Corollary 4.6]{behrend-manin}. For any $(\Sigma,x_1,\ldots,x_\ell,f\cl \Sigma\to\bP^N)$
    belonging to $\Mbar_{g,\ell}^*(\bP^N,m)$, its image $C\subset\bP^N$ in $\Mbar_g^*(\bP^N,m)$ is obtained by contracting a forest of spherical components in $\Sigma$ and therefore, 
    $$H^1(\Sigma,f^*\cO_{\bP^N}(1)) =  H^1(C,\cO_C(1)) = 0.$$ 
    As before, we now conclude that $\Mbar_{g,\ell}^*(\bP^N,m)$ is smooth of the expected dimension. 
    
    Next, observe that $\Mbar_{g,\ell}^{*,p}(\bP^N,m)\subset\Mbar_{g,\ell}^*(\bP^N,m)$
    consists of precisely those stable maps $(\Sigma,x_1,\ldots,x_\ell,f\cl \Sigma\to\bP^N)$ for which
    \begin{enumerate}[(i)]
        \item $f\cl \Sigma\to\bP^N$ does not map any irreducible component of $\Sigma$ to a point, i.e., $f^*\cO_{\bP^N}(1)$ is an ample line bundle on $\Sigma$ and
        \item the line bundle $(\omega_\Sigma(x_1 + \cdots + x_\ell))^{\otimes p}\otimes f^*\cO_{\bP^N}(-1)$ has degree zero on each irreducible component of $\Sigma$.
    \end{enumerate}
    Since ampleness is an open condition in proper families \cite[Corollary 9.6.4]{EGA4part3}, imposing condition (i) gives an open subscheme. By the next paragraph, imposing condition (ii) also gives an open subscheme and this concludes the proof. 
    
    Define the \emph{combinatorial type} of any $(\Sigma,x_1,\ldots,x_\ell,f)$ in $\Mbar_{g,\ell}^*(\bP^N,m)$ to consist of its dual graph with labels on each vertex encoding the genus of the corresponding irreducible component, the number of marked points on it and the degree of $f^*\cO_{\bP^N}(1)$ on it. The combinatorial type stratifies $\Mbar_{g,\ell}^*(\bP^N,m)$ into locally closed subsets. Note that whether or not (ii) is satisfied depends only on the combinatorial type. Moreover, if $(\Sigma,x_1,\ldots,x_\ell,f)$ satisfies (ii), then this is true for any deformation as well (which includes smoothing some or all of the nodes in $\Sigma$). Thus, (ii) is satisfied away from a closed union of strata.
\end{proof}

\begin{corollary}\label{alg-base-spaces-are-smooth-cor}
    For integers $N,m,p\ge 1$ and $\ell\ge 0$, 
    \begin{align*}
        \normalfont\Mbar_{g,[\ell]}^{*,p}(\bP^N,m)    
    \end{align*}
    is a second countable complex manifold of complex dimension
    \begin{align*}
        (N-3)(1-g)+m(N+1)+\ell.    
    \end{align*}
\end{corollary}
\begin{proof}
    Since $\Mbar_{g,[\ell]}^{*,p}(\bP^N,m)$ is the quotient of $\normalfont\Mbar_{g,\ell}^{*,p}(\bP^N,m)$ by a free $S_\ell$ action, this follows from Lemma \ref{alg-base-spaces-are-smooth-qproj}. Second countability of $\Mbar_{g,[\ell]}^{*,p}(\bP^N,m)$ follows from the quasi-projectivity of $\Mbar_{g,\ell}^{*,p}(\bP^N,m)$.
\end{proof}

\begin{definition}[Framed stable maps]\label{framed-stable-map} 
    Fix a polarization $\cO_X(1)\to X$ taming $J$ as in Definition \ref{polarization-defined} and write $d:=\langle[\Omega],A\rangle$. An \emph{$\Omega$-stable map} (of genus $g$ and class $A$) is a smooth map $u:C\to X$ satisfying the following conditions.
	\begin{enumerate}[(i)]
		\item $C$ is a prestable curve of genus $g$ and $u_*[C] = A$.
		\item For any irreducible (resp. unstable irreducible) component $C'\subset C$, we have 
  $$\int_{C'}u^*\Omega = \langle[\Omega],u_*[C']\rangle\ge 0\qquad \text{ 
 (resp. $>0$)}.$$
	\end{enumerate}
	We will refer to $\Omega$-stable maps of genus $g$ and class $A$ as just stable maps when $\Omega,g,A$ are clear from the context. An \emph{equivalence} of stable maps $u:C\to X$ and $v:\Sigma\to X$ is a biholomorphic map $\varphi:\Sigma\to C$ such that $u\circ\varphi = v$. Given any stable map $u:C\to X$, we define the holomorphic line bundle $\fL_u\to C$ to be
    \begin{align}\label{domain-polarization-defined}
	    \fL_u := \omega_C\otimes (u^*\cO_X(1))^{\otimes 3}
    \end{align}
    where $\omega_C$ denotes the dualizing line bundle of $C$ and the holomorphic structure on the line bundle $u^*\cO_X(1)$ is defined by $(u^*\nabla)^{0,1}$. 
    
    Fix an integer $p\ge 1$ and set $m:=p(2g-2+3d)$ and $N:=m-g$. A \emph{framed stable map} is a smooth map $(u,\iota): C\to X\times\bP^N$ satisfying the following conditions.
    \begin{enumerate}[(i)]
        \item $u: C\to X$ is a stable map.
        \item $\iota: C\hkra\bP^N$ belongs to $\Mbar_g^*(\bP^N,m)$ and the line bundles $\iota^*\cO_{\bP^N}(1)$ and $\fL_u^{\otimes p}$ have the same degree on each irreducible component of $C$.
    \end{enumerate}
    Since $\iota$ is an embedding, we will usually identify $C$ with its image under $\iota$ and denote the framed stable map by $(C\subset\bP^N,u:C\to X)$.
    An \emph{equivalence} of framed stable maps $(u,\iota): C\to X\times\bP^N$ and $(v,\kappa): \Sigma\to X\times\bP^N$ is a biholomorphic map $\varphi: \Sigma\to C$ such that $u\circ\varphi = v$ and $\iota\circ\varphi = \kappa$.
\end{definition}

While a stable map $u:C\to X$ may have non-trivial automorphisms (i.e., self-equivalences), a framed stable map $(u,\iota):C\to X\times\bP^N$ can never have any non-trivial automorphisms since $\iota$ is an embedding. 

\begin{remark}
    For any stable map $u:C\to X$, the line bundle $\fL_u$ defined in \eqref{domain-polarization-defined} is ample. Therefore, we can promote $u:C\to X$ to a framed stable map by choosing a complex basis $(s_0,\ldots,s_N)$ of $H^0(C,\fL_u^{\otimes p})$ for a large integer $p\ge 1$ and taking the associated projective embedding $[s_0:\cdots:s_N]:C\hookrightarrow\bP^N$. Refer to Lemma \ref{Stabilisation is Very Positive} for more details.
\end{remark}

\begin{remark}
    The forgetful map from the set of framed stable maps to $\Mbar_g^*(\bP^N,m)$ is equivariant with respect to the natural $\PGL(N+1,\bC)$-actions.
\end{remark}

\subsection{Reducing the symmetry group} Even though the space of framed pseudo-holomorphic maps carries a $\PGL(N+1)$-action, the perturbations we use to achieve transversality only admit a $\PU(N+1)$ symmetry. In this subsection, we explain how to continuously select a $\PU(N+1)$-orbit inside each $\PGL(N+1)$-orbit of framed stable maps.

\begin{definition}[Logarithm for Hermitian matrices]\label{positive-definite-matrices}
    Given an integer $N\ge 2$, define the Lie groups
    \begin{align*}
        \cG := \PGL(N+1,\bC),\quad G:=\text{PU}(N+1) = U(N+1)/S^1,
    \end{align*}
    and let $\cH$ be the space of $(N+1)\times(N+1)$ Hermitian positive definite matrices modulo the action of positive real scalars. Any equivalence class $[A]$ in $\cH$ then has a unique representative with determinant $1$ which we denote by $[A]_\nu$. 
    
    The group $\cG$ has a left action on $\cH$ as follows: an element $[T]\in\cG$ maps an element $[A]\in\cH$ to $[TAT^*]\in\cH$. The Lie algebra $\fs\fu(N+1)$ consisting of skew-Hermitian trace-free $(N+1)\times(N+1)$ matrices carries a natural adjoint $G$-action. With this, the exponential map
	\begin{align*}
	\fs\fu(N+1)&\to\cH\\
	\text{i}M&\mapsto[\exp M]
	\end{align*}
	is a $G$-equivariant diffeomorphism. We let $\text{i}\log:\cH\to\fs\fu(N+1)$ denote its inverse. We can identify this with a map $\cG/G\to\fs\fu(N+1)$, also denoted $\text{i}\log$, via the isomorphism $P:\cG/G\to\cH$ provided by Lemma \ref{polar-decomp}\eqref{metric-coset} below.
\end{definition}

\begin{lemma}[Polar decomposition for $\cG$]\label{polar-decomp}
    Consider the setup of Definition \ref{positive-definite-matrices}. We then have the following.
	\begin{enumerate}[\normalfont(i)]
		\item\label{trivial-principal-bundle} The multiplication map $\cH\times G\to\cG$ is a diffeomorphism. 
		\item\label{convex-combs} If $\Lambda$ is a finite set, then the linear combination map
		\begin{align*}
		c_\Lambda:(\bR_{\ge 0}^\Lambda\setminus\{0\})\times\cH^\Lambda&\to\cH \\
		(\{t_i\}_{i\in \Lambda}, \{[A_i]_{i\in \Lambda}\})&\mapsto[\textstyle\sum_{i\in \Lambda}t_i[A_i]_\nu]
		\end{align*}
		is $\cG$-equivariant, where $\bR^\Lambda_{\ge 0}\setminus\{0\}$ carries the trivial action.
		\item\label{metric-coset} The map $\cG\to\cH$ given by $T\mapsto TT^*$ descends to a $\cG$-equivariant diffeomorphism $P:\cG/G\to\cH$. The identity $G$-coset is mapped to the class $[I]\in\cH$ of the identity matrix under $P$.
	\end{enumerate}
\end{lemma}
\begin{proof}
	The first assertion is an immediate consequence of the polar decomposition in $\GL(N+1,\bC)$. Note that $\cG$-equivariance in the second and third assertions are true by definition. Using the first assertion, we can view the map $P$ as the squaring map $[A]\mapsto[A^2]$ on $\cH$, which shows that it is a diffeomorphism.
\end{proof}

We will need the next definition when we describe the one of the auxiliary data required for our construction of global Kuranishi charts.

\begin{definition}[Polyfold of stable maps]\label{polyfolds-defined}
    Fix a polarization $\cO_X(1)\to X$ taming $J$. We define $Z = Z_{A,g}(X)$ to be the polyfold of $\Omega$-stable genus $g$ maps\footnote{In \cite{HWZ17}, $Z$ is called the space of \emph{stable curves} to indicate that these are stable maps up to equivalence. We prefer not to use the term `stable curve' in this context in order to avoid confusion with the moduli space $\Mbar_{g,n}$.} to $X$ in class $A$ which are not necessarily $J$-holomorphic, as constructed in \cite[Theorem 3.37]{HWZ17}. Note that the topology on $\Mbar_g(X,A;J)$ induced by $Z$ is the same as the topology of Gromov convergence, as explained in \cite[Remark 4.10]{HWZ17}.
\end{definition}

\begin{definition}[Good covering]\label{good-covering}
    Fix a polarization $\cO_X(1)\to X$ taming $J$ and write $d:=\langle[\Omega],A\rangle$. Given an integer $p\ge 1$, set $m := p(2g-2+3d)$ and $N:= m-g$. We define a \emph{good covering} to be a collection
    \begin{align*}
        \cU=\{(U_i,D_i,\chi_i)\}_{i\in\Lambda}    
    \end{align*}
    of tuples indexed by a finite set $\Lambda$ such that we have the following properties.
	\begin{enumerate}[1.]
		\item\label{stabilizing-divisor} For each $i\in \Lambda$,  $U_i$ is an open subset of $Z$ (from Definition \ref{polyfolds-defined}) and $D_i\subset X$ is a codimension $2$ submanifold-with-boundary satisfying the following properties for any stable map $(C,u)\in U_i$.
		\begin{enumerate}[(i)]
	        \item\label{transverse-marking} The set $u^{-1}(D_i)$ consists of exactly $3d$ distinct non-nodal points of $C$. Moreover, for each irreducible component $C'\subset C$, we have 
            \begin{align*}
                \#(u^{-1}(D_i)\cap C') = 3\langle [\Omega],u_*[C']\rangle.
            \end{align*}
            In particular, $C$ equipped with $u^{-1}(D_i)$ as marked points is a stable curve.
                \item\label{avoid-boundary} We have $u(C)\cap\partial D_i=\varnothing$ and the map $u$ is transverse to $D_i$.
        \end{enumerate}
        As a result, for any framed stable map $(C\subset\bP^N,u:C\to X)$ with $(C,u)\in U_i$, there is a well-defined element
        \begin{align*}
            \text{st}_{D_i}(C\subset\bP^N,u)\in\Mbar^{\;*,p}_{g,[3d]}(\bP^N,m)
        \end{align*}
        given by adding $u^{-1}(D_i)$ to $C$ as a set of $3d$ unordered marked points.
		\item For each $i\in \Lambda$, $\chi_i:Z\to[0,1]$ is a nonzero sc-smooth function supported in $U_i$. 
        \item For every point $(C,u)$ in $\Mbar_g(X,A;J)$, there exists an index $i\in\Lambda$ such that $(C,u)\in U_i$ and $\chi_i(C,u)>0$.
	\end{enumerate}
    Given a good covering $\cU$ as above, let $V_\cU\subset Z$ be the open subset where the sc-smooth function $\sum_{i\in \Lambda}\chi_i$ is strictly positive. Assume that we are additionally given a smooth $\cG$-equivariant map
    \begin{align*}
        \lambda:\Mbar^{\;*,p}_{g,[3d]}(\bP^N,m)\to\cG/G,
    \end{align*}
    with $\cG$ and $G$ as in Definition \ref{positive-definite-matrices}. Then, for any framed stable map $(C\subset\bP^N,u)$ such that $(C,u)\in V_\cU$, we have a well-defined element
    \begin{align*}
        \lambda_\cU(C\subset\bP^N,u)\in \cG/G
    \end{align*}
    defined by the formula
    \begin{align}\label{lambda-defined}
        \lambda_\cU(C\subset\bP^N,u):=c_\Lambda(\{\chi_i(C,u)\}_{i\in \Lambda},\{\lambda(\text{st}_{D_i}(C\subset\bP^N,u))\}_{i\in \Lambda}),
    \end{align}
    where we are using the identification $\cG/G\simeq \cH$ and the notation $c_\Lambda$ from Lemma \ref{polar-decomp}. Note that the assignment $(C\subset\bP^N,u)\mapsto\lambda_\cU(C\subset\bP^N,u)$ is $\cG$-equivariant.
\end{definition}

\begin{remark}
    The map $\lambda_\cU$ described in Definition \ref{good-covering} allows us to distinguish a class of \emph{unitarily framed stable maps} from among all framed stable maps. Morally, this could have been done by choosing local slices for the $\cG$-action on the space of framed stable maps followed by a partition of unity argument. The good covering $\cU$ allows us to implement this strategy without having to appeal to slice theorems in infinite dimensions.
\end{remark}

\subsection{Description of the auxiliary data and the chart} We can now describe the auxiliary data necessary for the construction of a global Kuranishi chart for $\Mbar_{g,n}(X,A;J)$, as well as the construction itself.

\begin{definition}[Auxiliary data]\label{aux-choices-defined}
	An \emph{auxiliary datum} for the moduli space $\Mbar_{g,n}(X,A;J)$ is a tuple
    \begin{align*}
        (\nabla^X,\cO_X(1),p,\cU,\lambda,k)    
    \end{align*}
    with the following properties.
	\begin{enumerate}[(i)]
		\item $\nabla^X$ is a $\bC$-linear connection on the tangent bundle $T_X$, which is viewed as a complex vector bundle using $J$.
		\item\label{polarization-on-target} $\cO_X(1)\to X$ is a polarization taming $J$ as in Definition \ref{polarization-defined}. In particular, $\cO_X(1)$ has a Hermitian connection with curvature form $-2\pi\text{i}\Omega$. We set 
        \begin{align*}
            d := \langle[\Omega],A\rangle.         
        \end{align*}
		\item\label{framed-maps-etc} $p\ge 1$ is an integer. We set
        \begin{align*}
            m:=p(2g-2+3d),\quad &N:=m-g,\\
            \cG:=\PGL(N+1,\bC),\quad &G:=\PU(N+1).
        \end{align*}
		\item\label{aux-good-covering} $\cU$ is a good covering as in Definition \ref{good-covering} and
        \begin{align*}
            \lambda:\Mbar^{\;*,p}_{g,[3d]}(\bP^N,m)\to\cG/G    
        \end{align*}
        is a smooth $\cG$-equivariant map, 
        with $\lambda_\cU$ being the map defined by \eqref{lambda-defined}.
		\item $k\ge 1$ is an integer.
	\end{enumerate}
\end{definition}

\begin{construction}\label{high-level-description-defined}
	\normalfont{
		Having fixed an auxiliary datum $(\nabla^X,\cO_X(1),p,\cU,\lambda,k)$ and the notation as in Definition \ref{aux-choices-defined}, the associated global Kuranishi chart
        \begin{align*}
            \cK = (G,\cT/\Mbar^*_g(\bP^N,m),\cE,\fs)    
        \end{align*}
        for $\Mbar_g(X,A;J)$ consists of the following data.
		\begin{enumerate}[(i)]
			\item (Thickening) $\mathcal T$ consists of all tuples $(C\subset\bP^N,u:C\to X,\eta,\alpha)$ satisfying the following properties.
			\begin{enumerate}[(a)]
				\item\label{thickening-condition-1} $(C\subset\bP^N,u:C\to X)$ is a framed stable map lying in the domain of $\lambda_\cU$.
				\item\label{thickening-condition-2} The element $\eta$ belongs to the finite-dimensional $\bC$-vector space
				\begin{align}\label{obstruction-space-defined}
				E_{(C\subset\bP^N,u)}:= H^0(C,{T^*}^{0,1}_{\bP^N}|_C\otimes u^*T_X\otimes\cO_C(k))\otimes\overline{H^0(\bP^N,\cO_{\bP^N}(k))}.
				\end{align}
				Here, we use the complex linear identification
                \begin{align}\label{antilinear-dual-metric-identification}
                    {T^*}^{0,1}_{\bP^N}\simeq T_{\bP^N},    
                \end{align}
                given by the Fubini--Study metric, to endow the ${T^*}^{0,1}_{\bP^N}$ with a holomorphic structure. Similarly, we endow $u^*T_X$ with the holomorphic structure given by $(u^*\nabla^X)^{0,1}$. On the normalization $\tilde C\to C$, the equation
				\begin{align*}
				\delbar_J\tilde u + \langle\eta\rangle\circ d\iota_{\tilde C} = 0\in\Omega^{0,1}(\tilde C,\tilde u^*T_X)
				\end{align*}
				is satisfied. Here, $\tilde u$ and $\iota_{\tilde C}$ denote the pullbacks to $\tilde C$ of the map $u$ and the inclusion $C\subset\bP^N$ respectively. The $\bC$-linear contraction operator
				\begin{align}\label{inner-product-pairing-defined}
				\langle\cdot\rangle: E_{(C\subset\bP^N,u)}\to\Omega^0(C,{T^*}^{0,1}_{\bP^N}|_C\otimes u^*T_X)
				\end{align}
				is induced by the standard Hermitian metric on the line bundle $\cO_{\bP^N}(k)$.
				
				\item\label{thickening-condition-3} The element $\alpha\in H^1(C,\cO_C)$ is such that we have the identity
				\begin{align}\label{fix-polarization-alpha}
				[\cO_C(1)]\otimes[\fL_u^{\otimes p}]^{-1} = \exp\alpha\in \text{Pic}(C)
				\end{align}
				in the Picard group of $C$, i.e., the group of isomorphism classes of holomorphic line bundles with group operation given by tensor product of line bundles. Here, $\exp:H^1(C,\cO_C)\to\text{Pic}(C)$ is the canonical exponential map, recalled in Appendix \ref{line-bundles-on-families-of-curves}.
			\end{enumerate}
			The natural $U(N+1)$-action on $\cT$ descends to a $G$-action with respect to which the natural forgetful morphism $\pi:\cT\to\Mbar_g^*(\bP^N,m)$ is equivariant.
			\item (Obstruction bundle) $\cE\to\cT$ is the family of vector spaces over $\cT$ whose fibre over a given point $(C\subset\bP^N,u,\eta,\alpha)\in\cT$ is given by
			\begin{align}\label{obstruction-bundle-fibre-def}
			\fs\fu(N+1)\oplus E_{(C\subset\bP^N,u)} \oplus H^1(C,\cO_C).
			\end{align}
			It carries a natural fibrewise linear action of $G$ that lifts the $G$-action on $\cT$.
			\item (Obstruction section) The section $\fs:\cT\to\cE$ is defined by the formula
			\begin{align*}
			\fs(C\subset\bP^N,u,\eta,\alpha) = (\text{i}\log\lambda_\cU(C\subset\bP^N,u),\eta,\alpha).
			\end{align*}
			Here, $\text{i}\log:\cG/G\to\fs\fu(N+1)$ is the `polar decomposition' map from Definition \ref{positive-definite-matrices}.
		\end{enumerate}
        The section $\fs$ is $G$-equivariant and there is a natural forgetful map 
        \begin{align*}
            \fs^{-1}(0)/G\to\Mbar_g(X,A;J).    
        \end{align*}
		The associated global Kuranishi chart 
        \begin{align}\label{marked-point-chart-defined}
            \cK_n = (G,\cT_n/\Mbar^*_{g,n}(\bP^N,m),\cE_n,\fs_n)    
        \end{align}
        for $\Mbar_{g,n}(X,A;J)$ is defined by pulling back $\cT,\cE$ and $\fs$ along the natural forgetful map $$\Mbar^*_{g,n}(\bP^N,m)\to\Mbar^*_g(\bP^N,m).$$
	}
\end{construction}

\begin{remark}[Exceptional case]\label{A=0-exception}
    In the special case when $\lspan{[\omega],A} = 0$, $J$-holomorphic stable maps are just stable curves mapping to a point in $X$ and thus, there are no stable maps when $g = 0, n\le 2$ or $g=1,n=0$. 
    
    When $g\le 1$ and $\lspan{[\omega],A} = 0$, we must modify Construction \ref{high-level-description-defined} by twisting $\fL_u$ from \eqref{domain-polarization-defined} by the marked points, i.e., we associate the holomorphic line bundle $\omega_C(p_1 + \cdots + p_n)\otimes (u^*\cO_X(1))^{\otimes 3}$ to the stable map $(C,p_1,\ldots,p_n,u:C\to X)$. We leave this modification to the reader and ignore this special case for the remainder.
\end{remark}

Note that Construction \ref{high-level-description-defined} defines only the points of the spaces $\cT$ and $\cE$ and does not describe any additional structure on these (e.g., a rel--$C^\infty$ structure, a vector bundle structure). In order to ensure that $\cT$ as defined above is a manifold and $\cE$ is a vector bundle on it, at least near $\fs^{-1}(0)$, we will need to restrict to a special subclass of auxiliary data.

\begin{definition}[Unobstructed auxiliary datum]\label{unobstructed-aux}
	Given an auxiliary datum $(\nabla^X,\cO_X(1),p,\cU,\lambda,k)$ as in Definition \ref{aux-choices-defined}, say that it is \emph{unobstructed} if the following properties hold for any stable $J$-holomorphic map $u:C\to X$ in $\Mbar_g(X,A;J)$.
	\begin{enumerate}[(a)]
		\item\label{very-ample-acyclic-line-bundle} The line bundle $\fL_u^{\otimes p}\to C$ is very ample and we have $H^1(C,\fL_u^{\otimes p}) = 0$. Here, $\fL_u$ is the holomorphic line bundle on $C$ introduced in Definition \ref{framed-stable-map}.
		\item\label{killing-obstructions} For every complex basis $\cF = (s_0,\ldots,s_N)$ of $H^0(C,\fL_u^{\otimes p})$, we get a corresponding projective embedding $\iota_{C,\cF} = [s_0:\cdots:s_N]:C\hookrightarrow\bP^N$ and this yields a framed stable map $(C\subset\bP^N,u)$ as in Definition \ref{framed-stable-map}. If $\lambda_\cU(C\subset\bP^N,u)\in\cG/G$ is the identity coset, then we have
        \begin{align*}
            H^1(C,{T^*}^{0,1}_{\bP^N}|_C\otimes u^*T_X\otimes\cO_C(k)) = 0,
        \end{align*}
        and the map
        \begin{align*}
			D(\delbar_J)_u\oplus(\langle\cdot\rangle\circ d\iota_{\tilde C,\cF}):\Omega^0(C,u^*T_X)\oplus E_{(C\subset\bP^N,u)}\to\Omega^{0,1}(\tilde C,\tilde u^*T_X)
		\end{align*}
			is surjective. Here, $D(\delbar_J)_u$ is the linearization of the non-linear Cauchy--Riemann operator $\delbar_J$ at $u$ and the map $\langle\cdot\rangle$ is as in \eqref{inner-product-pairing-defined}.
	\end{enumerate}
\end{definition}

\begin{discussion}[Groupoid equivalence]\label{groupoid-equivalence}
With reference to Construction \ref{high-level-description-defined}, let us explain why we have an identification of $\fs^{-1}(0)/G$ and $\Mbar_g(X,A;J)$ as sets and, in fact, even as groupoids when the auxiliary datum is unobstructed.\footnote{In fact, for this, we only need condition \eqref{very-ample-acyclic-line-bundle} of Definition \ref{unobstructed-aux}.}

On the one hand, define the groupoid $[\fs^{-1}(0)/G]$ as follows. Its objects are the points of $\fs^{-1}(0)$, i.e., a framed stable map $(C\subset\bP^N,u:C\to X)$ such that
\begin{enumerate}[\normalfont(i)]
    \item $u$ is $J$-holomorphic,
    \item the holomorphic line bundles $\cO_C(1)$ and $\fL_u^{\otimes p}$ are isomorphic and
    \item $\lambda_\cU(C\subset\bP^N,u)\in\cG/G$ is the identity coset.
\end{enumerate}
Its morphisms are described as follows. For every $\gamma\in G$ and $x\in\fs^{-1}(0)$, there is a morphism $\gamma:x\to\gamma\cdot x$ in $[\fs^{-1}(0)/G]$. For $x = (C\subset\bP^N,u)$, note that $\gamma\cdot x = (\gamma(C)\subset\bP^N,u\circ \gamma_C^{-1})$, where $\gamma_C:C\to\gamma(C)$ is the restriction of $\gamma:\bP^N\to\bP^N$. 

On the other hand, consider the groupoid of stable $J$-holomorphic maps which we denote, by abuse of notation, as $\Mbar_g(X,A;J)$. Its objects are the stable $J$-holomorphic maps $u:C\to X$ and its morphisms are equivalences of stable maps. There is an obvious forgetful functor 
\begin{align*}
    \ff: [\fs^{-1}(0)/G]\to\Mbar_{g,n}(X,A;J)    
\end{align*}
defined as follows. On objects, $\ff(C\subset\bP^N,u)=(C,u)$, i.e., $\ff$ forgets the embedding of $C$ into $\bP^N$. For $\gamma\in G$, the morphism $\gamma:(C\subset\bP^N,u)\to(\gamma(C)\subset\bP^N,u\circ\gamma_C^{-1})$ is mapped by $\ff$ to the morphism $\gamma_C:(C,u)\to(\gamma(C),u\circ\gamma_C^{-1})$. We claim that $\ff$ is an equivalence of groupoids, i.e., it is fully faithful and essentially surjective.
    
To see that $\ff$ is essentially surjective, let $u:C\to X$ in $\Mbar_g(X,A;J)$ be given. Then, $\fL_u^{\otimes p}$ is very ample and has vanishing $H^1$. Thus, by Riemann--Roch, we have $\dim_\bC H^0(C,\fL_u^{\otimes p}) = N+1$. By taking a basis of global sections, we get a projective embedding $\iota: C\hookrightarrow\bP^N$ with $\cO_C(1)\simeq\fL_u^{\otimes p}$. It is possible that $\text{i}\log\lambda_\cU$ does not vanish on the resulting framed stable map $(\iota(C)\subset\bP^N,u\circ\iota^{-1})$, but this can be corrected by the action of a suitable $\gamma\in \cG$ since the action of $\cG$ on $\cG/G$ is transitive. This amounts to changing the basis of global sections. Thus, we have an object of $[\fs^{-1}(0)/G]$ whose image under $\ff$ is equivalent to $(C,u)$.

To see that $\ff$ is fully faithful, let $(C\subset\bP^N,u)$ and $(C'\subset\bP^N,u')$ be two given objects of $[\fs^{-1}(0)/G]$. Given two morphisms $\gamma_1,\gamma_2$ between them such that $\ff(\gamma_1) = \ff(\gamma_2)$, we must show that $\gamma_1 = \gamma_2$. But this is immediate from the fact that $\gamma_1,\gamma_2$ are linear automorphisms of $\bP^N$ and $C\subset\bP^N$ is not contained in any hyperplane. Conversely, given an equivalence $\varphi:(C,u)\to(C',u')$, we must show that it comes from some $\gamma:(C\subset\bP^N,u)\to(C'\subset\bP^N,u')$. Since $u'\circ\varphi = u$, we obtain a canonical isomorphism of holomorphic line bundles
\begin{align*}
    \fL_u\simeq\varphi^*\fL_{u'},    
\end{align*}
which induces an isomorphism $H^0(C,\fL_u^{\otimes p})\simeq H^0(C',\fL_{u'}^{\otimes p})$. Since the embeddings $C\subset\bP^N$ (resp. $C'\subset\bP^N$) are defined by choosing a basis of the global sections of $\fL_u^{\otimes p}$ (resp. $\fL_{u'}^{\otimes p}$), this gives a linear automorphism $\gamma\in\cG$ of $\bP^N$ such that $\gamma(C) = C'$ and $\gamma|_C=\varphi$. Since the $\cG$-equivariant map $\lambda_\cU$ sends both $(C\subset\bP^N,u)$ and $\gamma\cdot(C\subset\bP^N,u) = (C'\subset\bP^N,u')$ to the identity coset in $\cG/G$, we conclude that $\gamma\in G$. Thus, the morphism $\varphi$ is in the image of $\ff$.

\end{discussion}

\subsection{Statement of the main result} We can now formulate our main result on global Kuranishi charts for Gromov--Witten moduli spaces.

\begin{theorem}[Global Kuranishi charts for GW moduli spaces]\label{global-kuranishi-existence}
    Fix a closed symplectic manifold $(X,\omega)$, a class $A\in H_2(X,\bZ)$ and integers $g,n\ge 0$.
	\begin{enumerate}[\normalfont(1)]
		\item\label{achieveing-transversality} Fix $J\in\cJ_\tau(X,\omega)$. Unobstructed auxiliary data, in the sense of Definition \ref{unobstructed-aux}, exist. Moreover, any choices of connection $\nabla^X$ and polarization $\cO_X(1)$ taming $J$ extend to an unobstructed auxiliary datum $(\nabla^X,\cO_X(1),p,\cU,\lambda,k)$. 
		\item\label{transversality-implies-smoothness} Fix $J\in\cJ_\tau(X,\omega)$ and an unobstructed auxiliary datum $(\nabla^X,\cO_X(1),p,\cU,\lambda,k)$. Then, the result $\cK_n = (G,\cT_n,\cE_n,\fs_n)$ of Construction \ref{high-level-description-defined} is a stably complex rel--$C^\infty$ global Kuranishi chart over $\Mbar_{g,n}^*(\bP^N,m)$, when restricted to an open $G$-invariant neighbourhood $\cT_n^{\normalfont\text{reg}}\subset\cT_n$ of $\fs_n\inv(0)$.
		\item\label{global-kuranishi-uniqueness} The global Kuranishi charts obtained from Construction \ref{high-level-description-defined} have the following uniqueness properties.
		\begin{enumerate}[\normalfont(a)]
		    \item\label{chart-uniqueness-up-to-equivalence} Fix $J\in\cJ_\tau(X,\omega)$. Then, the global Kuranishi charts for $\Mbar_{g,n}(X,A;J)$ associated to any two unobstructed auxiliary data are stably complex rel--$C^\infty$ equivalent in the sense of Definition \ref{equivalence-of-charts-defined}.
		    \item\label{chart-uniqueness-up-to-cobordism} Given $J_0,J_1\in\cJ_\tau(X,\omega)$, there exist auxiliary data $(\nabla^{X,i},\cO_X(1),p,\cU,\lambda,k)$ for $i=0,1$ which are unobstructed and the associated global Kuranishi charts for $\Mbar_{g,n}(X,A;J_0)$ and $\Mbar_{g,n}(X,A;J_1)$ are stably complex rel--$C^\infty$ cobordant in the sense of Definition \ref{cobordism-of-charts-defined}.
		\end{enumerate}
	\end{enumerate}
\end{theorem}

We will prove Theorem \ref{global-kuranishi-existence}\eqref{achieveing-transversality}, \eqref{transversality-implies-smoothness} and  \eqref{global-kuranishi-uniqueness} in \textsection\ref{achieving-transversality-proof}, \textsection\ref{transversality-implies-smoothness-proof} and \textsection\ref{proof-of-equivalence-and-cobordism} respectively.

\begin{remark}[Virtual fundamental classes] Lemma \ref{vfc-invariant} and Theorem \ref{global-kuranishi-existence} yield a well-defined virtual fundamental class for $\Mbar_{g,n}(X,A;J)$. By Theorem \ref{global-kuranishi-existence}\eqref{chart-uniqueness-up-to-cobordism}, its pushforward to $X^n\times \Mbar_{g,n}$ does not depend on the choice of $J$. 
\end{remark}

\begin{remark}\label{AMS-comparison}
    We compare our construction for $g = 0$ with Abouzaid--McLean--Smith's construction, \cite[\textsection 6]{AMS21}, of a global Kuranishi chart for genus $0$ stable maps.
   For an overview of the construction from \cite{AMS21}, see \textsection\ref{sec:prelims} above. 
    
   The most immediate difference is the choice of auxiliary data that goes into both constructions. A more technical difference is that we endow our moduli spaces with canonical rel--$C^\infty$ structures and for this, we use some inputs from polyfold theory. In contrast, the construction of \cite{AMS21} uses classical gluing analysis to give the thickened moduli space a \emph{$\normalfont C^1_\text{loc}$ structure} (this is similar to a rel--$C^1$ structure in the sense of \cite{Swa21}). This turns out to be enough to get a well-defined vertical tangent bundle, which can then be used to endow the thickened moduli space with a smooth structure (via equivariant smoothing theory).

    On the other hand, one can use the techniques of this paper to endow the global Kuranishi chart constructed in \cite{AMS21} with a natural rel--$C^\infty$ structure. Then, a slight variant of the argument which proves Theorem \ref{global-kuranishi-existence}\eqref{chart-uniqueness-up-to-equivalence} shows that it is rel--$C^\infty$ equivalent to our $g = 0$ construction.
\end{remark}
\section{Transversality for auxiliary data}\label{achieving-transversality-proof}

In this section, we will prove Theorem \ref{global-kuranishi-existence}\eqref{achieveing-transversality} by showing how to choose the parameters $p,\cU,\lambda,k$ as in Definition \ref{aux-choices-defined} so that the auxiliary datum 
\begin{align*}
    (\nabla^X,\cO_X(1),p,\cU,\lambda,k)    
\end{align*}
is unobstructed in the sense of Definition \ref{unobstructed-aux}. We may assume that we are already given choices of $\nabla^X$ and $\cO_X(1)$. Indeed, it is obvious that $J$-linear connections on $T_X$ exist, while the existence of polarizations on $X$ taming $J$ is guaranteed by Lemma \ref{Suitable Line Bundle on Target}. We set
\begin{align*}
    d := \langle [\Omega], A\rangle
\end{align*}
as in Definition \ref{aux-choices-defined}.

\subsection{Choosing the integer $p$}\label{p-choice} Recall from  \eqref{domain-polarization-defined} that we defined the holomorphic line bundle 
$$\fL_u := \omega_C \otimes (u^*\cO_X(1))^{\otimes 3}$$
for any stable map $u \cl C\to X$.

\begin{lemma}[Existence of $p$]\label{Stabilisation is Very Positive}
	There exists a positive integer $p$ depending only on $g,d$ with the property that for any stable map  $(C,u)$ and any integer $q\ge p$, the line bundle $\fL_u^{\otimes q}\to C$ is very ample and we have $H^1(C,\fL_u^{\otimes q}) = 0$.
\end{lemma}

\begin{proof}
	$\fL_u$ has total degree $2g-2+3d$ on $C$ and, by the stability of $(C,u)$, it has degree $\ge 1$ on each irreducible component of $C$. It follows that $C$ has $\le 2g-2+3d$ irreducible components and thus, the number of possibilities for the dual graph of $C$, decorated by genus labels on the vertices, can be bounded in terms of $g,d$. Hence it suffices to find a $p$ for each possible decorated dual graph $\Gamma$ of $C$.
	
	For each vertex $v$ of the dual graph $\Gamma$ of $C$, let $C_v$ be the component of the normalization of $C$ corresponding to $v$. Let $g_v$ be the genus of $C_v$, $D_v\subset C_v$ be the subset consisting of the inverse images of the nodal points and $L_v$ be the pullback of $\fL_u$ to $C_v$. For any two points $a,b\in C_v$, Serre duality \cite[Theorem III.7.6]{Har77} yields
	\begin{align*}
	H^1(C_v,L_v^{\otimes q}(-D_v-a-b)) = H^0(C_v,\omega_{C_v}(D_v+a+b)\otimes {L_v^*}^{\otimes q})^* = 0
	\end{align*}
	for $q\ge p_{\Gamma} := 1 + \max_{v\in\Gamma}(2g_v + |D_v|)$ once we recall that $\deg_{C_v}L_v\ge 1$. This cohomology vanishing statement (for each $v\in\Gamma$) implies that $\fL_u^{\otimes q}$ is very ample and has trivial $H^1$ on $C$. Since the lower bound $p_\Gamma$ on $q$ depends only on the decorated dual graph $\Gamma$, the proof is complete.
\end{proof}

We fix $p\ge 1$ to be the smallest integer which satisfies the conclusion of Lemma \ref{Stabilisation is Very Positive} above. Having fixed the choice of $p$, we set
\begin{align*}
    m:=p(2g-2+3d),\quad &N:=m-g,\\
    \cG:=\PGL(N+1,\bC),\quad &G:=\PU(N+1).
\end{align*}
exactly as in Definition \ref{aux-choices-defined}\eqref{framed-maps-etc}. Lemma \ref{Stabilisation is Very Positive} shows that our choice of $p$ is such that condition \eqref{very-ample-acyclic-line-bundle} of Definition \ref{unobstructed-aux} is satisfied. The following observation will be useful later.

\begin{lemma}[Unobstructed projective embedding]\label{Regular Embedding via Framing}
Let $(C,u)$ be a stable map as in Definition \ref{framed-stable-map}. Then, any complex linear basis $\cF = (s_0,\ldots,s_N)$ of $H^0(C,\fL_u^{\otimes p})$ determines a point of $\Mbar_{g}^*(\bP^N,m)$ via the holomorphic projective embedding 
\begin{align*}
    \iota_{C,\cF} = [s_0:\cdots:s_N] : C\to \bP^N.
\end{align*}
\end{lemma}

\begin{proof}
    This is immediate from the isomorphism $\iota_{C,\cF}^*\cO_{\bP^N}(1) \simeq \fL_u^{\otimes p}$ of holomorphic line bundles. See also the essential surjectivity argument in Discussion \ref{groupoid-equivalence}.
\end{proof}

\subsection{Choosing the good covering $\cU$ and the map $\lambda$}\label{lambda-choice}

\begin{lemma}[Existence of $\cU$]\label{good-coverings-exist}
    Good coverings exist.
\end{lemma}
\begin{proof}
    Let $(C,u)\in\Mbar_g(X,A;J)$ be given. For any irreducible component $C'\subset C$ on which $u$ is non-constant, consider the $J$-holomorphic map $\tilde u_{C'}:\tilde C'\to X$ induced by $u$ on the normalization $\tilde C'$ of $C'$. Use \cite[Proposition 2.5.1]{MS12} to factor $\tilde u_{C'}$
    \begin{align*}
        \tilde C'\xrightarrow{\varphi}\Sigma\xrightarrow{f} X,
    \end{align*}
    where $\varphi$ is a holomorphic (possibly branched) cover of compact Riemann surfaces and $f$ is a somewhere injective $J$-holomorphic map. Let $\{f_\nu:\Sigma_\nu\to X\}_\nu$ be the finite list of somewhere injective $J$-holomorphic maps obtained in this way from $u$. For indices $\nu\ne\nu'$, the set $f_\nu^{-1}(f_{\nu'}(\Sigma_{\nu'}))\subset \Sigma_\nu$ is at most countable and can accumulate only at critical points of $f_\nu$ using \cite[Corollary 2.5.4 and Proposition 2.4.4]{MS12}.

    Fix a particular $\nu$ and let $d_\nu := \langle[\Omega],(f_\nu)_*[\Sigma_\nu]\rangle$. We can then choose $3d_\nu$ distinct points on $\Sigma_\nu$ such that these are injective points of $f_\nu$, they do not lie on $f_\nu^{-1}(f_{\nu'}(\Sigma_{\nu'}))$ and all their inverse images in the normalization $\tilde C$ of $C$ correspond to non-nodal points of $C$ at which $du\ne 0$. Choose a codimension $2$ submanifold-with-boundary $D_\nu\subset X$ such that $f_\nu$ is transverse to $D_\nu$, the set $f_\nu^{-1}(\partial D_\nu)$ is empty, the set $f_\nu^{-1}(D_\nu)$ consists precisely of the chosen $3d_\nu$ points and finally, $f_{\nu'}(\Sigma_{\nu'})$ is disjoint from $D_\nu$ for $\nu\ne\nu'$. Repeating this for all $\nu$ and ensuring that the $D_\nu$ are pairwise disjoint, define $D\subset X$ to be their disjoint union. By construction, we have the following.
    \begin{enumerate}[(i)]
        \item The set $u^{-1}(D)$ consists of exactly $3d$ distinct non-nodal points of $C$ and, for each irreducible component $C'\subset C$, we have
        \begin{align*}
            \#(u^{-1}(D)\cap C') = 3\langle[\Omega],u_*[C']\rangle.
        \end{align*}
        \item We have $u(C)\cap\partial D = \varnothing$ and the map $u$ is transverse to $D$.
    \end{enumerate}
    These are exactly conditions \eqref{transverse-marking} and \eqref{avoid-boundary} from Definition \ref{good-covering} for $(C,u)$. Since the polyfold $Z$ carries the $H^{3,\delta_0} = W^{3,2,\delta_0}$ topology (for a small $\delta_0>0$) and we have the Sobolev embedding $H^3 = W^{3,2}\subset C^1$ in dimension two, it follows that there is an open neighborhood $U\subset Z$ of $(C,u)$ such that conditions (i)--(ii) above hold for all stable maps $(C',u')\in U$.

    Next, note that the polyfold $Z$ admits sc-smooth cutoff functions since it is locally modeled on (retracts of) sc-Hilbert spaces. (To see this, combine Proposition 5.5 and Theorem 12.6 in \cite{HWZ21} with the example of sc-Hilbert spaces discussed in the paragraph following Definition 5.13 therein.) Thus, we may choose an sc-smooth function $\chi:Z\to [0,1]$ supported in $U$ such that $\chi(C,u)>0$.

    The above argument constructs a triple $(U,D,\chi)$ for any given stable map $(C,u)$. Compactness of $\Mbar_g(X,A;J)$ now implies that good coverings exist.
\end{proof}

Using Lemma \ref{good-coverings-exist}, we fix a good covering $\cU = \{(U_i,D_i,\chi_i)\}_{i\in\Lambda}$.

\begin{lemma}\label{Proper Action on Domain Stable Embedded Curves} For any $N,m,p\ge 1$ and $\ell\ge 0$, the $\cG$-action on 
\begin{align*}
    \Mbar^{\;*,p}_{g,[\ell]}(\bP^N,m)    
\end{align*}
is proper.
\end{lemma}
\begin{proof}
    Since $\Mbar^{\;*,p}_{g,\ell}(\bP^N,m)\to \Mbar^{\;*,p}_{g,[\ell]}(\bP^N,m)$ is a finite unbranched covering map, it suffices to show that the $\cG$-action on the former is proper. Thus, we must argue that the map
    \begin{align*}
        \cG\times\Mbar_{g,\ell}^{*,p}(\bP^N,m)&\to \Mbar_{g,\ell}^{*,p}(\bP^N,m)\times\Mbar_{g,\ell}^{*,p}(\bP^N,m) \\
        (g,x)&\mapsto(x,g\cdot x)
    \end{align*}
    is proper, i.e., the inverse image of any compact set is compact. Since this is a morphism of quasi-projective schemes, it will be enough to show that it is a proper morphism in the sense of algebraic geometry.\footnote{A proper morphism of quasi-projective schemes is a projective morphism by \cite[\href{https://stacks.math.columbia.edu/tag/0BCL}{Tag 0BCL}]{stacks-project}. Projective morphisms induce proper maps on the $\bC$-points, equipped with the usual topology, since Zariski closed sets are closed in the usual topology and the $\bC$-points of $\bP^r$ form a compact Hausdorff space in the usual topology.}
    
    We may use the Noetherian valuative criterion \cite[\href{https://stacks.math.columbia.edu/tag/0208}{Tag 0208}]{stacks-project} to test properness of the above morphism of schemes. To this end, let $R$ be a discrete valuation ring and $K$ its fraction field. Given any three morphisms 
    \begin{align*}
	\alpha,\alpha'&:\Spec R\to \Mbar^{\;*,p}_{g,\ell}(\bP^N,m)\\
	\gamma&:\Spec K\to\cG
	\end{align*}
	such that $\gamma\cdot\alpha_K = \alpha'_K$, we need to extend $\gamma$ to a morphism $\Spec R\to\cG$. Lift $\gamma$ to an element $\delta\in \GL(N+1,K)$. This lift is unique up to multiplication by an element of $K^\times$. We will lift $\delta$ to $GL(N+1,R)$ up to an element of $K^\times$. 
	
    The morphism $\alpha$ yields a projective flat family
    \begin{align*}
        (\pi_R:\cC_R\subset\bP^N_R\to\Spec R,\sigma_1,\ldots,\sigma_\ell)    
    \end{align*}
    of stable $\ell$-pointed genus $g$ curves with the marked points given by sections $\sigma_i$ of $\pi_R$ for $1\le i\le\ell$. Additionally, the line bundles $(\omega_{\pi_R}(\sigma_1 + \cdots + \sigma_\ell))^{\otimes p}$ and $\cO_{\cC_R}(1)$ have the same degree on each irreducible component of each geometric fibre of $\pi_R$, with $\omega_{\pi_R}$ being the relative dualizing line bundle, and the restriction map
	\begin{align*}
	R^{N+1} = H^0(\bP^N_R,\cO_{\bP^N_R}(1))\to H^0(\cC_R,\cO_{\cC_R}(1))
	\end{align*}
    gives an isomorphism of $R$-modules, where we use \cite[Theorem III.5.1(a)]{Har77} to compute the $H^0$ group on the left explicitly. Similarly, we get
    \begin{align*}
        (\pi_R':\cC'_R\subset\bP^N_R\to\Spec R,\sigma'_1,\ldots,\sigma'_\ell)    
    \end{align*}
    associated to $\alpha'$. The element $\delta$ now yields an isomorphism $\varphi:\cC_K\to\cC'_K$ over $\Spec K$ (mapping $\sigma_i$ to $\sigma'_i$ for $1\le i\le \ell$) of the two families and an isomorphism $\Phi:\cO_{\cC_K}(1)\simeq\varphi^*\cO_{\cC'_K}(1)$. Taking global sections of $\Phi$ recovers $\delta$. 
	
	By the uniqueness of stable reduction \cite[\href{https://stacks.math.columbia.edu/tag/0E97}{Tag 0E97}]{stacks-project}, we obtain a unique extension $\widehat\varphi:\cC_R\to\cC'_R$ of $\varphi$ to an isomorphism of families of stable $\ell$-pointed genus $g$ curves over $\Spec R$. Observe that $\cO_{\cC_R}(1)$ and $\widehat\varphi^*\cO_{\cC'_R}(1)$ have the same degree on each irreducible component of the geometric fibres of $\pi_R$. Since $\Spec K\subset\Spec R$ is dense and $\cO_{\cC_K}(1)\simeq\varphi^*\cO_{\cC'_K}(1)$, we may apply \cite[Proposition 1]{FP97} to get $\cO_{\cC_R}(1)\simeq\widehat\varphi^*\cO_{\cC'_R}(1)$. Taking global sections of this isomorphism now yields an element of $\GL(N+1,R)$ whose restriction to $K$ differs from $\delta$ by multiplication by an element of $K^\times$.
\end{proof}

\begin{lemma}[Reduction of structure group]\label{General reduction of orbits}
	If $M$ is a second countable smooth manifold with a proper action of $\cG$ with finite stabilizers, then there exists a smooth $\cG$-equivariant map $M\to\cG/G$.
\end{lemma}
\begin{proof}
    Write $\fg_\bC$ for the Lie algebra of $\cG$. We first construct such a map in a $\cG$-invariant neighborhood of any point in $M$. Given $x\in M$, replace it by a point in its orbit to assume that its stabilizer $\Gamma$ is contained in $G$. Choose a (locally closed) $\Gamma$-invariant submanifold $S\subset M$ passing through $x$ such that $T_xS$ is a complement to the ($\Gamma$-equivariant, injective) linearized action map $\fg_\bC\to T_xM$. Restricting the action map $\cG\times M\to M$ gives a smooth $\cG$-equivariant map
	\begin{align*}
	\Phi:(\cG\times S)/\Gamma\to M
	\end{align*}
	where $\Gamma$ acts on $\cG\times S$ by $\gamma\cdot(g,s) = (g\gamma^{-1},\gamma s)$. As
    \begin{align*}
        d\Phi_{(1,x)}\cl \fg_\bC\oplus T_xS\to T_xM    
    \end{align*}
    is an isomorphism, we may shrink $S$ to assume that $\Phi$ is a $\cG$-equivariant local diffeomorphism and, in particular, it is locally injective. We claim that $\Phi$ is actually injective after possibly shrinking $S$ further. If not, then we would have sequences $(g_\nu,x_\nu),(g_\nu',x_\nu')\in\cG\times S$ with
    \begin{align*}
        g_\nu x_\nu = g_\nu'x_\nu',\quad g_\nu^{-1}g_\nu'\not\in\Gamma,\quad\text{and}\quad x_\nu,x'_\nu\to x.
    \end{align*}
    By properness of the action, after passing to a subsequence, we have $g_\nu^{-1}g_\nu'\to\gamma$ for some $\gamma\in G$ with $\gamma x = x$, i.e., $\gamma\in\Gamma$. But now, the sequences $(1,x_\nu)\to(1,x)$ and $(g_\nu^{-1}g_\nu',x_\nu')\to(\gamma,x)$ contradict the local injectivity of $\Phi$. We conclude that $\Phi$ is a $\cG$-equivariant diffeomorphism onto its open image in $M$, i.e., $S$ is a local $\Gamma$-invariant slice at $x$ for the $\cG$-action.\footnote{It may be possible to replace this argument by an appeal to \cite[Proposition 2.2.2]{Pal61}, though that paper has a seemingly different definition of proper actions.} Define a $\cG$-equivariant map $\cN:=\Phi((\cG\times S)/\Gamma)\to\cG/G$ by applying $\Phi\inv$ followed by the obvious projection. Thus, we have solved the problem in the $\cG$-invariant neighborhood $\cN$ of $x$. Further, any smooth $\Gamma$-invariant compactly supported cutoff function $\chi$ on $S$ admits a $\cG$-invariant smooth extension $\tilde\chi$ to $\cN$ and can be extended by zero to obtain a $\cG$-invariant cutoff function $\tilde\chi$ on $M$. Here, we are using the fact that $\cG\cdot\text{supp }\chi$ is closed in $M$, which is a consequence of the action being proper. 
	
    Therefore, the statement follows if we can cover $M$ by a locally finite collection of $\cG$-invariant open subsets, each admitting a smooth $\cG$-equivariant map to $\cG/G$ and then use a $\cG$-invariant smooth partition of unity to patch them. Here we use \eqref{convex-combs}, \eqref{metric-coset} of Lemma \ref{polar-decomp} to make sense of convex combinations in $\cG/G$.
	
    To obtain such a locally finite cover, it suffices to show that the quotient space $N = M/\cG$ is metrizable (and therefore paracompact). By properness of the $\cG$-action, we know that $N$ is Hausdorff. Since $M$ is second countable, so is $N$. By the Urysohn metrization theorem, it remains to show that $N$ is a regular space (i.e., given $y\in N$ and a closed subset $C\subset N$ with $y\not\in C$, there exist open neighborhoods in $N$ separating them). Equivalently, given a closed $\cG$-invariant subset $F\subset M$ and a point $x\not\in F$, we need to find $\cG$-invariant disjoint neighborhoods $\cU$, $\cV$ of $x$, $F$ respectively. As before, let $\Gamma$ be the stabilizer of $x$ and $S$ be a local $\Gamma$-invariant slice at $x$ for the $\cG$-action. Then, $x\in S\setminus F$ and thus, we can choose a $\Gamma$-invariant open neighborhood $U$ of $x$ in $S\setminus F$ such that $\overline U$ is compact and is contained in $S\setminus F$. Now, we can take $\cU = \cG\cdot U$ and $\cV = M\setminus(\cG\cdot\overline{U})$. Note that $\cV\subset M$ is open because the action is proper.
\end{proof}

\begin{corollary}[Existence of $\lambda$]\label{Map to Reduce Orbits} 
For any $N,m,p\ge 1$ and $\ell\ge 0$, there exists a smooth $\cG$-equivariant map
\begin{align*}
    \lambda:\Mbar^{\;*,p}_{g,[\ell]}(\bP^N,m)\ra \cG/G.
\end{align*}
\end{corollary}
\begin{proof}
    This follows from Corollary \ref{alg-base-spaces-are-smooth-cor} combined with Lemmas \ref{Proper Action on Domain Stable Embedded Curves} and \ref{General reduction of orbits}.
\end{proof}

Using Corollary \ref{Map to Reduce Orbits}, we fix a smooth $\cG$-equivariant map $\lambda$ as above. Recall from \eqref{lambda-defined} that $\cU$ and $\lambda$ determine an element $\lambda_\cU(C\subset\bP^N,u)\in\cG/G$ for any framed stable map $(C\subset\bP^N,u)$ satisfying $(\sum_{i\in\Lambda}\chi_i)(C,u)>0$.

\begin{definition}\label{unitary-holo-maps}
    Define $\cM_\bC$ to be the set of framed stable maps $(C\subset\bP^N,u)$ such that $u$ is $J$-holomorphic and we have $\cO_C(1)\simeq\fL_u^{\otimes p}$ as holomorphic line bundles. We equip $\cM_\bC$ with the topology of Gromov convergence. Define $\cM\subset\cM_\bC$ to be the subspace consisting of the framed stable maps $(C\subset\bP^N,u)$ for which $\lambda_\cU(C\subset\bP^N,u)\in\cG/G$ is the identity coset.
\end{definition}

\begin{lemma}\label{projective-frame-bundle}
    Every convergent sequence in $\Mbar_g(X,A;J)$ is the image of some convergent sequence in $\cM_\bC$ under the natural forgetful map
    \begin{align*}
        \cM_\bC\to\Mbar_g(X,A;J).
    \end{align*}
\end{lemma}
\begin{proof}
    For any stable map $(C,u)$ in $\Mbar_g(X,A;J)$, we have $H^1(C,\fL_u^{\otimes p}) = 0$ by Lemma \ref{Stabilisation is Very Positive}. Therefore, the assignment
    \begin{align*}
        (C,u)\mapsto H^0(C,\fL_u^{\otimes p})
    \end{align*}
    defines a continuous complex vector orbi-bundle on $\Mbar_g(X,A;J)$, by the implicit function theorem and linear gluing analysis (e.g., following \cite[Appendix B]{P16}). The key observation now is that $\cM_\bC$ is the projective frame\footnote{For a $\bC$-vector bundle $E\to B$ of rank $r$, its projective frame bundle is the principal $\PGL(r,\bC)$-bundle obtained as the $\bC^\times$ quotient of its frame bundle $\text{Fr}(E)\to B$.} orbi-bundle associated to this vector orbi-bundle. From this, it is immediate that any convergent sequence in $\Mbar_g(X,A;J)$ can be lifted to a convergent sequence in $\cM_\bC$.
\end{proof}

\begin{lemma}\label{unitary-holo-maps-compact}
    The space $\cM$ from Definition \ref{unitary-holo-maps} is compact.
\end{lemma}
\begin{proof}
    The map $\lambda_\cU$ is continuous on $\cM_\bC$ and so, $\cM$ is closed in $\cM_\bC$. Thus, given any sequence $(u_\nu,\iota_\nu):C_\nu\to X\times\bP^N$ of framed stable maps in $\cM$, it will suffice to show that it has a convergent subsequence in $\cM_\bC$. By the compactness of $\Mbar_g(X,A;J)$, after passing to a subsequence, we may assume that we have $(C_\nu,u_\nu)$ converges to some $J$-holomorphic stable map $(C,u)$.

    Using Lemma \ref{projective-frame-bundle}, choose a sequence $(u_\nu,\iota'_\nu):C_\nu\to X\times\bP^N$ in $\cM_\bC$ converging to the point $(u,\iota'):C\to X\times\bP^N$ in $\cM_\bC$. Choose $\varphi_\nu\in\cG = \Aut(\bP^N)$ so that $\varphi_\nu\circ(u_\nu,\iota_\nu) = (u_\nu,\iota_\nu')$. It will suffice to show that $\varphi_\nu$ converges in $\cG$ after passing to a subsequence. For this, note that we have
    \begin{align*}
        \varphi_\nu G = \varphi_\nu\cdot\lambda_\cU(C_\nu,u_\nu,\iota_\nu)=\lambda_\cU(C_\nu,u_\nu,\iota_\nu')\to\lambda_\cU(C,u,\iota')\in\cG/G,
    \end{align*}
    where we have used the fact that $\lambda_\cU$ is $\cG$-equivariant. 
    By Lemma \ref{polar-decomp}\eqref{trivial-principal-bundle},  the projection map $\cG\to\cG/G$ is a principal $G$-bundle and thus proper. Hence, $\varphi_\nu$ has a convergent subsequence.
\end{proof}

\subsection{Choosing the integer $k$}\label{k-choice}

Recall that we fixed $\bC$-linear connections on $T_X$ and $\cO_X(1)$ and note that ${T^*}^{0,1}_{\bP^N}$ carries a $\bC$-linear connection obtained via the Fubini--Study connection and the identification \eqref{antilinear-dual-metric-identification}. As a result, for any framed stable map $C\to X\times\bP^N$, the pullbacks of these vector bundles to $C$ carry holomorphic structures induced by the pullbacks of these connections.

\begin{definition}[$l$-transversality]\label{l-transverse}
    Let $l\ge 1$ be an integer and let $(C\subset\bP^N,u)$ be a framed stable map lying in the space $\cM$ from Definition \ref{unitary-holo-maps}. Define
    \begin{align*}
        E_{(C\subset\bP^N,u),l} := H^0(C,{T^*}^{0,1}_{\bP^N}|_C\otimes u^*T_X\otimes\cO_C(l))\otimes\overline{H^0(\bP^N,\cO_{\bP^N}(l))}.
    \end{align*}
    We say that $(C\subset\bP^N,u)$ is \emph{$l$-transverse} if we have
    \begin{align}\label{constant-rank-obstruction-space-l}
            H^1(C,{T^*}^{0,1}_{\bP^N}|_C\otimes u^*T_X\otimes\cO_C(l)) = 0,
        \end{align}
        and the map
        \begin{align}
        \label{spanning-cokernel-l}
			D(\delbar_J)_u\oplus(\langle\cdot\rangle\circ d\iota_{\tilde C}):\Omega^0(C,u^*T_X)\oplus E_{(C\subset\bP^N,u),l}\to\Omega^{0,1}(\tilde C,\tilde u^*T_X)
		\end{align}
		is surjective. Here, $\iota_{\tilde C}$ and $\tilde u$ are the pullbacks to the normalization $\tilde C\to C$ of the inclusion $C\subset\bP^N$ and the map $u:C\to X$.
\end{definition}

Any $\bC$-linear functional $\hat\varphi:\bC^{N+1}\to\bC$ can be restricted to the fibres of the line bundle $\cO_{\bP^N}(-1)$ to obtain a global section $\varphi$ of $\cO_{\bP^N}(1)$. Identifying $\bC^{N+1}$ with its dual (using the standard basis), the map $\hat\varphi\mapsto\varphi$ yields a $\bC$-linear isomorphism $\bC^{N+1}\simeq H^0(\bP^N,\cO_{\bP^N}(1))$. We refer to this below as the \emph{standard identification}.

\begin{lemma}[Partition of unity]\label{handy-numerical-identity}
    Let $\hat\sigma_0,\ldots,\hat\sigma_N$ be an orthonormal basis of $\bC^{N+1}$ with respect to the standard Hermitian inner product and let $\sigma_0,\ldots,\sigma_N$ be the corresponding elements of $H^0(\bP^N,\cO_{\bP^N}(1))$ under the standard identification $\bC^{N+1} \simeq H^0(\bP^N,\cO_{\bP^N}(1))$. Then, using the standard Hermitian metric on the line bundle $\cO_{\bP^N}(1)$, we have the identity $\sum_{j=0}^N|\sigma_j(x)|^2 = 1$ for all points $x\in\bP^N$.
\end{lemma}
\begin{proof}
    Omitted.
\end{proof}

\begin{lemma}[Monotonicity of $l$-transversality]\label{monotone-l-transverse}
    Let $l\ge 1$ be an integer and let $(C\subset\bP^N,u)$ be a framed stable map in $\cM$.
    \begin{enumerate}[\normalfont(i)]
        \item If $(C\subset\bP^N,u)$ and $l$ satisfy \eqref{constant-rank-obstruction-space-l}, then the same is true after replacing $l$ by any larger integer.
        \item If $(C\subset\bP^N,u)$ and $l$ are such that \eqref{spanning-cokernel-l} is surjective, then the same is true after replacing $l$ by any larger integer.
    \end{enumerate}
    In particular, if $(C\subset\bP^N,u)$ is $l$-transverse, then it is also $l'$-transverse for all integers $l'\ge l$.
\end{lemma}
\begin{proof}
    \begin{enumerate}[(i)]
        \item Serre duality gives an identification
        \begin{align*}
            H^1(C,{T^*}^{0,1}_{\bP^N}|_C\otimes u^*T_X\otimes\cO_C(l))^* = H^0(C,({T^*}^{0,1}_{\bP^N}|_C\otimes u^*T_X\otimes\cO_C(l))^*\otimes\omega_C)   
        \end{align*}
        for any integer $l$, where $\omega_C$ is the dualizing line bundle of $C$. The dimension of the right side is monotonically decreasing in $l$ since, by picking a generic hyperplane section $D$ of $C\subset\bP^N$, the sheaf $\cO_C(l+1)^* = \cO_C(-l)\otimes\cO_C(-D)$ may be viewed as a subsheaf of $\cO_C(l)^* = \cO_C(-l)$. Thus, if \eqref{constant-rank-obstruction-space-l} holds for some $l\ge 1$, then it continues to hold even if we replace $l$ by any $l'\ge l$.
        \item Let $\sigma_0,\ldots,\sigma_N$ be as in Lemma \ref{handy-numerical-identity}. For any integer $l\ge 1$, define
        \begin{align*}
            E_{(C\subset\bP^N,u),l}&\to E_{(C\subset\bP^N,u),l+1}\\
            f\otimes\overline g&\mapsto\textstyle\sum_{j=0}^N f_j\otimes \overline g_j,
        \end{align*}
        where $f_j = f\otimes\sigma_j|_C$ and $g_j = g\otimes\sigma_j$. Lemma \ref{handy-numerical-identity} shows that these maps are compatible with the maps $\langle\cdot\rangle\circ d\iota_{\tilde C}: E_{(C\subset\bP^N,u),l}\to\Omega^{0,1}(\tilde C,\tilde u^*T_X)$ as $l\ge 1$ varies. This shows that if \eqref{spanning-cokernel-l} is surjective for some $l\ge 1$, then it remains surjective even if we replace $l$ by any $l'\ge l$.
    \end{enumerate}
    
    \noindent The assertion on $l$-transversality follows by combining (i) and (ii).
\end{proof}

\begin{lemma}[Extension of sections]\label{extension-to-projective-space}
    There exists a positive integer $k_0$ with the following property. For any $C\subset\bP^N$ corresponding to a point of $\Mbar^*_{g}(\bP^N,m)$ and any integer $l\ge k_0$, the restriction map
    \begin{align*}
        H^0(\bP^N,\cO_{\bP^N}(l))\to H^0(C,\cO_C(l))
    \end{align*}
    is surjective.
\end{lemma}
\begin{proof}
    Let $\cI_{C/\bP^N}$ be the ideal sheaf of $C\subset\bP^N$, defined by the short exact sequence
    \begin{align*}
        0\to\cI_{C/\bP^N}\to\cO_{\bP^N}\to\cO_C\to 0
    \end{align*}
    of coherent sheaves on $\bP^N$. It suffices to show that there is an integer $k_0\ge 1$, depending only on $g$, $m$ and $N$, such that $H^1(C,\cI_{C/\bP^N}(l)) = 0$ for all $l\ge k_0$. Noting that $C\subset\bP^N$ has Hilbert polynomial $P(t) = mt-g+1$, which depends only on $m$ and $g$, this follows from \cite[Lemma IX.4.1(i)]{ACG-moduli}.
\end{proof}

\begin{lemma}[Achieving $l$-transversality for one stable map]\label{achieving-l-transverse}
    Let $(C\subset\bP^N,u)$ be a framed stable map in $\cM$. Then, there is an integer $l\ge 1$ such that $(C\subset\bP^N,u)$ is $l$-transverse.
\end{lemma}
\begin{proof}
    We will find two separate values of $l$, with one satisfying \eqref{constant-rank-obstruction-space-l} and the other such that \eqref{spanning-cokernel-l} is surjective. By Lemma \ref{monotone-l-transverse}, the larger of these two values will yield $l$-transversality. 
    
    By Serre's vanishing theorem \cite[Theorem III.5.2]{Har77}, all sufficiently large values of $l$ satisfy \eqref{constant-rank-obstruction-space-l}. By \cite[Lemma 6.24 and Proposition 6.26]{AMS21},
    \begin{align*}
        \langle\cdot\rangle\circ d\iota_{\tilde C}: H^0(C,{T^*}^{0,1}_{\bP^N}|_C\otimes u^*T_X\otimes\cO_C(l))\otimes\overline{H^0(C,\cO_{C}(l))}\to\Omega^{0,1}(\tilde C,\tilde u^*T_X)
    \end{align*}
    induces a surjection onto $\coker(D(\delbar_J)_u)$ for arbitrarily large values of $l$. Combining this with Lemma \ref{extension-to-projective-space} gives an integer $l$ such that \eqref{spanning-cokernel-l} is surjective.
\end{proof}

We will now give an interpretation of $l$-transversality as a special case of transversality for pseudo-holomorphic maps (i.e., surjectivity of the linearized Cauchy--Riemann operator). This allows us to conclude (in Lemma \ref{open-l-transverse} below) that $l$-transversality is an open condition on $\cM$.

\begin{definition}\label{E_l-acs}
    For any integer $l\ge 1$, define a $\bC$-vector bundle on $X\times\bP^N$ by
    \begin{align*}
        E_l = T_X\boxtimes({T^*}^{0,1}_{\bP^N}\otimes\cO_{\bP^N}(l)\otimes\overline{H^0(\bP^N,\cO_{\bP^N}(l))}).
    \end{align*}
    Using the $\bC$-linear connections on all of the bundles involved, we obtain a splitting of the tangent bundle of $E_l$ into the horizontal $X$- and $\bP^N$-directions and the vertical $E_l$-direction, i.e., we get an identification
    \begin{align}\label{E_l-splitting}
        T_{E_l} = \pi_l^*(T_{X\times\bP^N}\oplus E_l),
    \end{align}
    where $\pi_l:E_l\to X\times\bP^N$ is the projection. At any point $(x,y,\eta)\in E_l$, with $x\in X$, $y\in\bP^N$ and $\eta\in (E_l)_{(x,y)}$, we use this splitting of the tangent space of $E_l$ to define the endomorphism
    \begin{align*}
        \tilde J_{(x,y,\eta)}^{E_l} := 
        \begin{bmatrix}  
            J_x & 2J_x\langle\eta\rangle & 0\\
            0 & J_y^\bP & 0\\
            0 & 0 & J^{E_l}_{(x,y)}
        \end{bmatrix}.
    \end{align*}
    Here, $J^\bP$ is the standard complex structure on $\bP^N$, $J^{E_l}$ is the multiplication by $i$ on the fibres of $E_l$ and $\langle\eta\rangle:T_{\bP^N,y}\to T_{X,x}$ is the $\bC$-antilinear map obtained from $\eta$ using the standard Hermitian inner product on $\cO_{\bP^N}(l)$.
    
    We readily check that $\tilde J^{E_l}$ squares to $-\text{Id}$ and thus, defines an almost complex structure on $E_l$. It is also clear that the inclusion of $X\times\bP^N$ into $E_l$ as the zero section is pseudo-holomorphic, as is the projection $E_l\to\bP^N$.
\end{definition}

The following statement is analogous to \cite[Lemma 6.18]{AMS21}.

\begin{lemma}[Gromov trick]\label{E_l-gromov-trick}
    Let $C$ be a Riemann surface and let $F:C\to E_l$ be a smooth map. Let $u:C\to X$ and $v:C\to\bP^N$ be the maps with $\pi_l\circ F = (u,v)$ and let $\eta$ be the smooth section of $(u,v)^*E_l$ induced by $F$. 
    
    The map $F$ is $\tilde J^{E_l}$-holomorphic if and only if
    \begin{enumerate}[\normalfont(1)]
        \item the map $v$ is holomorphic,
        \item we have $\eta\in H^0(C,u^*T_X\otimes v^*({T^*}^{0,1}_{\bP^N}\otimes\cO_{\bP^N}(l)))\otimes\overline{H^0(\bP^N,\cO_{\bP^N}(l))}$ and
        \item the equation $\delbar_Ju + \langle\eta\rangle\circ dv = 0$ is satisfied.
    \end{enumerate}
\end{lemma}
\begin{proof}
    Let $\nabla$ be the pullback connection on $(u,v)^*E_l$ and let $j_C$ be the almost complex structure on $C$. With respect to the splitting \eqref{E_l-splitting} of $T_{E_l}$, we have
    \begin{align}\label{delbar-E_l}
        dF = \begin{bmatrix}
            du \\ dv \\ \nabla\eta
        \end{bmatrix}\,\Rightarrow\,
        \delbar_{\tilde J^{E_l}}F = {\textstyle\frac12}(dF + \tilde J^{E_l}\circ dF\circ j_C) = 
        \begin{bmatrix}
            \delbar_J u + J\langle\eta\rangle\, dv\, j_C \\ \delbar_{J^\bP}v \\ \nabla^{0,1}\eta    
        \end{bmatrix},
    \end{align}
    from which the claim is immediate.
\end{proof}

\begin{lemma}[Transversality in $E_l$ is $l$-transversality]\label{l-transverse-is-transverse}
    Let $(u,\iota):C\to X\times\bP^N$ be a framed stable map in $\cM$. Regard $(u,\iota)$ as a $\tilde J^{E_l}$-holomorphic map $F:C\to E_l$ via the inclusion of $X\times\bP^N$ into $E_l$ as the zero section. 
    
    Then, the linearized Cauchy--Riemann operator $D(\delbar_{\tilde J^{E_l}})_F$ is surjective if and only if the framed stable map is $(u,\iota):C\to X\times\bP^N$ is $l$-transverse.
\end{lemma}
\begin{proof}
    Using \eqref{delbar-E_l}, we will first compute the linearization
    \begin{align*}
        D(\delbar_{\tilde J^{E_l}})_F : \Omega^0(C,F^*T_{E_l})\to\Omega^{0,1}(\tilde C,\tilde F^*T_{E_l}),
    \end{align*}
    where $\tilde F$ is the pullback of $F$ to the normalization $\tilde C\to C$. Similarly, let $\iota_{\tilde C}$ (resp. $\tilde u$) be the pullbacks of $\iota$ (resp. $u$) to $\tilde C$. Since $F$ maps into the zero section of $E_l$, we have a natural splitting $F^*T_{E_l} = u^*T_X\oplus\iota^*T_{\bP^N}\oplus(u,\iota)^*E_l$ into horizontal and vertical sub-bundles (defined without reference to the connection). This splitting coincides with the one induced by \eqref{E_l-splitting} and so, we need not distinguish between them in our computation below.
    
    In terms of the notation of \eqref{delbar-E_l}, we have $\eta = 0$ and $v = \iota$. Differentiating the expression in \eqref{delbar-E_l} in the $u$, $v$ and $\eta$ directions and writing the result as a matrix using the splitting $F^*T_{E_l} = u^*T_X\oplus\iota^*T_{\bP^N}\oplus(u,\iota)^*E_l$, we get
    \begin{align}\label{E_l-linearization}
        D(\delbar_{\tilde J^{E_l}})_F = \begin{bmatrix}
            D(\delbar_J)_u & 0 & \langle\cdot\rangle\,d\iota_{\tilde C} \\
            0 & D(\delbar_{J^\bP})_\iota & 0 \\
            0 & 0 & \nabla^{0,1}
        \end{bmatrix}.
    \end{align}
    For the purpose of verifying this computation, we may reduce to the case where $C$ is smooth, i.e., $\tilde C = C$. The following points now justify the computation \eqref{E_l-linearization}. 
    \begin{enumerate}[(a)]
        \item The bottom row is the linearization of $(u,v,\eta)\mapsto\nabla^{0,1}\eta$. Since we are computing it at a point with $\eta = 0$, the dependence of $\nabla^{0,1}$ on $(u,v)$ does not make an appearance in the linearization. Moreover, $\nabla^{0,1}$ is linear in the $\eta$ direction.
        \item The middle row records the linearization of $(u,v,\eta)\mapsto\delbar_{J^\bP}v$, which is manifestly independent of $(u,\eta)$. This explains the two entries which are zero. We also note that $D(\delbar_{J^\bP})_\iota$ is the standard Cauchy--Riemann operator on the holomorphic vector bundle $\iota^*T_{\bP^N}$.
        \item  The top row records the linearization of $(u,v,\eta)\mapsto\delbar_Ju + J\langle\eta\rangle\,dv\,j_C$ at a point with $\eta = 0$. The $\delbar_Ju$ term is clearly independent of $(v,\eta)$ and gives rise to the top left matrix entry $D(\delbar_J)_u$. The $(u,v)$ dependence in the term $J\langle\eta\rangle\,dv\,j_C$ makes no appearance in the linearization, since we are computing it at a point where $\eta = 0$. Moreover, by linearity of this term in the $\eta$ direction, we find that its linearization in this direction is
        \begin{align*}
            \eta'\mapsto J\langle\eta'\rangle\,d\iota\,j_C = J\langle\eta'\rangle\,J^{\bP}\,d\iota = -J^2\langle\eta'\rangle\,d\iota = \langle\eta'\rangle\,d\iota,      
        \end{align*}
        where we have used the holomorphicity of $\iota$ and the $\bC$-antilinearity of $\langle\eta'\rangle$.
    \end{enumerate}
     
    Now, let us investigate the surjectivity of \eqref{E_l-linearization}. From Lemma \ref{Regular Embedding via Framing} and the proof of smoothness in Lemma \ref{alg-base-spaces-are-smooth-qproj}, we already know that $D(\delbar_{J^\bP})_\iota$ is surjective. Thus, the surjectivity of \eqref{E_l-linearization} is equivalent to the surjectivity of the Cauchy--Riemann type operator on $u^*T_X\oplus(u,\iota)^*E_l$ given by the matrix 
    \begin{align*}
        L = \begin{bmatrix}
            D(\delbar_J)_u & \langle\cdot\rangle\,d\iota_{\tilde C} \\
            0 & \nabla^{0,1}
        \end{bmatrix}.
    \end{align*}
    As $L$ is upper triangular, we deduce that it is surjective if and only if $\nabla^{0,1}$ is surjective \emph{and} the restriction of $\langle\cdot\rangle\,d\iota_{\tilde C}$ to $\ker(\nabla^{0,1})$ maps onto $\coker (D(\delbar_J)_u)$. Observe that the surjectivity of $\nabla^{0,1}$ is the same as \eqref{constant-rank-obstruction-space-l}. Similarly, the surjectivity of the map $\ker(\nabla^{0,1})\to\coker(D(\delbar_J)_u)$ induced by $\langle\cdot\rangle\,d\iota_{\tilde C}$ is equivalent to the surjectivity of \eqref{spanning-cokernel-l}. We conclude that $D(\delbar_{\tilde J^{E_l}})_F$ is surjective if and only if $(u,\iota):C\to X\times\bP^N$ is $l$-transverse. 
\end{proof}

\begin{lemma}[Openness of $l$-transversality]\label{open-l-transverse}
    Let $l\ge 1$ be an integer. Then, the subset of $\cM$ consisting of $l$-transverse framed stable maps is open. 
\end{lemma}
\begin{proof}
    We have an embedding of $\cM$ inside the moduli space of $\tilde J^{E_l}$-holomorphic curves in $E_l$ using the embedding of $X\times\bP^N$ into $E_l$ as the zero section. 
    
    By
    \cite[Proposition 9.2.6(i)]{P16}, having surjective linearized Cauchy--Riemann operator is an open condition with respect to the Gromov topology on pseudo-holomorphic curves. Combining this with Lemma \ref{l-transverse-is-transverse}, we conclude that the subset of $\cM$ consisting of $l$-transverse framed stable maps is open.
\end{proof}

\begin{lemma}[Existence of $k$]\label{k-existence}
    There exists an integer $k\ge 1$ such that all framed stable maps in $\cM$ are $k$-transverse.
\end{lemma}
\begin{proof}
    Define the function $l_\text{min}:\cM\to\bZ_{\ge 1}$ by
    \begin{align*}
        l_\text{min}(C\subset\bP^N,u) = \min\{l\in\bZ_{\ge 1}:(C\subset\bP^N,u)\text{ is }l\text{-transverse}\},
    \end{align*}
    which is well-defined by Lemma \ref{achieving-l-transverse}. By definition, for $l_0\in\bZ_{\ge 1}$, we have $l_\text{min}(C\subset\bP^N,u)<l_0$ if and only if $(C\subset\bP^N,u)$ is $l$-transverse for some integer $l$ satisfying $1\le l< l_0$. Using Lemma \ref{open-l-transverse}, this shows that $l_\text{min}$ is an upper semicontinuous function and thus, it achieves a maximum value $k\in\bZ_{\ge 1}$ on the compact space $\cM$ (Lemma \ref{unitary-holo-maps-compact}). 
    
    To complete the proof, we need to show that any $(C\subset\bP^N,u)$ in $\cM$ is $k$-transverse. Since we have $l_\text{min}(C\subset\bP^N,u)\le k$, we can find some integer $l$ with $1\le l\le k$ such that $(C\subset\bP^N,u)$ is $l$-transverse. Lemma \ref{monotone-l-transverse} now shows that $(C\subset\bP^N,u)$ must be $k$-transverse.
\end{proof}

Fix $k\ge 1$ to smallest integer satisfying the conclusion of Lemma \ref{k-existence}. Observe that our choice of $k$ is such that condition \eqref{killing-obstructions} of Definition \ref{unobstructed-aux} is satisfied.

\subsection{Completing the proof of Theorem \ref{global-kuranishi-existence}\eqref{achieveing-transversality}}

The preceding subsections show how to choose an unobstructed auxiliary datum $(\nabla^X,\cO_X(1),p,\cU,\lambda,k)$. 

In more detail, it is always possible to find a $\bC$-linear connection $\nabla^X$ on $T_X$ and a polarization $\cO_X(1)\to X$, the latter following from Lemma \ref{Suitable Line Bundle on Target}. Having fixed $\nabla^X$ and $\cO_X(1)$, Lemmas \ref{Stabilisation is Very Positive}, \ref{good-coverings-exist}, \ref{Map to Reduce Orbits} and \ref{k-existence} respectively show how to choose $p$, $\cU$, $\lambda$ and $k$ so that conditions \eqref{very-ample-acyclic-line-bundle} and \eqref{killing-obstructions} of Definition \ref{unobstructed-aux} are satisfied for these choices. 
This completes the proof of Theorem \ref{global-kuranishi-existence}\eqref{achieveing-transversality}.
\section{Global Kuranishi chart}\label{transversality-implies-smoothness-proof}

In this section, we will prove Theorem \ref{global-kuranishi-existence}\eqref{transversality-implies-smoothness}. To break up the proof, let us first state the properties of our global Kuranishi chart explicitly. Their proofs will occupy the following subsections.

\begin{proposition}\label{gkc-properties-explicit} Suppose $(\nabla^X,\cO_X(1),p,\cU,\lambda,k)$ is an unobstructed auxiliary datum. The output $\cK_n = (G_n,\cT_n,\cE_n,\fs_n)$ of Construction \ref{high-level-description-defined} has the following properties.
    \begin{enumerate}[\normalfont(a)]
			\item\label{thickening-is-smooth} The projection $\cT_n\to\Mbar^*_{g,n}(\bP^N,m)$ carries a natural rel--$C^\infty$ structure of the expected dimension in a $G$-invariant open neighborhood $\cT^{\normalfont\text{reg}}_n$ of $\fs_n^{-1}(0)$.
			\item\label{bundle-section-is-smooth} $\cE^{\normalfont\text{reg}}_n := \cE_n|_{\cT^{\normalfont\text{reg}}_n}$ over $\cT^{\normalfont\text{reg}}_n\to\Mbar^*_{g,n}(\bP^N,m)$ carries a natural rel--$C^\infty$ vector bundle structure of the expected rank such that the section $\fs_n$ is of class rel--$C^\infty$ over $\normalfont\cT_n^\text{reg}$.
			\item\label{action-is-smooth} The $G$-action on $\cT^{\normalfont\text{reg}}_n$ and $\cE^{\normalfont\text{reg}}_n$ is rel--$C^\infty$ and the stabilizer of every point of $\cT_n$ in a neighborhood of $\fs_n^{-1}(0)$ is finite. In particular, the $G$-actions are locally linear in the sense of {\normalfont\cite[Definition 4.20]{AMS21}}.
			\item\label{footprint} The zero locus $\fs_n\inv(0)$ is compact and the forgetful map $$\fs_n^{-1}(0)/G\to\Mbar_{g,n}(X,A;J)$$ is a homeomorphism.
			\item\label{orientation-complex-structure} The virtual vector bundle given by
			\begin{align*}
			    T_{\cT^{\normalfont\text{reg}}_n/\Mbar^*_{g,n} (\bP^N,m)} - (\cE_n^{\normalfont\text{reg}}\oplus \underline{\fg})
			\end{align*}
			has a natural stably complex (virtual) vector bundle lift in a neighborhood of the zero locus $\fs_n^{-1}(0)$, where $T_{\cT^{\normalfont\text{reg}}_n/\Mbar^*_{g,n} (\bP^N,m)}$ is the vertical tangent bundle and $\underline\fg$ is the trivial bundle with fibre $\fg = \normalfont\text{Lie}(G)$.
		\end{enumerate}
  In summary, 
$$\cK_n^{\reg} := (G,\cT_n^{\reg}/\Mbar_{g,n}^*(\bP^N,m),\cE_n^{\reg},\fs|_{\cT_n^{\reg}})$$ 
is a stably complex global Kuranishi chart for $\Mbar_{g,n}(X,A;J)$ of the correct virtual dimension.
\end{proposition}

We first reduce the statement of Proposition \ref{gkc-properties-explicit} for $n \geq 0$ to the case of $n = 0$. As stated in Remark \ref{A=0-exception}, we leave the straightforward modification of this argument for the special case $\lspan{[\omega],A} = 0$ and $g\le 1$ to the reader.

\begin{lemma}[Reduction to $n = 0$]\label{marked-points-formal-consequence}
    Assume that $(\lspan{[\omega],A},g)$ is not $(0,0)$ or $(0,1)$. Then, Proposition \ref{gkc-properties-explicit} for $n>0$ follows from the case $n=0$.
\end{lemma}

\begin{proof}
    Recall, from \eqref{marked-point-chart-defined}, that the global Kuranishi chart $\cK_n$ for $n>0$ was defined by pulling back the chart $\cK$ for $n = 0$ under the natural forgetful map 
    \begin{equation}\label{forget-base}
        \Mbar_{g,n}^*(\bP^N,m)\to\Mbar_g^*(\bP^N,m).
    \end{equation}
    Thus, parts \eqref{thickening-is-smooth}, \eqref{bundle-section-is-smooth}, \eqref{action-is-smooth} and \eqref{orientation-complex-structure} of Proposition \ref{gkc-properties-explicit} for general $n$ are formal consequences of the $n = 0$ case. For part \eqref{footprint}, we first produce a continuous bijection 
    \begin{align}\label{marked-point-footprint-map}
        \fs_n^{-1}(0)/G\to\Mbar_{g,n}(X,A;J)
    \end{align}
    as follows. An element of $\fs_n^{-1}(0) = \fs^{-1}(0)\times_{\Mbar_g^*(\bP^N,m)}\Mbar^*_{g,n}(\bP^N,m)$ consists of
    \begin{enumerate}
        \item a stable map $(\Sigma,x_1,\ldots,x_n,f\cl \Sigma\to\bP^N)$, such that when we forget the marked points $x_1,\ldots,x_n$ and stabilize we get the element $C = f(\Sigma)\subset\bP^N$ of $\Mbar_g^*(\bP^N,m)$ and
        \item a stable $J$-holomorphic map $u\cl C\to X$ such that $\cO_C(1)\simeq\cL_u^{\otimes p}$ and $\lambda_\cU(C\subset\bP^N,u)$ is the identity coset in $\cG/G$.
    \end{enumerate}
    To this element, we associate the stable map $(\Sigma,x_1,\ldots,x_n,v\cl \Sigma\to X)$, where $v$ is the $J$-holomorphic map defined by the composition of $u\cl C\to X$ with the projection $f\cl \Sigma\to C$. To see that $(\Sigma,x_1,\ldots,x_n,v)$ is indeed stable, and therefore determines an element of $\Mbar_{g,n}(X,A;J)$, consider any irreducible component $\Sigma'$ of $\Sigma$ on which is contracted by $v$ (i.e., $v|_{\Sigma'}$ is constant). If $f\cl \Sigma\to C$ contracts $\Sigma'$, then the stability of $(\Sigma,x_1,\ldots,x_n,f)$ ensures that $\Sigma'$ has enough special points on it to make it a stable curve. If $f$ does not contract $\Sigma'$, but $u$ contracts its image $f(\Sigma') = C'$, then the stability of $u\cl C\to X$ ensures that $C'$ has enough special points on it to make it a stable curve. In either case, we conclude that $\Sigma'$ has enough special points on it. Thus, we obtain a well-defined map \eqref{marked-point-footprint-map}.

    To show bijectivity, we construct an inverse. Given a stable map $(\Sigma,x_1,\ldots,x_n,v\cl \Sigma\to X)$ belonging to $\Mbar_{g,n}(X,A;J)$, let $u\cl C\to X$ be the stable map obtained by forgetting the marked points and stabilizing it. Let $\kappa\cl \Sigma\to C$ be the canonical map which contracts the unstable components. We know from the $n = 0$ case (that is, the very ampleness of $\fL_u^{\otimes p}$) that $u\cl C\to X$ can be lifted to an element $(\iota: C\hookrightarrow\bP^N,u,0,0)$ belonging to $\fs^{-1}(0)$. It is now immediate that $(\Sigma,x_1,\ldots,x_n,\iota\circ \kappa\cl \Sigma\to\bP^N)$ can be used to lift $(\Sigma,x_1,\ldots,x_n,v)$ to $\fs_n^{-1}(0)$. This shows how to invert \eqref{marked-point-footprint-map}. 

    Since \eqref{forget-base} is proper, so is $\fs_n\inv(0)\to \fs\inv(0)$. It follows that $\fs_n\inv(0)/G$ is compact.
    Since $\Mbar_{g,n}(X,A;J)$ is Hausdorff, the continuous bijection \eqref{marked-point-footprint-map} is a homeomorphism.
\end{proof}

For the remainder of this section, we focus on the $n = 0$ case and omit $n$ from the notation. To keep the notation readable, we use the same notation for the thickening $\cT$ and the open locus $\cT^\text{reg}\subset\cT$ where it is cut-out transversely (and similarly for $\cE$). Thus, all the statements made in \textsection\ref{transversality-implies-smoothness-proof} are to be interpreted as being valid only over a sufficiently small $G$-invariant open neighborhood of the zero locus $\fs^{-1}(0)\subset\cT$.

\subsection{Thickening}\label{smoothness-of-thickening}

Consider the space $\cT'$ of tuples 
\begin{align*}
    (C\subset\bP^N,u\cl C\to X,\eta)    
\end{align*}
satisfying conditions \eqref{thickening-condition-1} and \eqref{thickening-condition-2} of Construction \ref{high-level-description-defined}. We will endow $\cT'$ with a natural rel--$C^\infty$ structure (over the smooth quasi-projective scheme $\Mbar^*_g(\bP^N,m)$ from Definition \ref{base-space-de}) by realizing it as a moduli space of pseudo-holomorphic curves.

Let $E:=E_k$ be the almost complex manifold from Definition \ref{E_l-acs}, with $k$ being the last entry of the auxiliary datum we fixed at the beginning of \textsection\ref{transversality-implies-smoothness-proof}. The natural projection $p_{\bP^N}:E\to\bP^N$ is pseudo-holomorphic, where $\bP^N$ carries the standard complex structure.
By Lemma \ref{E_l-gromov-trick} (Gromov trick), the points of $\cT'$ can be regarded as genus $g$ pseudo-holomorphic curves in $E$ lying in the class 
\begin{align*}
    \tilde A:= A\times[\text{pt}] + [\text{pt}]\times m[\bP^1]\in H_2(X\times\bP^N,\bZ) = H_2(E,\bZ).    
\end{align*}
In view of this, we will freely conflate points of $\cT'$ with pseudo-holomorphic curves in $E$. By Lemma \ref{l-transverse-is-transverse} and unobstructedness of the auxiliary datum, the linearized Cauchy--Riemann operator is surjective for any tuple $(C\subset\bP^N,u,0)$ in $\cT'$.

To endow $\cT'\to\Mbar_g^*(\bP^N,m)$ with a rel--$C^\infty$ structure, we will apply the representability result recalled in Theorem \ref{rel-smooth-representability} to the following moduli functor.

\begin{definition}[Moduli functor for $\cT'$]\label{functor-of-points}
    Define the functor
    \begin{align*}
        \fF:(C^\infty/\cdot)^\text{op}\to(\text{Sets})
    \end{align*}
    as follows. Given a rel--$C^\infty$ manifold $Z/T$, we define the set $\fF(Z/T)$ to consist of all commutative diagrams of the form
    \begin{center}
    \begin{tikzcd}
        Z  \arrow[d,""]&\cC_Z \arrow[l,""]\arrow[d,""]\arrow[r,"F"] & E\arrow[d,"p_{\bP^N}"]\\ 
        T & \cC_T \arrow[r,"f"]\arrow[l,"\pi_T"] & \bP^N
    \end{tikzcd} 
    \end{center}
    satisfying the following properties.
    \begin{enumerate}
        \item The diagram
        \begin{center}
        \begin{tikzcd}
            \cC_T \arrow[d,"\pi_T"] \arrow[r,"f"] & \bP^N \\
            T & 
        \end{tikzcd}
        \end{center}
        is the pullback of the universal family on $\Mbar_g^*(\bP^N,m)$ along a continuous map $T\to\Mbar_g^*(\bP^N,m)$.
        The map $\cC_Z\to Z$ is obtained by pulling back $\cC_T\to T$ along $Z\to T$ (and so, $\cC_Z/\cC_T$ has a natural rel--$C^\infty$ structure).
        \item The morphism $(F,f)\cl \cC_Z/\cC_T\to E/\bP^N$ is of class rel--$C^\infty$.
        \item\label{functor-submersion} For each $z\in Z$ with image $t\in T$, the restriction 
        \begin{align*}
            F|_z\cl \pi_T^{-1}(t)\to E    
        \end{align*}
        of the map $F$ to the fibre of $\cC_Z\to Z$ over $z$ is a pseudo-holomorphic stable map to $E$ of genus $g$ and class $\tilde A$ with surjective linearized Cauchy--Riemann operator. Moreover, the projection map from the kernel of the linearized operator of $F|_z\cl \pi_T^{-1}(t)\to E$ to the kernel of the linearized operator of $p_{\bP^N}\circ (F|_z) = f|_t\cl \pi_T^{-1}(t)\to\bP^N$ is surjective.
    \end{enumerate} 
    For rel--$C^\infty$ morphisms $Z'/T'\to Z/T$, the associated functorial maps 
    \begin{align*}
        \fF(Z/T)\to\fF(Z'/T')    
    \end{align*}
    are given by pullbacks of such diagrams.
\end{definition}

\begin{proposition}[Rel--$C^\infty$ structure on $\cT'$]\label{rep-shear}
    The functor $\fF$ is represented by a rel--$C^\infty$ structure on an open subset of $\cT'/\Mbar_g^*(\bP^N,m)$, all of whose points have surjective linearized Cauchy--Riemann operators. 
    
    This open subset contains all elements of $\cT'$ of the form $(C\subset\bP^N,u,0)$, where $(C\subset\bP^N,u)$ lies in the space $\cM$ from Definition \ref{unitary-holo-maps}.
\end{proposition}
\begin{proof}
    With respect to the obvious notion of open covers in $(C^\infty/\cdot)$, the functor $\fF$ is a \emph{sheaf} (see \cite[Definition 2.13]{Swa21}). Thus, the representability of $\fF$ by a rel--$C^\infty$ manifold can be checked locally (see \cite[Proposition 2.16]{Swa21}).
    
    Fix a point $\tilde x = (C\subset\bP^N,u,\eta)$ in $\cT'$, with $x = (C\subset\bP^N)$ being its image in $\Mbar_g^*(\bP^N,m)$, such that it defines an element of $\fF(\text{pt}/\text{pt})$. For later use, let us denote the inclusion of $C$ in $\bP^N$ as $\iota:C\hookrightarrow\bP^N$. By assumption, the linearized Cauchy--Riemann operators of the pseudo-holomorphic maps $(u,\iota,\eta):C\to E$ and $\iota:C\to\bP^N$ are surjective and the map between from the kernel of the linearized operator of $(u,\iota,\eta)$ to the kernel of the linearized operator of $\iota$ is also surjective. We will show representability of $\fF$ in a neighborhood of $\tilde x$ and $x$. In particular, this discussion will apply to any $\tilde x = (C\subset\bP^N,u,0)$ such that $(C\subset\bP^N,u)$ lies in the space $\cM$ from Definition \ref{unitary-holo-maps} (by Lemma \ref{l-transverse-is-transverse} and unobstructedness of the auxiliary datum fixed at the beginning of \textsection\ref{transversality-implies-smoothness-proof}).

    Choose a minimal collection $p_1,\ldots,p_r$ of non-singular points on $C$ so that $(C,p_1,\ldots,p_r)$ is a stable pointed curve. Choose a collection of pairwise disjoint local divisors $D_1,\ldots,D_r\subset\bP^N$, defined as the zero loci of local holomorphic functions $h_1,\ldots,h_r$, such that the intersection $C\cap D_i = \{p_i\}$ is transverse for $1\le i\le r$. Choose a local holomorphic universal deformation
    \begin{center}
    \begin{tikzcd}
        C \arrow[r,hook,"i"] \arrow[d] & \fC \arrow[d,"\pi_\fB"] \\
        \{s\} \arrow[r,hook] \arrow[u,bend left,"p_j"] & \fB \arrow[u,bend left,"\sigma_j"]
    \end{tikzcd}    
    \end{center}
  of the stable pointed curve $(C,p_1,\ldots,p_r)$, with $(\fB,s)$ being a pointed complex manifold of complex dimension $3g-3+r$.\footnote{Explicitly, take a complex manifold $\fB$ with a finite group $\Gamma$ acting on it by biholomorphic maps such that $[\fB/\Gamma]$ provides a complex orbifold chart for $\Mbar_{g,r}$, with $s\in\fB$ being a point mapping to $(C,p_1,\ldots,p_r)$ under the composition $\fB\to[\fB/\Gamma]\hookrightarrow\Mbar_{g,r}$. Let $(\pi_\fB:\fC\to\fB,\{\sigma_j\})$ be the pullback of the universal family of stable $r$-pointed curves under $\fB\to\Mbar_{g,r}$ and let $i$ be an isomorphism of $(C,x_1,\ldots,x_r)$ with the fibre of this family over $s$. We point out that the family $(\pi_\fB:\fC\to\fB,\{\sigma_j\})$ carries a holomorphic action of $\Gamma$ extending the original action of $\Gamma$ on $\fB$ but this is not relevant for our argument.} We will denote the fibre of this family over any $s'\in\fB$ by $(\fC_{s'},\sigma_1(s'),\ldots,\sigma_r(s'))$. 
  
  Now, consider the moduli spaces
    \begin{align}\label{unconstrained-mod-spaces}
        \fM^\text{reg}(\pi_\fB,\bP^N)_{m[\bP^1]}\quad\text{and}\quad\fM^\text{reg}(\pi_\fB,E)_{\tilde A}
    \end{align}
    of pseudo-holomorphic maps from fibres of $\pi_\fB$ to $\bP^N$ and $E$ (respectively) which have surjective linearized Cauchy--Riemann operators. By Theorem \ref{rel-smooth-representability}, both these moduli spaces carry natural rel--$C^\infty$ structures over $\fB$ and represent the corresponding moduli functors (Definition \ref{rel-smooth-moduli-functors}). The projection $p_{\bP^N}:E\to\bP^N$ induces a natural rel--$C^\infty$ map between these moduli spaces.
    Note that the points $x$ and $\tilde x$ yield points 
    \begin{align*}
        y = (s,\iota \circ i^{-1}\cl \fC_s\to \bP^N)\quad\text{and}\quad\tilde y = (s,(u,\iota,\eta) \circ i^{-1}\cl \fC_s\to E)
    \end{align*}
    in these two moduli spaces respectively. 
    
    Evaluation at the $r$ marked points (given by the sections $\sigma_1,\ldots,\sigma_r$ of $\pi_\fB$) gives rel--$C^\infty$ maps from the moduli spaces in \eqref{unconstrained-mod-spaces} to $(\bP^N)^r/\text{pt}$. Since $p_1,\ldots,p_r$ was chosen to be a \emph{minimal} collection of marked points which stabilize $C$, it follows that these evaluation maps are vertically transverse to $D:=D_1\times\cdots\times D_r\subset(\bP^N)^r$ at $y$ and $\tilde y$ (this is proved, e.g., in \cite[Lemma 6.7]{Swa21}). Thus, by taking the inverse image of $D\subset(\bP^N)^r$ under the evaluation map and using vertical transversality to $D$, we obtain rel--$C^\infty$ structures (near $y$ and $\tilde y$) on the moduli spaces
    \begin{align*}
        \fM^\text{reg}(\pi_\fB,\bP^N)_{m[\bP^1],D}\quad\text{and}\quad\fM^\text{reg}(\pi_\fB,E)_{\tilde A,D}
    \end{align*}
    obtained from \eqref{unconstrained-mod-spaces} by requiring that the $i^\text{th}$ marked point is mapped to $D_i = h_i^{-1}(0)$ for $1\le i\le r$.

    For any $[\hat C\subset\bP^N]$ in $\Mbar_g^*(\bP^N,m)$ close to $x$, we may stabilize $\hat C$ by adding the $r$ marked points obtained by intersecting it with $D_1,\ldots,D_r$. Thus, applying the universal property of $(\fC\to\fM,\sigma_1,\ldots,\sigma_r)$ near $x$, we obtain a map
    \begin{align*}
        \Mbar_g^*(\bP^N,m)\to\fB
    \end{align*}
    by intersecting with $D_1,\ldots,D_r$. We also get an isomorphism of the universal curve on $\Mbar_g^*(\bP^N,m)$ with the pullback of $\fC\to\fB$ which, at $x$, coincides with the $i:C\to\fC_s$. This yields, near $x$ and $\tilde x$, the following commutative diagram.
    \begin{center}
    \begin{tikzcd}
         \cT' \arrow[d] \arrow[r] & \fM^\text{reg}(\pi_\fB,E)_{\tilde A,D} \arrow[d] \\
         \Mbar_g^*(\bP^N,m) \arrow[r] & \fM^\text{reg}(\pi_\fB,\bP^N)_{m[\bP^1],D} 
    \end{tikzcd}    
    \end{center}
    The horizontal maps here are such that $\tilde x\mapsto\tilde y$ and $x\mapsto y$. Since there are obvious forgetful maps which are inverse to the horizontal maps (near $y$ and $\tilde y$), we conclude that the horizontal maps are local homeomorphisms near $x$ and $\tilde x$. In particular, we have obtained rel--$C^\infty$ structures on $\cT'/\fB$ and $\Mbar_g^*(\bP^N,m)/\fB$. Moreover, Assumption \eqref{functor-submersion} in Definition \ref{functor-of-points} implies that the natural rel--$C^\infty$ map $\cT'/\fB\to\Mbar_g^*(\bP^N,m)/\fB$ is a rel--$C^\infty$ submersion. Lemma \ref{fibre-product} now yields a rel--$C^\infty$ structure on $\cT'/\Mbar_g^*(\bP^N,m)$ near the points $\tilde x$ and $x$. 
    
    We are left to check\footnote{Checking this is precisely what allows us to work locally without having to separately establish compatibility of the locally constructed rel--$C^\infty$ structures.} that the the rel--$C^\infty$ structure we have obtained indeed represents the functor $\fF$ near $\tilde x$ and $x$. For this, we will apply the purely formal Lemma \ref{fibre-product} with $Y' = \cT'$, $Y = \Mbar_g^*(\bP^N,m)$ and $S = \fB$. Given a rel--$C^\infty$ manifold $Z/T$, we need to check that a map $Z/T\to Y'/S\times_{Y/S}Y/Y$ is the same as an element of $\fF(Z/T)$ provided both of them have images close to $(\tilde x,x)$. The data of a map $Z/T\to Y'/S\times_{Y/S}Y/Y$ corresponds to
    \begin{enumerate}
        \item a diagram 
        \begin{center}
        \begin{tikzcd}
            \fC \arrow[d,"\pi_\fB"] & \fC_T \arrow[d,"\pi_T"] \arrow[l] & \arrow[l] \fC_Z \arrow[r,"F"] \arrow[d,"\pi_Z"] & E \\
            \fB \arrow[u,bend left,"\sigma_i"] & T \arrow[l] & \arrow[l] Z &
        \end{tikzcd}
        \end{center}
        as in Definition \ref{rel-smooth-moduli-functors} with the additional condition that the marked points induced by $\sigma_i$ map under $F$ to $p_{\bP^N}^{-1}(D_i)\subset E$ for $1\le i\le r$ and
        \item a diagram\footnote{Here, we are using the fact that rel--$C^\infty$ maps $Z/T\to Y/Y$ are the same as rel--$C^\infty$ maps $T/T\to Y/S$, since both just amount to continuous maps $T\to Y$.}
        \begin{center}
        \begin{tikzcd}
            \fC \arrow[d,"\pi_\fB"] & \arrow[l] \fC_T  \arrow[d,"\pi_Z"] & \arrow[l,equal] \fC_T  \arrow[d,"\pi_Z"] \arrow[r,"f"] & \bP^N \\
            \fB \arrow[u,bend left,"\sigma_i"] & \arrow[l] T & \arrow[l,equal] T &
        \end{tikzcd}   
        \end{center}
        as in Definition \ref{rel-smooth-moduli-functors} with the additional condition that the marked points induced by $\sigma_i$ map under $f$ to $D_i\subset \bP^N$ for $1\le i\le r$
    \end{enumerate}
    such that the same diagram appears when we compose the first diagram with the projection $p_{\bP^N}:E\to\bP^N$ and pullback the second one from $T/T$ to $Z/T$. More succinctly, this amounts to the data of a single diagram
    \begin{center}
    \begin{tikzcd}
        & \bP^N & \arrow[l,"p_{\bP^N}"] E \\
        \fC \arrow[d,"\pi_\fB"] & \fC_T \arrow[u,"f"] \arrow[d,"\pi_T"] \arrow[l] & \arrow[l] \fC_Z \arrow[d,"\pi_Z"] \arrow[u,"F"] \\
            \fB \arrow[u,bend left,"\sigma_i"] & T \arrow[l] & \arrow[l] Z
    \end{tikzcd}
    \end{center}
    such the bottom two squares are pullbacks, the map $(F,f)\cl \fC_Z/\fC_T\to E/\bP^N$ is rel--$C^\infty$, the map $f$ (resp. $F$) is pseudo-holomorphic on each fibre of $\pi_T$ (resp. $\pi_Z$) and $f$ maps the marked points induced by $\sigma_i$ to $D_i\subset\bP^N$ for $1\le i\le r$. If all the maps $f_t$ (resp. $F_z$) here are close to $y$ (resp. $\tilde y$), then the maps to $\fC/\fB$ in this diagram can be recovered from the remaining data by intersecting $f$ with $D_1,\ldots,D_r$. Once we drop $\fC/\fB$ from the diagram, this amounts to an element of $\fF(Z/T)$, which completes the proof of representability near $x$ and $\tilde x$.
\end{proof}

\begin{definition}[Dual Hodge bundle on $\cT'$]\label{dual-hodge-on-thickening}
    Using Lemma \ref{Vector Bundle via Homology of Structural Line}, define $\bH^*$ to be the algebraic vector bundle on the smooth quasi-projective scheme $\Mbar_g^*(\bP^N,m)$ whose fibre over any $C\subset\bP^N$ is given by $H^1(C,\cO_C)$. Viewing $\bH^*$ as a continuous vector bundle and pulling it back to $\cT'$ gives a natural rel--$C^\infty$ vector bundle $\bH_{\cT'}^*$ on $\cT'/\Mbar_g^*(\bP^N,m)$ with the same description of the fibres.
\end{definition}

For the next statement, recall that Lemma \ref{rep-shear} allows us to regard the space $\cM$ from Definition \ref{unitary-holo-maps} as a subset of the rel--$C^\infty$ manifold $\cT'/\Mbar_g^*(\bP^N,m)$. The proof uses some results developed in Appendix \ref{line-bundles-on-families-of-curves}. Specifically, we need Lemma \ref{injectivity-of-exp}, which establishes a version of injectivity near the origin of the exponential map
$\exp \cl H^1(C,\cO_C)\to \Pic^0(C)$ for families of curves.

\begin{proposition}\label{alpha-smoothness-technical}
    There exists a rel--$C^\infty$ section $\alpha$ of the vector bundle $\bH_{\cT'}^*$, defined near $\cM\subset\cT'$ with the following property. 
    For any $(C\subset\bP^N,u,\eta)$ in the domain of $\alpha$, we have the identity
    \begin{align}\label{correcting-polarization-on-curve}
        [\cO_C(1)]\otimes[\fL_u^{\otimes p}]^{-1} = \exp(\alpha(C\subset\bP^N,u,\eta))\in\normalfont\text{Pic}^0(C)
    \end{align}
    and, moreover, the section $\alpha$ is identically zero on $\cM$. 
    
    Two sections $\alpha_1,\alpha_2$ with this property necessarily agree on a neighborhood of $\cM$. In other words, the germ near $\cM$ of the section $\alpha$ is uniquely determined by the above property.
\end{proposition}
\begin{proof}
    The uniqueness of the germ of $\alpha$ near $\cM$ follows from Lemma \ref{injectivity-of-exp} applied to the universal curve on $\cS = \Mbar_g^*(\bP^N,m)$ and the fact that $\alpha$ is required to vanish along $\cM$. Thus, it suffices to construct $\alpha$ locally near any point of $\cM$, since the local constructions will patch together by Lemma \ref{injectivity-of-exp}.

    Given any $q = (C\subset\bP^N,u,0)\in\cT'$ corresponding to $(C\subset\bP^N,u)\in\cM$, we need to construct the section $\alpha$ in a neighborhood of $q$. Define 
    \begin{align*}
        \cC^\circ_g\subset\cC_g\xrightarrow{\pi}\Mbar_g^*(\bP^N,m)
    \end{align*}
    to be the complement of the nodes in the fibres of the universal curve on $\Mbar^*_g(\bP^N,m)$. We claim that it suffices to find an integer $r\ge 1$ and rel--$C^\infty$ maps
    \begin{align*}
        \tau_1,\ldots,\tau_r,\tau'_1,\ldots,\tau'_r\cl \cT'/\Mbar^*_g(\bP^N,m)\to\cC_g^\circ/\Mbar^*_g(\bP^N,m),    
    \end{align*}
    defined near $q$ and covering the identity map on $\Mbar_g^*(\bP^N,m)$, with the following properties. First, for $1\le i\le r$, we have $\tau_i(q) = \tau_i'(q)$. Second, for all points $\hat q = (\hat C\subset\bP^N,\hat u,\hat\eta)$ near $q$, we have a holomorphic line bundle isomorphism
    \begin{align*}
            \cO_{\hat C}(1)\simeq\fL_u^{\otimes p}\otimes\cO_{\hat C}(D_{\hat q})    
    \end{align*}
    where the divisor $D_{\hat q}$ is defined by 
    \begin{align*}
        D_{\hat q}:=\sum_{i=1}^r\tau_i(\hat q) - \sum_{i=1}^r\tau'_i(\hat q).    
    \end{align*}
    
    The reason why it suffices to find $\tau_i,\tau_i'$ as above is that we can then define the desired section $\alpha$ by the explicit formula
    \begin{align}\label{alpha-formula}
        \alpha(\hat q):=\sum_{i=1}^r\rho(\tau_i(\hat q),\tau_i'(\hat q)),    
    \end{align}
    where $\rho$ denotes the holomorphic Abel--Jacobi type map
    \begin{align*}
        \rho = \rho_{\cC_g/\Mbar_g^*(\bP^N,m)}\cl \cC^\circ_g\times_{\Mbar^*_g(\bP^N,m)}\cC^\circ_g\to\bH^*
    \end{align*}
    constructed in Lemma \ref{abel-jacobi-construction}. We find from Lemma \ref{abel-jacobi-construction} that $\rho$ is defined near the diagonal $\Delta_{\cC^\circ_g}$, vanishes along the diagonal and has the property that it maps any $(x,x')$ in its domain to an element $\rho(x,x')$ satisfying
    \begin{align*}
        [\cO_{\hat C}(x-x')] = \exp(\rho(x,x'))\in\text{Pic}^0(\hat C),
    \end{align*}
    where $\hat C$ is the common fibre of the universal curve on which $x$ and $x'$ lie. Thus, the section $\alpha$ defined by \eqref{alpha-formula} will satisfy the condition from \eqref{correcting-polarization-on-curve}. To see why $\alpha$ vanishes at points $\hat q\in\cM$ near $q$, we argue as follows. First, we have $\alpha(q) = 0$ since $\tau_i(q) = \tau_i'(q)$ for $1\le i\le r$. Second, by Lemma \ref{injectivity-of-exp} and the holomorphic triviality of the line bundle $\cO_C(D_{\hat q})$, we must have $\alpha(\hat q) = 0$ for $\hat q\in\cM$ near $q$. Note that we get $\alpha(\hat q) = 0$ even though we might not have $\tau_i(\hat q) = \tau_i'(\hat q)$ for $1\le i\le r$.
    
    To complete the proof, we will now show how to construct an integer $r$ and maps $\tau_i,\tau_i'$ as above. We first choose an integer $\ell\gg 1$ such that the (isomorphic) holomorphic line bundles on $C$
    \begin{align*}
        L_1&:=\cO_C(\ell+1)\otimes(\omega_C^*)^{\otimes p} \\
        L_2&:=\cO_C(\ell)\otimes(u^*\cO_X(1))^{\otimes 3p}
    \end{align*}
    are very ample and have vanishing first cohomology. Let $r:=\deg(L_1) = \deg(L_2)$ and $s_1$ (resp. $s_2$) be a holomorphic section of $L_1$ (resp. $L_2$) such that the sections $s_1$ and $s_2$ have the same vanishing locus on $C$ consisting of $r$ distinct non-singular points $z_1,\ldots,z_r\in C$. Due to the very ampleness of $L_2$, we may also fix a holomorphic section $s_2'$ of $L_2$ which is non-vanishing at each of the points $z_1,\ldots,z_r\in C$.
    
    Let $\Mbar^*_g(\bP^N,m)\xleftarrow{\pi}\cC_g\xrightarrow{F}\bP^N$ be the universal map on $\Mbar^*_g(\bP^N,m)$ and let $\omega_\pi$ be the relative dualizing line bundle of $\pi$. Define the coherent sheaf 
    \begin{align*}
        \cL_1:=\pi_*(F^*\cO_{\bP^N}(\ell+1)\otimes(\omega_\pi^*)^{\otimes p})
    \end{align*}
    on $\Mbar^*_g(\bP^N,m)$. Using $H^1(C,L_1) = 0$ and the theorem on cohomology and base change \cite[Theorem III.12.11]{Har77}, we see that $\cL_1$ is locally free near the point $C\subset\bP^N$. Thus, we can find a local holomorphic section $\sigma_1$ of $\cL_1$ with 
    \begin{align*}
        \sigma_1(C\subset\bP^N) = s_1 \in H^0(C,\cO_C(\ell+1)\otimes\omega_C^{-\otimes p}).    
    \end{align*}
    Using this, for $\hat q = (\hat C\subset\bP^N,\hat u,\hat\eta)$ close to $q$, define $\tau_1(\hat q),\ldots,\tau_r(\hat q)\in\hat C$ to be the $r$ distinct zeros of $\sigma_1(\hat C\subset\bP^N)$. Note that the numbering of $\tau_1(\hat q),\ldots,\tau_r(\hat q)$ is determined by $\tau_i(\hat{q})$ being a small perturbation of $z_i\in C$ for $1\le i\le r$. The functions $\tau_1,\ldots,\tau_r$ are obviously rel--$C^\infty$, since they are pullbacks to $\cT'$ of continuous (in fact, holomorphic) sections of $\pi$.
    
    Finally, we are left to construct $\tau'_1,\ldots,\tau'_r$. Consider the rel--$C^\infty$ manifold $\cL_2/\Mbar_g^*(\bP^N,m)$ parametrizing tuples $(\hat C\subset\bP^N,\hat u,\hat\eta,\hat s_2)$ where $(\hat C\subset\bP^N,\hat u,\hat\eta)\in \cT'$ is a point close to $q$ and
    \begin{align*}
        \hat s_2\in H^0(\hat C,\cO_{\hat C}(\ell)\otimes(\hat u^*\cO_X(1))^{\otimes 3p}).
    \end{align*}
    In more detail, the rel--$C^\infty$ structure on $\cL_2/\Mbar_g^*(\bP^N,m)$ near $(C\subset\bP^N,u,0,s_2)$ is obtained by arguing exactly as in Lemma \ref{rep-shear}, with $H^1(C,L_2) = 0$ ensuring the surjectivity of the linearized operator in the $\hat s_2$-direction. Moreover, $H^1(C,L_2) = 0$ implies that the natural projection
    \begin{align}\label{L_2-proj}
        \cL_2/\Mbar_g^*(\bP^N,m)\to\cT'/\Mbar_g^*(\bP^N,m)
    \end{align}
    is a rel--$C^\infty$ submersion\footnote{With some more work using Theorem \ref{rel-smooth-obs-bundle-rep}, we can even show that $\cL_2\to\cT'$ is a rel--$C^\infty$ vector bundle near $q$.} near $(C\subset\bP^N,u,0,s_2)$. Thus, we may choose a rel--$C^\infty$ section $\sigma_2$ of \eqref{L_2-proj}, defined near $q$, with
    \begin{align*}
        \sigma_2(q) = s_2\in H^0(C,\cO_C(\ell)\otimes(u^*\cO_X(1))^{\otimes 3p}).
    \end{align*}
    Using this, for $\hat q = (\hat C\subset\bP^N,\hat u,\hat\eta)\in\cT'$ close to $q$, define $\tau_1'(\hat q),\ldots,\tau_r'(\hat q)\in\hat C$ to be the $r$ distinct zeros of $\sigma_2(\hat q)$ close to $z_1,\ldots,z_r \in C$ respectively. 
    
    We need to show that $\tau_1',\ldots,\tau_r'$ are rel--$C^\infty$ near $q$. For this, recall that $s_2'$ is a global section of $L_2$ on $C$ which is non-vanishing on $z_1,\ldots,z_r$. Choose a rel--$C^\infty$ section $\sigma_2'$ of \eqref{L_2-proj}, defined near $q$, with $\sigma_2'(q) = s_2'$. Then, for $\hat q\in\cT'$ near $q$, we can realize $\tau'_1(\hat q),\ldots,\tau_r'(\hat q)$ as the $r$ distinct zeros of the meromorphic function $\sigma_2(\hat q)/\sigma_2'(\hat q)$ close to $z_1,\ldots,z_r \in C$ respectively. Now, if a holomorphic function $\varphi(\zeta)$ of one variable is defined for $|\zeta|\le 1$, has no zeros on $|\zeta| = 1$ and has a unique zero in $|\zeta|<1$ then, using the residue theorem, this zero is given by the formula
    \begin{align*}
        \frac1{2\pi i}\oint_{|\zeta|=1}\zeta\frac{\varphi'(\zeta)}{\varphi(\zeta)}d\zeta.
    \end{align*}
    This complex analysis fact (applied near $z_1,\ldots,z_r$) and the rel--$C^\infty$ dependence of $\sigma_2(\hat q)$ and $\sigma_2'(\hat q)$ on $\hat q$ show that $\tau'_1(\hat q),\ldots,\tau'_r(\hat q)$ are also rel--$C^\infty$ functions of $\hat q$. This completes the construction of $r,\tau_i,\tau'_i$ and concludes the proof.
\end{proof}

Observe that $\fs^{-1}(0)$ is identified with $\cM$ under the projection $\cT\to\cT'$. Therefore, a $G$-invariant neighborhood $\cT^{\reg}\subset \cT$ of $\fs^{-1}(0)$ can be realized as the graph of the rel--$C^\infty$ section of $\bH_{\cT'}^*$ constructed in Proposition \ref{alpha-smoothness-technical}. In particular, $\cT^{\reg}$ carries a canonical rel--$C^\infty$ structure over $\Mbar_g^*(\bP^N,m)$. This completes the proof of Proposition \ref{gkc-properties-explicit}\eqref{thickening-is-smooth}.

\subsection{Obstruction bundle}\label{smoothness-of-obstruction-bundle}

Recall from \eqref{obstruction-bundle-fibre-def} that the obstruction bundle $\cE$ has three summands. The first summand is a trivial bundle with fibre $\fs\fu(N+1)$ and therefore has a natural rel--$C^\infty$ vector bundle structure. The third summand acquires a natural rel--$C^\infty$ structure as it is the pullback to $\cT$ of the rel--$C^\infty$ vector bundle $\bH^*_{\cT'}\to\cT'$ from Definition \ref{dual-hodge-on-thickening}. 

It remains to consider the second summand of the obstruction bundle. This summand is pulled back under the projection map $\cT\to\cT'$. Its fibre over any point $(C\subset\bP^N,u,\eta)\in\cT'$ is given by the vector space $E_{(C\subset\bP^N,u)}$ of \eqref{obstruction-space-defined}. Denoting by $\iota \cl C\hkra \bP^N$ the inclusion, we see that $$E_{(C\subset\bP^N,u)} = H^0(C,(u,\iota)^*E).$$
As the auxiliary datum fixed at the beginning of \textsection\ref{transversality-implies-smoothness-proof} is unobstructed, the second summand of $\cE$ carries a canonical rel--$C^\infty$ vector bundle structure by Theorem \ref{rel-smooth-obs-bundle-rep}.
This completes the proof of the first half of Proposition \ref{gkc-properties-explicit}\eqref{bundle-section-is-smooth}.

\subsection{Obstruction section}\label{smoothness-of-obstruction-section}

Recall from Construction \ref{high-level-description-defined} that the obstruction section $\fs$ is given by the formula
\begin{align*}
    \fs(C\subset\bP^N,u,\eta,\alpha) = (\text{i}\log\lambda_\cU(C\subset\bP^N,u),\eta,\alpha),
\end{align*}
with the  three components corresponding to the three summands of the obstruction bundle $\cE$. It is immediate from the construction of the rel--$C^\infty$ vector bundle structure on $\cE$ and the Yoneda lemma that the second component of $\fs$ is rel--$C^\infty$. Since $\cT$ was constructed as the graph of a rel--$C^\infty$ section defined on $\cT'$, the fact that the third component of $\fs$ is rel--$C^\infty$ is immediate.

We are left to prove that $\lambda_\cU$ is also rel--$C^\infty$ on $\cT$. Clearly, it is pulled back from $\cT'$, and so it suffices to prove that $\lambda_\cU$ is rel--$C^\infty$ on $\cT'$. For this, we need to show that each term appearing in the linear combination \eqref{lambda-defined} is rel--$C^\infty$ on $\cT'$. 

We first tackle the terms $(C\subset\bP^N,u,\eta)\mapsto\chi_i(C,u)$. To see that these are rel--$C^\infty$ on $\cT'$, we argue as follows. First, define $\tilde Z$ to be the polyfold of maps to $E$ in class $\tilde A$ whose projection to $X$ is an $\Omega$-stable map (similar to Definition \ref{polyfolds-defined}). There is a natural sc-smooth projection $\tilde Z\to Z$, where $Z$ is the polyfold from Definition \ref{polyfolds-defined}. The projection is given by composing a map to $E$ with the projection $E\to X$. It is sc-smooth because it is given by a linear projection onto a direct summand in local charts. This shows that the pullback of $\chi_i$ is sc-smooth on $\tilde Z$. As explained in Appendix \ref{rel-smooth-addendum}, every point in $\cT'$ has a local rel--$C^\infty$ chart given by some finite dimensional submanifold $M\subset\tilde Z$. (Note that $M$ inherits a $C^\infty$ structure from $Z$. The rel--$C^\infty$ structure on $M$ is obtained from this by disallowing differentiation in the directions corresponding to deformations of the map to $\bP^N$.) We conclude that the term $(C\subset\bP^N,u,\eta)\mapsto\chi_i(C,u)$ is rel--$C^\infty$ on $\cT'$. 

Next, consider the term $(C\subset\bP^N,u,\eta)\mapsto\text{st}_{D_i}(C\subset\bP^N,u)$, over the open subset $\cT'_i\subset\cT'$ where conditions \eqref{transverse-marking} and \eqref{avoid-boundary} of Definition \ref{good-covering} hold for $D_i$. The following statement will show that $\lambda_\cU$ is rel--$C^\infty$ on $\cT'$. Its proof uses a technical result developed in Appendix \ref{rel-smooth-addendum}.

\begin{proposition}[Intersection with divisor is rel--$C^\infty$]
    The assignment given by $$(C\subset\bP^N,u,\eta)\mapsto\normalfont\text{st}_{D_i}(C\subset\bP^N,u)$$ 
    defines a rel--$C^\infty$ map
    \begin{align*}
        \normalfont\text{st}_{D_i}\cl \cT'_i/\Mbar_g^*(\bP^N,m)&\to\Mbar^{\;*,p}_{g,[3d]}(\bP^N,m)/\Mbar_g^*(\bP^N,m).
    \end{align*}
\end{proposition}
\begin{proof}
    Consider the following diagram representing the universal family over $\cT'_i$.
    \begin{center}
    \begin{tikzcd}
        \cT_i' \arrow[d] & \arrow[l] \cC_{\cT_i'} \arrow[d] \arrow[r,"f_i"] & E \arrow[d] \\
        \Mbar_g^*(\bP^N,m) & \arrow[l] \cC_g \arrow[r] & \bP^N
    \end{tikzcd}
    \end{center}
    In particular, $f_i\cl \cC_{\cT'_i}/\cC_g\to E/\bP^N$ is a rel--$C^\infty$ map.
    Define the open subset 
    \begin{align*}
        \cC_g^\circ\subset\cC_g\to\Mbar_g^*(\bP^N,m)    
    \end{align*}
    to be the complement of the nodes in the fibres of the universal curve on $\Mbar_g^*(\bP^N,m)$. Then, the open subset $\cC^\circ_{\cT'_i}\subset\cC_{\cT'_i}$ obtained as the fibre product
    \begin{align*}
        \cC^\circ_g\times_{\Mbar_g^*(\bP^N,m)}\cT'_i
    \end{align*}
    carries a natural rel--$C^\infty$ structure over $\Mbar_g^*(\bP^N,m)$. Restrict the above diagram to $\cC^\circ_g$ and compose with the projection $p_X:E\to X$ to get the following diagram.
    \begin{center}
    \begin{tikzcd}
        \cT_i' \arrow[d] & \arrow[l] \cC^\circ_{\cT_i'} \arrow[d] \arrow[r,"f_i^\circ"] & E \arrow[r, "p_X"] & X \\
        \Mbar_g^*(\bP^N,m) & \arrow[l] \cC_g^\circ & &
    \end{tikzcd}
    \end{center}
    By Proposition \ref{extra-smoothness}, the map
    \begin{align*}
        f_i^\circ\cl \cC^\circ_{\cT'_i}/\Mbar_g^*(\bP^N,m)\to E/\text{pt}    
    \end{align*}
    is rel--$C^\infty$. Since the fibres of $\cC^\circ_{\cT'_i}\to\cT_i'$ are transverse\footnote{These fibre directions are `made available for differentiation' by Proposition \ref{extra-smoothness}.} to $D_i\subset X$ under the map $p_X\circ f_i^\circ$, the inverse function theorem with parameters (e.g., see \cite[Lemma 5.10]{Swa21}) shows that the subset 
    \begin{align*}
        \tilde\cT_i' := (p_X\circ f_i^\circ)^{-1}(D_i)\subset\cC^\circ_{\cT'_i}    
    \end{align*}
    is a rel--$C^\infty$ submanifold over $\Mbar_g^*(\bP^N,m)$. By assumption, the projection $\tilde\cT_i'\to\cT_i'$ is a degree $3d$ covering space.     
    The zig-zag of natural rel--$C^\infty$ projection maps
    \begin{align*}
        \cT_i'/\Mbar_g^*(\bP^N,m)\xleftarrow{}\tilde\cT'_i/\Mbar_g^*(\bP^N,m)\to\cC_g^\circ/\Mbar_g^*(\bP^N,m)
    \end{align*}
    now allows us to construct $\text{st}_{D_i}$ as a rel--$C^\infty$ map. Explicitly, we may locally take $3d$ sections of the covering space $\tilde\cT'_i\to\cT_i'$ (one section per sheet) and then project these along $\tilde\cT_i'\to\cC^\circ_g$. This gives a locally defined map to $\Mbar_{g,3d}^{*,p}(\bP^N,m)$, where the latter is realized as a subset of the $3d$-fold fibre product of $\cC^\circ_g$ over $\Mbar^*_g(\bP^N,m)$. Upon passing to the quotient under the free permutation action of $S_{3d}$ on the marked points, these locally defined maps patch together to give $\text{st}_{D_i}$.
\end{proof}

This completes the proof of the second half of Proposition \ref{gkc-properties-explicit}\eqref{bundle-section-is-smooth}.

\subsection{Group action and zero locus}\label{smoothness-of-group-action}

Since the rel--$C^\infty$ structure on $\cT$ is induced from $\cT'$, the $G$-action defines a rel--$C^\infty$ map
\begin{align*}
    (G\times\cT)/(G\times\Mbar^*_g(\bP^N,m))\to\cT/\Mbar^*_g(\bP^N,m)
\end{align*}
if the corresponding map with $\cT$ replaced by $\cT'$ is of class rel--$C^\infty$. But this is clear once we observe that the corresponding natural transformation at the level of the functor $\fF$ from Definition \ref{functor-of-points} is well-defined and is therefore represented by a rel--$C^\infty$ map by the Yoneda lemma. The next two statements establish finiteness of stabilizers near $\fs^{-1}(0)$, respectively local linearity in the sense of \cite[Definition 4.20]{AMS21}, of the $G$-action. The latter property says that near any $x \in \cT'$ we can find local coordinates of $\cT'$ in which the stabilizer $G_x$ acts linearly. 

\begin{lemma}[Finite stabilizers]\label{finite-stabilizers-lemma}
    In a sufficiently small neighborhood of $\fs^{-1}(0)$ in $\cT$, the stabilizer of any point is a finite subgroup of $G$.
\end{lemma}
\begin{proof}
    Let $\Gamma\subset G$ be the stabilizer of a point $(C\subset\bP^N,u,0,0)$ lying in $\fs^{-1}(0)$. The equivalence of groupoids established in Discussion \ref{groupoid-equivalence} shows that $\Gamma$ is naturally identified with the automorphism group of the stable map $(C,u)$. Thus, $\Gamma$ is finite. This implies that the action of $G$ on $\cT$ also has finite stabilizers in a neighborhood of the subset $\fs^{-1}(0)\subset\cT$.
\end{proof}

To see local linearity, let $\tilde{x}\in \cT$ be a point close to $\fs\inv(0)$ with image $x$ in $\Mbar_g^*(\bP^N,m)$. Due to smoothness of the action on $\Mbar_g^*(\bP^N,m)$, we may pick a coordinate chart $V\sub \Mbar_g^*(\bP^N,m)$ centred at $x$ on which $G_{\tilde x}$ acts linearly. Letting $U$ be any $G_{\tilde{x}}$-invariant neighbourhood with $\pi(U) \subset V$, it remains to show the following lemma.

\begin{lemma}[Locally linearizing an action]\label{local-linear-action}
    Let $V$ be a finite dimensional vector space, $\pi:M\to V$ a rel--$C^\infty$ manifold and $\Gamma$ a finite group. Assume that we are given a rel--$C^\infty$ action of $\Gamma$ on $M/V$ which covers a linear $\Gamma$-representation $\theta$ on $V$. Let $x\in\pi^{-1}(0)$ be fixed by $\Gamma$. Then, $M/V$ has a rel--$C^\infty$ chart at $x$ in which the $\Gamma$-action is linear.
\end{lemma}
\begin{proof}
    By shrinking to a $\Gamma$-invariant coordinate neighborhood of $x\in M$, we may assume that $M$ is an open subset of a product $V\times W$, with $W$ another finite dimensional vector space with $x=(0,0)$. The action of any $g\in\Gamma$ is given, in these coordinates, by a rel--$C^\infty$ map $(v,w)\mapsto(\theta(g)v,\varphi_g(v,w))$ with
    \begin{align*}
        \varphi_g(\theta(h)v,\varphi_h(v,w)) \equiv \varphi_{gh}(v,w).
    \end{align*}
    Define the representation $\rho\cl \Gamma\to \text{GL}(W)$ by $\rho(g) = \frac{\partial\varphi_g}{\partial w}(0,0)$. Define the map $(v,w)\mapsto(v,T(v,w))$ by
    \begin{align*}
       T(v,w) = \frac1{|\Gamma|}\sum_{g\in\Gamma}\rho(g)^{-1}\varphi_g(v,w)
    \end{align*}
    and observe that $\frac{\partial T}{\partial w}(0,0)$ is the identity map. Using the implicit function theorem with parameters \cite[Lemma 5.10]{Swa21}, it follows that $(v,w)\mapsto(v,T(v,w))$ is a rel--$C^\infty$ coordinate change on $M/V$ near $x$. A direct check shows that in these new coordinates the $\Gamma$-action is given by $\theta\oplus\rho$.
\end{proof}

This completes the proof of Proposition \ref{gkc-properties-explicit}\eqref{action-is-smooth}.

\begin{lemma}[Zeros of the obstruction section]\label{zero-locus-homeo-lemma}
    The zero locus $\fs\inv(0)$ is compact and the natural map 
    \begin{align*}
        \fs^{-1}(0)/G\to\Mbar_g(X,A;J)   
    \end{align*}
    is a homeomorphism.
\end{lemma}
\begin{proof}
    Recall that $\fs^{-1}(0)$ is identified with the space $\cM$ (Definition \ref{unitary-holo-maps}) and thus, it is compact by Lemma \ref{unitary-holo-maps-compact}. The forgetful map $\fs^{-1}(0)\to\Mbar_g(X,A;J)$ is evidently continuous and therefore descends to the $G$-quotient as a continuous map.
    
    Bijectivity of the natural map $\fs^{-1}(0)/G\to\Mbar_g(X,A;J)$ follows by passing to the set of isomorphism classes in the groupoid equivalence of Discussion \ref{groupoid-equivalence}. Since $\fs^{-1}(0)/G$ is compact and $\Mbar_g(X,A;J)$ is known to be Hausdorff, this continuous bijection must be a homeomorphism.
\end{proof}

This completes the proof of Proposition \ref{gkc-properties-explicit}\eqref{footprint}.

\subsection{Stable almost complex structure}
    To establish the assertion of Proposition \ref{gkc-properties-explicit}\eqref{orientation-complex-structure}, we begin by observing that the bundle $\cE\oplus\underline{\fg}$ is already a complex vector bundle. Indeed, the second and third summand in \eqref{obstruction-bundle-fibre-def} are complex vector spaces and we have an obvious identification $\fg\oplus\fs\fu(N+1) = \fs\fl(N+1,\bC)$.
    
    It remains to find a stable almost complex structure on the vertical tangent bundle of ${\cT/\Mbar^*_g(\bP^N,m)}$. Note that the vertical tangent bundle is given by
    \begin{align*}
        (\iota:C\hookrightarrow\bP^N,u,\eta,\alpha)\mapsto\ker\left(\ker D(\delbar_{\tilde J})_{(\iota,u,\eta)}\twoheadrightarrow\ker D(\delbar_{J^\bP})_\iota\right).
    \end{align*}
    Since $\ker D(\delbar_{J^\bP})_\iota$ is a $\bC$-vector space, we focus on the vector bundle
    \begin{align}\label{vertical-tangent-over-artin}
        (\iota:C\hookrightarrow\bP^N,u,\eta,\alpha)\mapsto\ker D(\delbar_{\tilde J})_{(\iota,u,\eta)}
    \end{align}
    instead. 
    We are left to show how to produce a stable complex lift of the topological vector bundle defined by \eqref{vertical-tangent-over-artin}. Imitating the idea of \cite[Theorem 3.1.6(i)]{MS12}, we accomplish this by canonically decomposing the linearization
    \begin{align*}
        D(\delbar_{\tilde J})_{(\iota,u,\eta)} = D(\delbar_{\tilde J})_{(\iota,u,\eta)}^\bC + A_{(\iota,u,\eta)}
    \end{align*}
    into a $\bC$-linear Cauchy--Riemann operator and a $\bC$-antilinear zeroth order operator and then using the straight line homotopy
    \begin{align}\label{homotopy-CR-op}
        L_{(\iota,u,\eta)}(t) := D(\delbar_{\tilde J})_{(\iota,u,\eta)}^\bC + tA_{(\iota,u,\eta)},\quad t\in[0,1].
    \end{align}
    
    Arguing as in Lemma \ref{k-existence}, we produce an integer $l\ge k$ such that the $\bC$-vector space $E_{(C\subset\bP^N,u),l}$ from Definition \ref{l-transverse} maps surjectively onto the cokernel of \eqref{homotopy-CR-op} for each $(\iota,u)$ in $\cM$, $\eta = 0$ and $t\in[0,1]$.\footnote{For $t\ne 1$, the operator $L_{(\iota,u,\eta)}(t)$ is not the linearized Cauchy-Riemann operator of a pseudo-holomorphic map in any obvious way. The analogue of Lemma \ref{open-l-transverse} (which gives openness of surjectivity) therefore needs to be shown directly, using a (linear) gluing argument as in \cite[Appendix B]{P16}.}

    For $t\in[0,1]$ and $(\iota:C\hookrightarrow\bP^N,u,\eta,\alpha)$ in a neighborhood of $\fs^{-1}(0)\subset\cT$, we have the surjective operator
    \begin{align*}
        L_{(\iota,u,\eta)}(t;l)&:\Omega^0(C,u^*T_X)\oplus E_{(C\subset\bP^N,u),l}\to\Omega^{0,1}(\tilde C,\tilde u^*T_X)\\
        (\xi,\eta')&\mapsto L_{(\iota,u,\eta)}(t)\,\xi + \langle\eta'\rangle\circ d\iota_{\tilde C}.
    \end{align*}
    We claim that
    \begin{align}\label{stabilized-1-param-kernels}
        (\iota:C\hookrightarrow\bP^N,u,\eta,\alpha,t)\mapsto \ker L_{(\iota,u,\eta)}(t;l)
    \end{align}
    defines a topological vector bundle on a neighborhood of the compact subset $\fs^{-1}(0)\times[0,1]\subset\cT\times[0,1]$. To see this, we can construct continuous local trivializations using linear gluing analysis near nodal curves, e.g., following \cite[Appendix B]{P16}, along with the fact that $E_{(C\subset\bP^N,u),l}$ kills the cokernel of \eqref{homotopy-CR-op}.\footnote{There is no need to check that two such continuous local trivializations are compatible in any stronger sense since we are only constructing topological vector bundles.} 
    For each $t_0\in[0,1]$, we have a projection from $\eqref{stabilized-1-param-kernels}|_{t=t_0}$ to the $\bC$-vector bundle
    \begin{align}\label{stab-for-acs}
        (\iota:C\hookrightarrow\bP^N,u,\eta,\alpha)\mapsto E_{(C\subset\bP^N,u),l},
    \end{align}
    given by $(\xi,\eta')\mapsto\eta'$. Since $L_{(\iota,u,\eta)}(1) = D(\delbar_{\tilde J})_{(\iota,u,\eta)}$ is surjective near $\fs^{-1}(0)\subset\cT$, this produces a short exact sequence of vector bundles
    \begin{align*}
        0\to\eqref{vertical-tangent-over-artin}\to\eqref{stabilized-1-param-kernels}|_{t=1}\to\eqref{stab-for-acs}\to 0.
    \end{align*}
    Next, observe that the fibres of the $\bR$-vector bundle
    \begin{align*}
        \ind^{\bC} := \eqref{stabilized-1-param-kernels}|_{t=0} 
    \end{align*}
    carry natural $\bC$-vector space structures, since the operator $L_{(\iota,u,\eta)}(0;l)$ is $\bC$-linear. Taken together with the above short exact sequence, the next lemma implies that \eqref{vertical-tangent-over-artin} is stably equivalent to a $\bC$-vector bundle.

    \begin{lemma}
        The fibrewise $\bC$-linear structure on $\ind^\bC$ is continuous. Thus, $\ind^\bC$ carries the structure of a $\bC$-vector bundle.
    \end{lemma}
    \begin{proof}
        The assertion is local. Thus, it is enough to prove it near a given point 
        \begin{align*}
            (\iota:C\hookrightarrow\bP^N,u,\eta,\alpha)\in\cT    
        \end{align*}
        in the open neighborhood of $\fs^{-1}(0)\subset\cT$ over which the vector bundle $\ind^\bC$ is defined. By assumption, the $\bC$-linear operator
        \begin{align*}
            L_{(\iota,u,\eta)}(0;l)&:\Omega^0(C,u^*T_X)\oplus E_{(C\subset\bP^N,u),l}\to\Omega^{0,1}(\tilde C,\tilde u^*T_X)\\
            (\xi,\eta')&\mapsto D(\delbar_{\tilde J})_{(\iota,u,\eta)}^\bC\,\xi + \langle\eta'\rangle\circ d\iota_{\tilde C}
        \end{align*}
        is surjective. Fix a finite collection of non-singular points $p_1,\ldots,p_r\in C$ and, for each $1\le i\le r$, a $\bC$-linear subspace
        \begin{align}\label{eqn:ev-subspace-choice}
            W_i\subset V_i:=(T_X)_{u(p_i)}\oplus (E_l)_{(u(p_i),\iota(p_i))},
        \end{align}
        where $E_l$ is the $\bC$-vector bundle from Definition \ref{E_l-acs},
        such that the map
        \begin{align}\label{eqn:ev-subspace-choice-works}
            \eval:\ker L_{(\iota,u,\eta)}(0;l)\to\bigoplus_{1\le i\le r} V_i/W_i
        \end{align}
        induced by the evaluation map $(\xi,\eta')\mapsto(\xi(p_i),\eta'(p_i))_{1\le i\le r}$ is a $\bC$-linear isomorphism. For further details on why such choices of $p_i$ and $W_i$ exist, see \cite[Equation (B.7.3)]{P16} or \cite[Lemma 4.10]{Swa21}.

        Now, choose local holomorphic sections $\tilde p_1,\ldots,\tilde p_r$ of the universal family over $\Mbar_g^*(\bP^N,m)$, which are defined near the point $[\iota:C\hookrightarrow\bP^N]$ and coincide with $p_1,\ldots,p_r$ at this point. Extend the inclusions \eqref{eqn:ev-subspace-choice} to $\bC$-linear $C^\infty$ sub-bundles
        \begin{align*}
            \tilde W_i\subset \tilde V_i:=(T_X\boxtimes\cO_{\bP^N})\oplus E_l
        \end{align*}
        on a neighborhood of $(u(p_i),\iota(p_i))\in X\times\bP^N$, for $1\le i\le r$. 
        
        For $1\le i\le r$, evaluation at the image of the section $\tilde p_i$ induces a fibrewise $\bC$-linear map $\widetilde\eval_i$ from $\ind^\bC$ to the pullback of $\tilde V_i/\tilde W_i$. Linear gluing analysis shows that $\widetilde\eval_i$ is a continuous map of vector bundles. The isomorphism \eqref{eqn:ev-subspace-choice-works} now shows that the map $\widetilde\eval := \bigoplus_{1\le i\le r}\widetilde\eval_i$
        identifies $\ind^\bC$ with a $\bC$-vector bundle near $(\iota,u,\eta,\alpha)\in\cT$ in a continuous and fibrewise $\bC$-linear fashion.
    \end{proof}
    
    This completes the proof of Proposition \ref{gkc-properties-explicit}\eqref{orientation-complex-structure} and therefore the proof of Theorem \ref{global-kuranishi-existence}\eqref{transversality-implies-smoothness}.

\section{Uniqueness up to equivalence and cobordism}\label{proof-of-equivalence-and-cobordism}

In this section, we will prove Theorem \ref{global-kuranishi-existence}\eqref{global-kuranishi-uniqueness}. The case with $n$ marked points is a consequence of the case with no marked points (as in Lemma \ref{marked-points-formal-consequence}) and so, we shall focus only on the latter.

\subsection{Equivalence}\label{equivalence-of-charts}

Fix $J\in\cJ_\tau(X,\omega)$ and let
\begin{align*}
    (\nabla^{X,i},\cO_{X,i}(1),p_i,\cU_i,\lambda_i,k_i)    
\end{align*}
for $i = 0,1$ be any two choices of unobstructed auxiliary data for $\Mbar_g(X,A;J)$ and let $\cK_i = (G_i,\cT_i/\cB_i,\cE_i,\fs_i)$ for $i=0,1$ be the associated rel--$C^\infty$ global Kuranishi charts from Construction \ref{high-level-description-defined}. Consider the doubly thickened rel--$C^\infty$ global Kuranishi chart 
\begin{align*}
    \cK = (G_0\times G_1,\cT_{01}/\cB_{01},\cE_{01},\fs_{01})
\end{align*}
which is defined as follows.
\begin{enumerate}[(i)]
    \item $\cB_{01}$ is the space of embedded prestable genus $g$ curves $C\subset\bP^{N_0}\times\bP^{N_1}$ such that applying the coordinate projection $\bP^{N_0}\times\bP^{N_1}\to\bP^{N_i}$ yields a point $\iota_i:C\hookrightarrow\bP^{N_i}$ of $\cB_i$ for $i=0,1$. As in Lemma \ref{alg-base-spaces-are-smooth-qproj}, $\cB_{01}$ is a smooth quasi-projective variety of the expected dimension and the natural forgetful maps $\cB_{01}\to\cB_i$ are algebraic submersions.
    \item $\cT_{01}$ consists of tuples $(C\subset\bP^{N_0}\times\bP^{N_1},u:C\to X,\eta_0,\alpha_0,\eta_1,\alpha_1)$ satisfying the following properties.
    \begin{enumerate}[(a)]
        \item $C\subset\bP^{N_0}\times\bP^{N_1}$ lies in $\cB_{01}$.
        \item For $i=0,1$, $(\iota_i:C\hookrightarrow\bP^{N_i},u:C\to X)$ is a framed stable map lying in the domain of $\lambda_{\cU_i}$.
        \item For $i=0,1$, $\eta_i$ belongs to the finite dimensional vector space 
        \begin{align*}
            E^i_{(C\hookrightarrow\bP^{N_i},u)} := H^0(C,u^*T_X\otimes\iota_i^*({T^*}^{0,1}_{\bP^{N_i}}\otimes\cO_{\bP^{N_i}}(k_i)))\otimes\overline{H^0(\bP^{N_i},\cO_{\bP^{N_i}}(k_i))},    
        \end{align*}
        defined as in \eqref{obstruction-space-defined}, and we have
        \begin{align*}
            \delbar_J\tilde u + \langle\eta_0\rangle\circ d\tilde\iota_{0} + \langle\eta_1\rangle\circ d\tilde\iota_{1} = 0\in\Omega^{0,1}(\tilde C,\tilde u^*T_X)
        \end{align*}
        with $\tilde\iota_i$ being the pullback of $\iota_i$ to the normalization $\tilde C\to C$.
        \item For $i=0,1$, $\alpha_i\in H^1(C,\cO_C)$ satisfies the analogue of \eqref{fix-polarization-alpha} for $i=0,1$.
    \end{enumerate}
    \item The fibre of $\cE_{01}$ over a point $(C\subset\bP^{N_0}\times\bP^{N_1},u:C\to X,\eta_0,\alpha_0,\eta_1,\alpha_1)$ in $\cT_{01}$ is given by
    \begin{align}\label{double-obstruction}
        \bigoplus_{i=0,1}\left(\fs\fu(N_i+1)\oplus E^i_{(C\hookrightarrow\bP^{N_i},u)} \oplus H^1(C,\cO_C)\right)
    \end{align}
    and $\fs_{01} = \fs_0\oplus\fs_1$ at this point is given by $\fs_i = (\text{i}\log\lambda_{\cU_i}(C\hookrightarrow\bP^{N_i},u),\eta_i,\alpha_i)$ for $i=0,1$. By abuse of notation, we denote the vector bundle summands for $i=0,1$ from \eqref{double-obstruction} by $\cE_i$.
    \item The group $G_0\times G_1$ acts on $\cT_{01}\to\cB_{01}$ and $\cE_{01}$ in the evident way.
\end{enumerate}

Arguing as in \textsection\ref{transversality-implies-smoothness-proof}, we conclude that $\cK$ is a stably complex rel--$C^\infty$ global Kuranishi chart for $\Mbar_g(X,A;J)$ and that $\cT_{01},\cE_{01}$ and $\fs_{01}$ are actually rel--$C^\infty$ with base $\cB_0$ or $\cB_1$ (and therefore also with base $\cB_{01}$). By symmetry, it suffices to show that $\cK_0$ and $\cK$ are stably complex rel--$C^\infty$ equivalent. To see this, we first apply (Base modification) to $\cK$ and the submersion $\cB_{01}\to\cB_0$ to obtain 
\begin{align*}
    \cK'_0 := (G_0\times G_1,\cT_{01}/\cB_0,\cE_{01},\fs_{01}).
\end{align*}

The next observation is the key to showing that $\cK'_0$ and $\cK_0$ can be related by the moves (Germ equivalence), (Stabilization) and (Group enlargement). The explicit use of (Germ equivalence) will be hidden as we will always work in a sufficiently small $(G_0\times G_1)$-invariant neighborhood of $\fs_{01}^{-1}(0)$.

\begin{lemma}\label{double-thickening-transversality}
    At any $\hat x = (C\subset\bP^{N_0}\times\bP^{N_1},u:C\to X,0,0,0,0)\in\fs_{01}^{-1}(0)$, the vertical linearization
    \begin{align*}
        d\fs_1|_{\hat x}:T_{\cT_{01}/\cB_0}|_{\hat x}\to\cE_1|_{\hat x}
    \end{align*}
    is surjective.
\end{lemma}
\begin{proof}
    Let us write 
    \begin{align*}
        d\fs_1|_{\hat x} = L_1\oplus L_2\oplus L_3    
    \end{align*}
    using the direct sum decomposition of $\cE_1|_{\hat x}$ from \eqref{double-obstruction}. Since the given auxiliary datum $(\nabla^{X,0},\cO_{X,0}(1),p_0,\cU_0,\lambda_0,k_0)$ is unobstructed, it follows that the restriction 
    \begin{align*}
        T_{\cT_{01}/\cB_{01}}|_{\hat x}\to E^1_{(C\hookrightarrow\bP^{N_1},u)}
    \end{align*}
    of $L_2$ to the subspace $T_{\cT_{01}/\cB_{01}}|_{\hat x}\subset T_{\cT_{01}/\cB_0}|_{\hat x}$ is surjective. 
    
    It therefore suffices to show that $(L_1\oplus L_3)|_{\ker L_2}$ is surjective. Restricting the projection $T_{\cT_{01}/\cB_0}|_{\hat x}\to T_{\cB_{01}/\cB_0}|_{\hat{x}}$ to $\ker L_2$, we obtain a map
    \begin{align}\label{projection-to-domains}
        \ker L_2\to T_{\cB_{01}/\cB_0}|_{\hat{x}} = H^0(C,\iota_1^*T_{\bP^{N_1}}).
    \end{align}
    The map \eqref{projection-to-domains} is surjective and has a natural splitting
    \begin{align*}
        \sigma:H^0(C,\iota_1^*T_{\bP^{N_1}})\to\ker L_2    
    \end{align*}
    described as follows. Given any element of $H^0(C,\iota_1^*T_{\bP^{N_1}})$, corresponding to an infinitesimal holomorphic deformation of the map $\iota_1:C\hookrightarrow\bP^{N_1}$, the map $\sigma$ sends it to the infinitesimal deformation of $\hat x$ which keeps $(\iota_0,u)$ constant, keeps $\eta_0 = 0$, $\alpha_0 = 0$ and $\eta_1 = 0$ and follows the given infinitesimal deformation of $\iota_1:C\hookrightarrow\bP^{N_1}$. Observe that the infinitesimal deformation of $\alpha_1$ is automatically determined by the deformation of $\iota_1$ and the fact that $u$ remains fixed. Pulling back $(L_1\oplus L_3)|_{\ker L_2}$ along $\sigma$ yields a map
    \begin{align}\label{variation-of-line-bundle}
        \hat L_1\oplus\hat L_3:T_{\cB_{01}/\cB_0}|_{\hat x} = H^0(C,\iota_1^*T_{\bP^{N_1}})\to \fs\fu(N_1+1)\oplus H^1(C,\cO_C).
    \end{align}
    
    It will suffice to show that \eqref{variation-of-line-bundle} is surjective.
    The map $\hat L_3$ in \eqref{variation-of-line-bundle} can be identified with the connecting map of the long exact sequence in cohomology obtained by pulling back the Euler exact sequence
    \begin{align*}
        0\to\cO_{\bP^{N_1}}\to\cO_{\bP^{N_1}}(1)^{N_1+1}\to T_{\bP^{N_1}}\to 0    
    \end{align*}
    to $C$ along $\iota_1$. This shows that $\hat L_3$ is surjective and that $\ker\hat L_3$ is identified with $\fs\fl(N_1+1,\bC)$, corresponding to the infinitesimal action of $\cG_1:=\PGL(N_1+1,\bC) = \normalfont\text{PSL}(N_1+1,\bC)$ on $T_{\cB_{01}/\cB_0}|_{\hat x}$. It remains to show that the restriction
    \begin{align}\label{lie-alg-proj}
        \fs\fl(N_1+1,\bC)\to\fs\fu(N_1+1)
    \end{align}
    of $\hat L_1$ to the subspace $\ker L_3\subset T_{\cB_{01}/\cB_0}|_{\hat x}$ is surjective. The $\cG_1$-equivariance of the map $\lambda_{\cU_1}$ ensures that \eqref{lie-alg-proj} is simply the linearization of the projection map $\cG_1\to\cG_1/G_1$ at the identity element of $\cG_1$. Thus, \eqref{lie-alg-proj} is surjective. 
\end{proof}

From Lemma \ref{double-thickening-transversality}, it is immediate that $\cK_0'' := (G_0\times G_1,\fs_1^{-1}(0)/\cB_0,\cE_0,\fs_0)$ is a rel--$C^\infty$ global Kuranishi chart related to $\cK_0$ by (Group enlargement). To conclude, we need to show that $\cK_0''$ and $\cK_0'$ are related by (Stabilization) with the role of $W$ in Definition \ref{equivalence-of-charts-defined} played by $\cE_1$. This readily follows from Lemma \ref{double-thickening-transversality} and a rel--$C^\infty$ tubular neighborhood argument.
\\\\
\indent This completes the proof of Theorem \ref{global-kuranishi-existence}\eqref{chart-uniqueness-up-to-equivalence}.

\subsection{Cobordism}

Fix $J_0,J_1\in\cJ_\tau(X,\omega)$. Connect them by a smooth path $\gamma:[0,1]\to\cJ_\tau(X,\omega)$ and write $J_t:=\gamma(t)$. Choose a smooth family of $J_t$-linear connections $\nabla^{X,t}$ on $T_X$. By Lemma \ref{Suitable Line Bundle on Target}, we can find a polarization $\cO_X(1)$ with associated symplectic form $\Omega$ taming the image of $\gamma$. Choose $p$ as in \textsection\ref{p-choice}, depending only on $g$ and $d = \langle[\Omega],A\rangle$. Using the compactness of the parametrized moduli space $\Mbar_g(X,A;\gamma)$, we can now choose $(\cU,\lambda)$ and $k$ as in \textsection\ref{lambda-choice} and \textsection\ref{k-choice} such that, for each $t\in[0,1]$, the auxiliary datum $(\nabla^{X,t},\cO_X(1),p,\cU,\lambda,k)$ is unobstructed for $J_t$. This yields a stably complex rel--$C^\infty$ global Kuranishi chart $\cK_t$ for $\Mbar_g(X,A;J_t)$ for each $t\in[0,1]$. Repeating the arguments of \textsection\ref{transversality-implies-smoothness-proof} with the parameter $t$, we see that the family $\{\cK_t\}_{t\in[0,1]}$ fits together to exhibit a stably complex rel--$C^\infty$ cobordism between $\cK_0$ and $\cK_1$.
\\\\
\indent This completes the proof of Theorem \ref{global-kuranishi-existence}\eqref{chart-uniqueness-up-to-cobordism}.

\section{Product formula for Gromov--Witten invariants}\label{products}

In this section, we prove the product formula for Gromov--Witten invariants as an application of the global Kuranishi chart construction.
Recall from \eqref{gw-class-de} that, for a closed symplectic manifold $(X,\omega)$, a class $A\in H_2(X,\bZ)$ and  integers $g,n\ge 0$, the associated \emph{Gromov--Witten class} is defined to be
\begin{align*}
    \text{GW}^{(X,\omega)}_{A,g,n}:=(\text{ev}\times\text{st})_*[\Mbar_{g,n}(X,A;J)]^\text{vir}\in H_{\vdim}(X^n\times\Mbar_{g,n};\bQ)
\end{align*}
where $J\in\cJ_\tau(X,\omega)$ was arbitrary and we choose one of the equivalent global Kuranishi charts provided by Theorem \ref{global-kuranishi-existence}. Here $\eva$ evaluates a stable map at its marked points, while $\normalfont\text{st}$ stabilizes the domain. If $2g-2 + n\leq 0$, we take $\Mbar_{g,n}$ to be a point. Pairing the Gromov--Witten class with cohomology classes $\alpha_1,\ldots,\alpha_n\in H^*(X;\bQ)$ and $\beta\in H^*(\Mbar_{g,n};\bQ)$ yields the \emph{Gromov--Witten invariants} of $X$.

\subsection{Gromov--Witten invariants of a product} Let $(X_i,\omega_i)$ be closed symplectic manifolds for $i = 0,1$ and set $(X,\omega) := (X_0,\omega_0)\times (X_1,\omega_1)$. Given homology classes $A_i \in H_2(X_i,\bZ)$ for $i = 0,1$, let $\fA\subset H_2(X,\bZ)$ denote the set of classes $A$ which project to $A_i$ under the coordinate projection $\pr_i:X\to X_i$ for $i=0,1$. Fix $g,n\ge 0$ and $J_i\in\cJ_\tau(X_i,\omega_i)$ and set $J := J_0\times J_1$. Define the moduli space
\begin{align*}
    \Mbar_{g,n}(X,\fA;J) = \coprod_{A\in\fA}\Mbar_{g,n}(X,A;J).
\end{align*}
All components of this disjoint union have the same virtual dimension. Furthermore, as $\lspan{[\omega],A}$ is independent of $A\in \fA$, Gromov compactness implies that only finitely many components of $\Mbar_{g,n}(X,\fA;J)$ are nonempty.\par 
The projections $\pr_i$ induce a natural map
\begin{align}\label{forget-from-product}
    \Phi:\Mbar_{g,n}(X,\fA;J)\to\Mbar_{g,n}(X_0,A_0;J_0)\times\Mbar_{g,n}(X_1,A_1;J_1).
\end{align}

\begin{theorem}[Product formula]\label{gw-product}
Whenever $2g-2+n>0$ and $(g,n)$ is neither $(1,1)$ nor $(2,0)$, we have
    \begin{align*}
       \notag &\Phi_*\vfc{\Mbar_{g,n}(X,\fA;J)} \\ &\quad = (\normalfont\text{st}\times\text{st})^*\pcd([\Delta])\cap(\vfc{\Mbar_{g,n}(X_0,A_0;J_0)}\times\vfc{\Mbar_{g,n}(X_1,A_1;J_1)}).
    \end{align*}
    Here $\Delta:\Mbar_{g,n}\to \Mbar_{g,n}\times \Mbar_{g,n}$ is the diagonal map and $\pcd([\Delta])$ is the Poincar\'e dual of its image.
\end{theorem}

Theorem \ref{gw-product-formula}, stated in the introduction, is obtained from Theorem \ref{gw-product} as an application of Poincar\'e duality with $\bQ$-coefficients for the orbifold $\Mbar_{g,n}$ and the fact that $\Delta^*(\gamma_0\times\gamma_1) = \gamma_0\smile\gamma_1$ for cohomology classes $\gamma_i\in H^*(\Mbar_{g,n};\bQ)$.

Specializing to the case $g=0$ and $n=3$, we obtain the following consequence.

\begin{corollary}[K\"unneth formula for quantum cohomology]\label{quantum-kunneth} The canonical K\"unneth map
\begin{align*}
    QH^*(X_0,\omega_0)\otimes_{\Lambda_0}QH^*(X_1,\omega_1)\to QH^*(X,\omega)
\end{align*}
is an isomorphism of $\Lambda_0$-algebras.
\end{corollary}

\subsection{Proof of the product formula}  Use Theorem \ref{global-kuranishi-existence} and polarizations $\cO_{X_i}(1)$ on $X_i$ taming $J_i$, for $i=0,1$, to choose a single unobstructed auxiliary datum $(\nabla^X,\cO_X(1),p,\cU,\lambda,k)$ with $\cO_X(1) = \cO_{X_0}(1)\boxtimes\cO_{X_1}(1)$. We obtain a global Kuranishi chart 
\begin{align}\label{product-independent-chart}
    \cK = (G,\cT/\cB,\cE,\fs)    
\end{align}
for the whole of $\Mbar_{g,n}(X,\fA;J)$ by taking the disjoint union over $A\in\fA$ of the resulting charts $\cK_A$ for $\Mbar_{g,n}(X,A;J)$.\par
On the other hand, Theorem \ref{global-kuranishi-existence}\eqref{achieveing-transversality} yields unobstructed auxiliary data
\begin{align*}
    (\nabla^{X_i},\cO_{X_i}(1),p_i,\cU_i,\lambda_i,k_i)    
\end{align*}
for the moduli spaces $\Mbar_{g,n}(X_i,A_i;J_i)$. Let $\cK_i = (G_i,\cT_i/\cB_i,\cE_i,\fs_i)$ be the associated global Kuranishi charts provided by Theorem \ref{global-kuranishi-existence}\eqref{transversality-implies-smoothness} for $i = 0,1$. Recall that $\cB_i = \Mbar_{g,n}^*(\bP^{N_i},m_i)$. Define $\cN$ to be the inverse image of $\cB_0\times\cB_1$ under the natural morphism
\begin{align*}
    \Mbar_{g,n}(\bP^{N_0}\times\bP^{N_1},(m_0,m_1)) \to\Mbar_{g,n}(\bP^{N_0},m_0)\times \Mbar_{g,n}(\bP^{N_1},m_1).
\end{align*}
Denote the morphism induced by the restriction to $\cN$ as 
\begin{align*}
    \Psi:\cN\to\cB_0\times\cB_1.
\end{align*}
Note that $\Psi$ naturally factors through the fibre product $\cB_0\times_{\Mbar_{g,n}}\cB_1$. Note that this is fibre product is in the sense of orbifolds. That is, a point of $\cB_0\times_{\Mbar_{g,n}}\cB_1$ consists of points $(C_i,x^i_1,\ldots,x^i_n,u_i)$ of $\cB_i$ for $i=0,1$, along with an isomorphism $\varphi:(C_0,x^0_1,\ldots,x^0_n)^\text{st}\to(C_1,x^1_1,\ldots,x^1_n)^\text{st}$ between their stabilizations.

\begin{lemma}\label{Regularity of Product Maps}
$\cN$ is a complex manifold of the expected dimension. Moreover, when we have $2g-2+n>0$, the natural morphism
\begin{align}\label{proper-birational-map}
    \cN\to\cB_0\times_{\Mbar_{g,n}}\cB_1    
\end{align}
is proper and birational, i.e., an isomorphism after excising closed complex analytic subsets, which are nowhere dense, from $\cN$ and $\cB_0\times_{\Mbar_{g,n}}\cB_1$.
\end{lemma}

\begin{proof} 
Consider any stable map $(C,x_1,\ldots,x_n,u:C\to\bP^{N_0}\times\bP^{N_1})$ in $\cN$. Let $(C_i,x^i_1,\ldots,x_n^i,u_i:C_i\to\bP^{N_i})$ be the corresponding point of $\cB_i$, for $i = 0,1$, with $\kappa_i:C\to C_i$ being the natural morphism. Any component of $C$ which is contracted by both $\kappa_0$ and $\kappa_1$ must be a sphere with at least three special points. Since $\text{Aut}(C_i,x^i_1,\ldots,x_n^i,u_i)$ is trivial for $i=0,1$, it follows that $\text{Aut}(C,x_1,\ldots,x_n,u)$ is also trivial. Finally, we observe that 
\begin{align*}
    H^1(C,u^*T_{\bP^{N_i}}) = H^1(C_i,u_i^*T_{\bP^{N_i}}) = 0    
\end{align*}
for $i=0,1$, where the first equality comes from the fact that $C_i$ is obtained from $C$ by sequentially contracting spheres with $\le 2$ special points on which the map $u$ is constant. Thus, $\cN$ is a complex manifold of the expected dimension.

Let us now assume that $2g-2+n>0$ and show that \eqref{proper-birational-map} is proper and birational. Properness is clear since the source and the target of \eqref{proper-birational-map} are proper over $\cB_0\times\cB_1$. For birationality, we excise out the closed analytic subsets corresponding to nodal curves and show that the map on the complement is an isomorphism. 
The nodes of any $(C,x_1,\ldots,x_n,u)$ in $\cN$ can be smoothed since we have $H^1(C,u^*T_{\bP^{N_0}\times\bP^{N_1}}) = 0$ and thus, nodal curves are nowhere dense in $\cN$. A point of $\cB_0\times_{\Mbar_{g,n}}\cB_1$ corresponds to points $(C_i,x^i_1,\ldots,x^i_n,u_i)$ of $\cB_i$ for $i = 0,1$ along with the data of an isomorphism $\varphi:(C_0,x^0_1,\ldots,x^0_n)^\text{st}\to(C_1,x^1_1,\ldots,x^1_n)^\text{st}$ between their stabilizations. Lifting this to a point in $\cN$, smoothing the nodes and mapping it back to $\cB_0\times_{\Mbar_{g,n}}\cB_1$ shows that nodal curves are nowhere dense in $\cB_0\times_{\Mbar_{g,n}}\cB_1$ as well. 
Finally, on the complement of the nodal curves \eqref{proper-birational-map} has an explicit inverse described as follows. Given $(C_i,x^i_1,\ldots,x^i_n,u_i)$ in $\cB_i$ for $i=0,1$ and an isomorphism $\varphi:C_0\to C_1$, we map it to the point of $\cN$ given by $(C_0,x^0_1,\ldots,x^0_n,u=(u_0,u_1\circ\varphi))$.
\end{proof} 

\begin{remark}
    It is crucial in Lemma \ref{Regularity of Product Maps} that the fibre product is taken in the sense of orbifolds (or stacks) and not over the underlying coarse moduli space $\overline{M}_{g,n}$. Indeed, when $(g,n)$ is $(1,1)$ or $(2,0)$, the corresponding map $\cN\to\cB_0\times_{\overline{M}_{g,n}}\cB_1$ is still proper but of degree $2$.
\end{remark}

\begin{lemma}\label{Zero Locus Fibre Product} The natural map
\begin{align*}
    (\fs_0^{-1}(0)\times\fs_1^{-1}(0))\times_{\cB_0\times\cB_1}\cN\to\Mbar_{g,n}(X,\fA;J)
\end{align*}
descends to a homeomorphism on the $(G_0\times G_1)$-quotient.
\end{lemma}

\begin{proof}

Continuity of the map is evident. Since the source is compact and the target is Hausdorff, it suffices to argue that we get a bijection after passing to the $(G_0\times G_1)$-quotient. Suppose we are given a point  $(C,x_1,\ldots,x_n,u:C\to X_0\times X_1)$ of $\Mbar_{g,n}(X,\fA;J)$. We get points $(C_i,x^i_1,\ldots,x^i_n,u_i:C\to X_i)$ of $\Mbar_{g,n}(X_i,A_i;J_i)$ and associated contraction maps $\kappa_i:C\to C_i$ for $i=0,1$ by applying the map $\Phi$ from \eqref{forget-from-product}. Since $\cK_i$ is a global Kuranishi chart for $i=0,1$, we can lift $(C_i,x^i_1,\ldots,x^i_n,u_i:C\to X_i)$ to a point $(C_i,x^i_1,\ldots,x^i_n,u_i,\iota_i:C_i\to\bP^{N_i})\in\fs_i^{-1}(0)$ which is unique up to the action of $G_i$. The contraction maps $C\to C_i$ and the maps $\iota_i:C_i\to\bP^{N_i}$ now uniquely determine a map $C\to\bP^{N_0}\times\bP^{N_1}$ whose stability follows from that of $u$. Thus, each point in $\Mbar_{g,n}(X,\fA;J)$ has an inverse image in $(\fs_0^{-1}(0)\times\fs_1^{-1}(0))\times_{\cB_0\times\cB_1}\cN$ which is unique up to the action of $G_0\times G_1$.
\end{proof}

    Lemma \ref{Zero Locus Fibre Product} shows that
    \begin{align*}
        \cK_\Psi := \Psi^*(\cK_0\times\cK_1)
    \end{align*}
    defines a rel--$C^\infty$ global Kuranishi chart for $\Mbar_{g,n}(X,\fA;J)$.

\begin{lemma}\label{Equivalent Charts for Product} $\cK_\Psi$ is stably complex rel--$C^\infty$ equivalent to $\cK$ from \eqref{product-independent-chart}.
\end{lemma}

\begin{proof} 
This is a straightforward adaptation of the argument used in \textsection\ref{equivalence-of-charts} to prove Theorem \ref{global-kuranishi-existence}\eqref{chart-uniqueness-up-to-equivalence}.
\end{proof}

\begin{remark}
Lemma \ref{Equivalent Charts for Product} corresponds to the comparison of obstruction theories in \cite[Proposition 6]{Beh99}.
\end{remark}

The following fact will be crucial for our proof of the product formula.

\begin{lemma}\label{Image and Degree of Product Map} Assume $2g-2 + n > 0$ and $(g,n)$ is neither $(1,1)$ nor $(2,0)$. Then, we have
	\begin{align}\label{degree-product-map 1} 
	{\Psi}_*\fcl{\cN} = \normalfont(\text{st}\times\text{st})^*\pcd([\Delta])\cap \fcl{\cB_0\times\cB_1}
	\end{align}
	in the Borel--Moore homology of $\cB_0\times\cB_1$ over $\bQ$.
\end{lemma}

\begin{proof}
Since $(g,n)$ is neither $(1,1)$ nor $(2,0)$, the forgetful map $\Mbar_{g,n}\to\overline{M}_{g,n}$ from the moduli stack to the coarse moduli space is an isomorphism over a Zariski open subset $\overline{M}_{g,n}^*\subset\overline{M}_{g,n}$. From this and Lemma \ref{Regularity of Product Maps}, it follows that $\Psi$ maps $\cN$ onto the closed subscheme $$\cB_{01}:=\cB_0\times_{\overline{M}_{g,n}}\cB_1\subset\cB_0\times\cB_1$$ 
birationally. 
Let $\cN'\subset\cB_{01}$ be the maximal Zariski open subset for which $\cN'\to\overline{M}_{g,n}$ has image contained in $\overline{M}^*_{g,n}$, the restriction $\Psi^{-1}(\cN')\to\cN'$ is an isomorphism, and the projection $\cB_0\times\cB_1\to\overline{M}_{g,n}\times\overline{M}_{g,n}$ is a submersion at the points of $\cN'$.\par 
Then, the set $F:=\cB_{01}\setminus\cN'$ is Zariski closed in $\cB_0\times\cB_1$ and $\dim_\bC F\le\dim_\bC\cN' - 1$. By construction, it follows that \eqref{degree-product-map 1} holds over the complement of $F$ in $\cB_0\times\cB_1$. To conclude that \eqref{degree-product-map 1} holds over all of $\cB_0\times\cB_1$, we use the excision exact sequence in Borel--Moore homology and the fact that the Borel--Moore homology of $F$ is supported in degrees $\le\dim_\bR\cN-2$.
\end{proof}

\begin{remark}\label{stacky-cases}
    Replacing $\pcd([\Delta])\cap(\cdot)$ in Lemma \ref{Image and Degree of Product Map} by a Gysin pullback along the map $\Delta$ of orbifolds, it is possible to extend Lemma \ref{Image and Degree of Product Map} to cover the cases when $(g,n)$ is $(1,1)$ or $(2,0)$. We do not pursue this generalization here.
\end{remark}

\begin{proof}[Proof of Theorem \ref{gw-product}]
    The left side is \emph{a priori} defined using the global Kuranishi chart $\cK$, but by virtue of Lemma \ref{Equivalent Charts for Product} we can replace $\cK$ by $\cK_\Psi$.
    Let $\tilde\cT/\cN$ be the thickening in the global Kuranishi chart $\cK_\Psi$ and let 
    $$\tilde\Psi:\tilde\cT\to\cT_0\times\cT_1$$ 
    be the natural map covering $\Psi$. Let $\cN'\subset \cB_{01}$ be the subset defined in the proof of Lemma \ref{Image and Degree of Product Map} and let $\tilde\cT'$ be the preimage of $\Psi\inv(\cN')$  under the forgetful map $\tilde{\pi}\cl \tilde\cT\to \cN$. As $\tilde{\pi}$ is a submersion, $\tilde\cT'$ is an open submanifold whose complement has real codimension $\geq 2$. Thus, the argument used to prove Lemma \ref{Image and Degree of Product Map} implies that we also have the identity
    \begin{align}\label{fund-class-identity-thickening}
        \tilde\Psi_*[\tilde\cT/(G_0\times G_1)] = (\text{st}\times\text{st})^*\pcd([\Delta])\cap \fcl{(\cT_0/G_0)\times(\cT_1/G_1)}
    \end{align}
    in the Borel--Moore homology of $\cT_0/G_0\times\cT_1/G_1$ over $\bQ$.
    This implies the desired result once we recall the definition of the virtual fundamental class from \eqref{vfc-de}.
\end{proof}

\appendix
\section{Holomorphic line bundles on families of curves}\label{line-bundles-on-families-of-curves}

\subsection{Case of a single curve}
Let $C$ be a prestable curve. Consider the following natural exact sequence of sheaves of abelian groups on $C$ in the usual (i.e., complex analytic) topology.
\begin{align*}
    0\to\bZ\xrightarrow{2\pi \text{i}}\cO_C\xrightarrow{\exp}\cO_C^\times\to 0.
\end{align*}
Since $C$ is connected, applying $H^0$ to this exact sequence gives 
\begin{align*}
    0\to\bZ\xrightarrow{2\pi \text{i}}\bC\xrightarrow{\exp}\bC^\times\to0.
\end{align*}
The remainder of the long exact sequence in sheaf cohomology contains
\begin{align*}
    0\to H^1(C,\bZ)\xrightarrow{2\pi\text{i}} H^1(C,\cO_C)\xrightarrow{\exp} H^1(C,\cO_C^\times)\to H^2(C,\bZ)\to 0,
\end{align*}
where we have used the vanishing of $H^2(C,\cO_C)$ to get the last $0$. Using the \v{C}ech description of sheaf cohomology, we see that $\text{Pic}(C):= H^1(C,\cO_C^\times)$ is the group of isomorphism classes of holomorphic line bundles on $C$ under the tensor product operation, i.e., the Picard group of $C$. The map $H^1(C,\cO_C^\times)\to H^2(C,\bZ)$ is given by the first Chern class \cite[Chapter 1, pages 139--143]{Griffiths-Harris}. From this, we get
\begin{align}\label{fundamental-pic-iso}
    \frac{H^1(C,\cO_C)}{H^1(C,\bZ)}\xrightarrow{\simeq} \text{Pic}^0(C):=\ker(H^1(C,\cO_C^\times)\xrightarrow{c_1}H^2(C,\bZ)).
\end{align}
The group $H^1(C,\cO_C)$ has the alternate Dolbeault description 
\begin{align*}
    H^1(C,\cO_C) = \coker(\delbar:\Omega^0(C,\bC)\to\Omega^{0,1}(\tilde C,\bC)),
\end{align*}
where $\tilde C\to C$ is the normalization. It will be useful to have an explicit description of the isomorphism \eqref{fundamental-pic-iso} in Dolbeault terms.

\begin{lemma}[Line bundle from Dolbeault class]\label{line-bundle-dolbeault}
    Let $\alpha\in\Omega^{0,1}(\tilde C,\bC)$. Then, the image of $[\alpha]\in H^1(C,\cO_C)$ in $\normalfont\text{Pic}^0(C)$ is $[\cL_\alpha]$, where  $\cL_\alpha$ is holomorphic line bundle defined by equipping $C\times\bC$ with the Cauchy--Riemann operator
    \begin{align*}
        \Omega^0(C,\bC)&\to\Omega^{0,1}(\tilde C,\bC)\\
        F&\mapsto \delbar F - F\alpha.
    \end{align*}
\end{lemma}
\begin{proof}
    Using the $\delbar$-Poincar\'e lemma, choose an open cover $\{U_i\}_i$ of $C$ and smooth functions $f_i:U_i\to\bC$ such that $\delbar f_i = \alpha|_{U_i}$. Near a node, we are applying the $\delbar$-Poincar\'e lemma on $\tilde C$ and then adjusting the primitives by adding constants to ensure that they agree on the inverse images of the node in $\tilde C$. 
    
    The $\cO_C$-valued \v{C}ech $1$-cocycle $\{f_{ij}\}_{i,j}$ defined by
    \begin{align*}
        f_{ij} = f_i|_{U_i\cap U_j}-f_j|_{U_i\cap U_j}
    \end{align*}
    represents $[\alpha]\in H^1(C,\cO_C)$. This means its image in $\text{Pic}^0(C)$ is the isomorphism class of the holomorphic line bundle $\cL$ obtained by gluing the trivial holomorphic line bundles $U_i\times\bC\to U_i$ using the transition functions
    \begin{align*}
        (U_i\times\bC)|_{U_i\cap U_j}\xrightarrow{\simeq} (U_j\times\bC)|_{U_i\cap U_j}
    \end{align*}
    given by $g_{ij} = \exp(f_{ij}):U_i\cap U_j\to\bC^\times$. Define the smooth isomorphism
    \begin{align*}
        \varphi:\cL\to C\times\bC    
    \end{align*}
    of complex line bundles using $\exp(f_i):U_i\to\bC^\times$ over the open cover $\{U_i\}_i$. This is well-defined since we have $\exp(f_i) = g_{ij}\exp(f_j)$ over $U_i\cap U_j$. Pushing forward the Cauchy--Riemann operator $\delbar_\cL$ under the isomorphism $\varphi$, we obtain the Cauchy--Riemann operator on $C\times\bC$ which is given on $U_i\subset C$ by the expression
    \begin{align*}
        F\mapsto \exp(f_i)\delbar(F\exp(-f_i)) = \delbar F-F\cdot\alpha|_{U_i}.
    \end{align*}
    Thus, $\cL$ and $\cL_\alpha$ are isomorphic as holomorphic line bundles.
\end{proof}

\begin{corollary}[Image of $H^1(C,\bZ)$ via Dolbeault classes]\label{integral-dolbeault}
    Let $\alpha\in\Omega^{0,1}(\tilde C,\bC)$. Then, $[\alpha]\in H^1(C,\cO_C)$ lies in the image of $2\pi\normalfont\text{i}:H^1(C,\bZ)\to H^1(C,\cO_C)$ or, equivalently, lies in the kernel of $\exp:H^1(C,\cO_C)\to\normalfont\text{Pic}^0(C)$, if and only if there exists a smooth function $\varphi:C\to\bC^\times$ such that $\alpha = \varphi^{-1}\delbar\varphi$.
\end{corollary}
\begin{proof}
    We know that $[\alpha]$ lies in the image of $2\pi\text{i}:H^1(C,\bZ)\to H^1(C,\cO_C)$ if and only if the line bundle $\cL_\alpha$ from Lemma \ref{line-bundle-dolbeault} is holomorphically trivial. Now, $\cL_\alpha$ is holomorphically trivial if and only if it has a nowhere vanishing global holomorphic section. The proof is complete once we note that a global holomorphic section of $\cL_\alpha$ is the same as a smooth function $\varphi:C\to\bC$ such that $\delbar \varphi - \varphi\alpha = 0$.
\end{proof}

\subsection{Case of a family of curves} Let $\cS$ be a smooth quasi-projective scheme and let $\pi:\cC\to\cS$ be a flat projective algebraic family of prestable genus $g$ curves on $\cS$. In particular, $\cC$ is a reduced scheme. Since $\cS$ is quasi-projective, its underlying smooth manifold is second countable and thus has partitions of unity.

For $s\in\cS$, we denote the fibre of $\pi$ over $s$ by $C_s$ and its normalization by $\tilde C_s\to C_s$. Let $\cC^\circ\subset\cC$ the subset over which the map $\pi$ is a submersion, i.e., $\cC^\circ$ is the complement of the nodes in the fibres of $\pi$. Note that $\cC^\circ\subset\cC$ is an open subset in the usual topology.\footnote{It can also be shown to be open in the Zariski topology, but we do not need this here.}

\begin{definition}[Hodge bundle and its dual]
    Define the \emph{Hodge bundle} to be the coherent sheaf $\bH_{\cC/\cS} := \pi_*\omega_{\cC/\cS}$ where $\omega_{\cC/\cS}$ is the relative dualizing line bundle of $\pi$. Define the \emph{dual Hodge bundle} to be the coherent sheaf $\bH^*_{\cC/\cS}:=R^1\pi_*\cO_C$.
\end{definition}

\begin{lemma}\label{Vector Bundle via Homology of Structural Line}
$\bH_{\cC/\cS}$ and $\bH_{\cC/\cS}^*$ are algebraic vector bundles of rank $g$ on $\cS$ which are dual to each other and their fibres over any $s\in S$ are canonically identified with $H^0(C_s,\omega_{C_s})$ and $H^1(C_s,\cO_{C_s})$ respectively.
\end{lemma}

\begin{proof} 
The fact that $\bH_{\cC/\cS}$ and $\bH_{\cC/\cS}^*$ are locally free sheaves with the claimed description of fibres is a consequence of the theorem on cohomology and base change \cite[Theorem III.12.11]{Har77}. The assertion that they are dual is a consequence of Serre duality \cite[Theorem III.7.6]{Har77} for curves.
\end{proof}

For the remainder of this appendix, we will not be using the scheme structures of $\cC$ and $\cS$ except indirectly via Lemma \ref{Vector Bundle via Homology of Structural Line}.\footnote{In fact, with some more work, Lemma \ref{Vector Bundle via Homology of Structural Line} would also go through in the complex analytic category, so the algebraic structures are not actually necessary.} Thus, we will use `open' to mean `open in the usual (i.e. complex analytic) topology'. 

Let $\Lambda^{0,1}T^*_{\cC^\circ/\cS}$ be the bundle of $(0,1)$-forms on the fibres of $\cC^\circ\to S$, i.e., it is the $\bC$-antilinear dual of the vertical tangent bundle of $\cC^\circ\to\cS$. Consider sections
\begin{align}\label{dolbeault-forms-basis}
    \alpha_1,\ldots,\alpha_g\in C^\infty_c(\cC^\circ,\Lambda^{0,1}T^*_{\cC^\circ/\cS})
\end{align}
of compact support. The compact support condition ensures that, for any $s\in\cS$, the restrictions $\alpha_1(s),\ldots,\alpha_g(s)$ are supported away from the nodes of $C_s$ and therefore make sense as elements of $\Omega^{0,1}(\tilde C_s,\bC)$. 

\begin{lemma}\label{dolbeault-triv-point}
    Given any $s\in\cS$, there exist $\alpha_1,\ldots,\alpha_g$ as in \eqref{dolbeault-forms-basis} so that 
    \begin{align*}
        [\alpha_1(s)],\ldots,[\alpha_g(s)]\in H^1(C_s,\cO_{C_s})    
    \end{align*}
    forms a basis of $H^1(C_s,\cO_{C_s})$.    
\end{lemma}
\begin{proof}
    Let $U_s\subset C_s\cap\cC^\circ$ be any open subset of $C_s$ which meets each irreducible component of $C_s$. We claim that the natural map
    \begin{align*}
        \Omega^{0,1}_c(U_s,\bC)\to H^1(C_s,\cO_{C_s})
    \end{align*}
    is then surjective. Indeed, if this map is not surjective, then Serre duality provides us with a non-zero element $\gamma\in H^0(C_s,\omega_{C_s})$ which annihilates the image of this map. This implies that $\gamma$ identically vanishes on $U_s$ and therefore on the whole of $C_s$ by analytic continuation, which is a contradiction. 
    
    Thus, we may choose elements $\beta_1,\ldots,\beta_g\in\Omega^{0,1}_c(U_s,\bC)$ so that $[\beta_1],\ldots,[\beta_g]$ constitute a basis of $H^1(C_s,\cO_{C_s})$. It is now easy to find compactly supported $\alpha_1,\ldots,\alpha_g$ as in \eqref{dolbeault-forms-basis} so that $\alpha_i(s) = \beta_i$.
\end{proof}

\begin{lemma}[Local Dolbeault trivializations of $\bH^*_{\cC/\cS}$]\label{dolbeault-triv}
    Given $\alpha_1,\ldots,\alpha_g$ as in \eqref{dolbeault-forms-basis}, the set of points $s\in\cS$ for which
    \begin{align}\label{dolbeault-basis}
        [\alpha_1(s)],\ldots,[\alpha_g(s)]\in H^1(C_s,\cO_{C_s})   
    \end{align}
    forms a basis of $H^1(C_s,\cO_{C_s})$ is open in $\cS$. Over this open subset, the basis \eqref{dolbeault-basis} provides a $C^\infty$ trivialization of the vector bundle $\bH^*_{\cC/\cS}$.
\end{lemma}
\begin{proof}
    Let $s\in\cS$ be a point where \eqref{dolbeault-basis} is a basis. Let $\gamma_1,\ldots,\gamma_g$ be holomorphic sections of $\omega_{\cC/\cS}$, defined over a neighborhood of $s$, which give a local holomorphic trivialization of $\bH_{\cC/\cS} = \pi_*\omega_{\cC/\cS}$. The $g\times g$ matrix valued function
    \begin{align*}
        s'\mapsto M(s') = \left(\int_{C_{s'}\cap\cC^\circ}\gamma_i(s')\wedge\alpha_j(s')\right)_{i,j}
    \end{align*}
    satisfies $\det M(s)\ne 0$, where we are using the non-degeneracy of the Serre duality pairing. Since the $\gamma_i$ are holomorphic and the $\alpha_j$ are $C^\infty$ (and supported away from the nodes), $s'\mapsto M(s')$ is a $C^\infty$ function and we have $\det M(s')\ne 0$ for all $s'$ near $s$. Using the duality between $\bH_{\cC/\cS}$ and $\bH^*_{\cC/\cS}$, we see that
    \begin{align*}
        s'\mapsto[\alpha_1(s')],\ldots,[\alpha_g(s')]
    \end{align*} 
    indeed defines a $C^\infty$ trivialization of $\bH_{\cC/\cS}^*$ near $s$.
\end{proof}

\begin{lemma}[Injectivity of exp near zero section of $\bH^*_{\cC/\cS}$]\label{injectivity-of-exp}
    There is an open neighborhood $U$ of the zero section in $\bH^*_{\cC/\cS}$ with the following property. For each $s\in \cS$, the restriction of the exponential map
    \begin{align*}
        U\cap H^1(C_s,\cO_{C_s})\xrightarrow{\exp}\normalfont\text{Pic}^0(C_s)
    \end{align*}
    is injective.
\end{lemma}
\begin{proof}
    Fix a smooth Hermitian inner product on $\bH^*_{\cC/\cS}$ and let $\|\cdot\|$ be the resulting norm. It will be enough to produce a continuous function 
    \begin{align*}
        r:\cS\to(0,+\infty)    
    \end{align*}
    with the following property: for any $s\in\cS$, if $c_s\in H^1(C_s,\cO_{C_s})$ lies in the kernel of $\exp$ and also satisfies $\|c_s\|<r(s)$, then $c_s = 0$. Indeed, once we produce such $r$, we can take $U$ to be the open disc bundle of radius $\frac12r$ in $\bH^*_{\cC/\cS}$. Note that if $r_1,r_2:\cS\to(0,+\infty)$ are two functions with the above property then, for any continuous function $\chi:\cS\to[0,1]$, the function $\chi r_1 + (1-\chi)r_2$ also satisfies the same property. Thus, once we find such a function $r$ locally near every $s\in\cS$, we may patch these local functions using a partition of unity. 
    
    To find such a function $r$ near a given $s\in\cS$, it is enough to find an open set $V\subset\bH^*_{\cC/\cS}$ containing $0\in H^1(C_s,\cO_{C_s})$ such that if $c_{s'}\in H^1(C_{s'},\cO_{C_{s'}})\cap V$ lies in the kernel of $\exp$, then $c_{s'} = 0$. For the sake of a contradiction, assume that no such neighborhood $V$ of exists. Then, we can find a sequence $s_\nu\to s$ in $\cS$ and \emph{nonzero} elements $c_\nu\in H^1(C_{s_\nu},\cO_{C_{s_\nu}})$ converging\footnote{This convergence is taking place in the total space of $\bH^*_{\cC/\cS}$.} to $0\in H^1(C_s,\cO_{C_s})$ such that $\exp(c_\nu)$ is the isomorphism class of the trivial holomorphic line bundle on $C_{s_\nu}$ for all $\nu$. Choosing a local trivialization of $\bH^*_{\cC/\cS}$ near $s$ as in Lemmas \ref{dolbeault-triv-point} and \ref{dolbeault-triv}, given by $(0,1)$-forms $\alpha_1,\ldots,\alpha_g$ as in \eqref{dolbeault-forms-basis}, we may represent $c_\nu$ by
    \begin{align*}
        \beta_\nu = c_{\nu,1}\cdot\alpha_1(s_\nu) + \cdots + c_{\nu,g}\cdot\alpha_g(s_\nu)\in\Omega^{0,1}(\tilde C_{s_\nu},\bC),
    \end{align*}
    for some uniquely determined $c_{\nu,1},\ldots,c_{\nu,g}\in\bC$. By assumption, we have $c_{\nu,i}\to 0$ as $\nu\to\infty$ for each $1\le i\le g$. Thus, the $(0,1)$-forms $\beta_\nu$, which are compactly supported away from the nodes, converge uniformly (with all derivatives) to the $(0,1)$-form $0\in\Omega^{0,1}(\tilde C_s,\bC)$. Lemma \ref{integral-dolbeault} and the assumption on $\exp(c_\nu)$ yield smooth functions $\varphi_\nu:C_{s_\nu}\to\bC^\times$ such that 
    \begin{align}\label{linear-delbar-equation}
        \delbar\varphi_\nu = \varphi_\nu\beta_\nu.    
    \end{align}
    Since $\beta_\nu$ is supported away from the nodes of $C_{s_\nu}$, \eqref{linear-delbar-equation} shows that $\varphi_\nu$ is holomorphic near the nodes. After rescaling $\varphi_\nu$ suitably, we will assume that $\|\varphi_\nu\|_{L^\infty} = 1$ for all $\nu$. Using the $L^\infty$ bound on $\varphi_\nu$, the convergence of $\beta_\nu$ to $0$ and the \emph{linear} Cauchy--Riemann equation \eqref{linear-delbar-equation} satisfied by $\varphi_\nu$, we can pass to a subsequence so that $\varphi_\nu$ converges uniformly away from the nodes (with all derivatives) to a holomorphic function $\varphi:C_s\cap\cC^\circ\to\bC$. By a fact from complex analysis, stated in Lemma \ref{RRS}, the singularities are removable and we get a holomorphic function $\varphi:C_s\to\bC$.

    Since $C_s$ is connected and compact, $\varphi$ must be constant. Let $\zeta\in\bC$ be this constant value. We claim that $\|\varphi_\nu-\zeta\|_{L^\infty}$ converges to $0$. This is already clear away from the nodes. Near the nodes (where each $\beta_\nu$ is zero), this follows from the last assertion of Lemma \ref{RRS}. From $\|\varphi_\nu-\zeta\|_{L^\infty}\to 0$, we conclude that $|\zeta| = 1$ and that $\varphi_\nu$ has a well-defined logarithm, i.e., a smooth function $\psi_\nu:C_{s_\nu}\to\bC$ such that $\varphi_\nu = \exp(\psi_\nu)$ for $\nu\gg 1$. 
    
    This means $\beta_\nu = \varphi_\nu^{-1}\delbar\varphi_\nu = \delbar\psi_\nu$ is $\delbar$-exact for $\nu\gg 1$. This contradicts the assumption that the cohomology classes $c_\nu = [\beta_\nu]$ were all nonzero.
\end{proof}

\begin{lemma}[Analytic functions on degenerating annuli]\label{RRS}
    Let $t_\nu\in\bC$ be sequence converging to $0\in\bC$. For each $\nu$, let $f_\nu:A_{t_\nu}\to\bC$ be a holomorphic function on
    \begin{align*}
        A_{t_\nu} = \{(z,w)\in\bC^2:|z|\le 1,\,|w|\le 1,\,zw = t_\nu\}.
    \end{align*}
    If the functions $z\mapsto f_\nu(z,t_\nu/z)$ converge uniformly on $\{z\in\bC:\frac12\le|z|\le 1\}$ to some holomorphic function $F$ and the functions $w\mapsto f_\nu(t_\nu/w,w)$ converge uniformly on $\{w\in\bC:\frac12\le|w|\le 1\}$ to some holomorphic function $G$. 
    Then, there exists a holomorphic function $f:A_0\to\bC$ on
    \begin{align*}
        A_0 = \{(z,w)\in\bC^2:|z|\le 1,\,|w|\le 1,\,zw = 0\}
    \end{align*}
    satisfying $f(z,0) = F(z)$ for $\frac12\le|z|\le 1$ and $f(0,w) = G(w)$ for $\frac12\le|w|\le 1$. If $F$ and $G$ are both identically zero, then $\|f_\nu\|_{L^\infty}$ converges to $0$ as $\nu\to\infty$.
\end{lemma}
\begin{proof}
    The idea for this proof is taken from \cite{RRS}. Using Laurent expansions for analytic functions on the annuli $\frac12\le |z|\le 1$ and $\frac12\le |w|\le 1$, we may write
    \begin{align*}
        f_\nu(z,t_\nu/z) = \sum_{n\in\bZ} a_{\nu,n}z^n&\xrightarrow{\nu\to\infty} F(z) = \sum_{n\in\bZ} a_nz^n,\\
        f_\nu(t_\nu/w,w) = \sum_{n\in\bZ} b_{\nu,n}w^n&\xrightarrow{\nu\to\infty} G(w) = \sum_{n\in\bZ} b_nw^n.
    \end{align*}
    For each $\nu$ and $n\ge 0$, we have $a_{\nu,-n} = b_{\nu,n}t_\nu^n$ and $b_{\nu,-n} = a_{\nu,n}t_\nu^n$. In the limit, we therefore have $a_0 = b_0$ and $a_{-n} = b_{-n} = 0$ for all $n\ge 1$. 
    
    The functions $F(z)$ and $G(w)$ have no negative powers of $z$ and $w$ in their respective Laurent expansions. So, they extend to holomorphic functions defined for $0\le |z|\le 1$ and $0\le |w|\le 1$ respectively. Moreover, $a_0 = b_0$ implies that $F(0) = G(0)$. Thus, there exists a well-defined holomorphic function $f:A_0\to\bC$ with the desired properties.

    For the last assertion, note that $F = 0$ and $G = 0$ implies that 
    \begin{align*}
        \sup_{\frac12\le|z|\le 1}|f_\nu(z,t_\nu/z)|\quad\text{and}\quad\sup_{\frac12\le|w|\le 1}|f_\nu(t_\nu/w,w)|
    \end{align*}
    converge to zero. The maximum principle for holomorphic functions applied to each $f_\nu$, now shows that $\|f_\nu\|_{L^\infty}$ also converges to zero.
\end{proof}

\begin{lemma}[Abel--Jacobi type map for $\cC/\cS$]\label{abel-jacobi-construction}
    There exists a holomorphic map
    \begin{align*}
        \rho_{\cC/\cS}:\cC^\circ\times_\cS\cC^\circ\to\bH^*_{\cC/\cS},
    \end{align*}
    defined near the diagonal $\Delta_{\cC^\circ}\subset\cC^\circ\times_\cS\cC^\circ$, with the following property. 
    
    For any $s\in\cS$ and $y_1,y_2\in C_s$ such that $(y_1,y_2)$ is in the domain of $\rho_{\cC/\cS}$, the element $\rho_{\cC/\cS}(y_1,y_2)\in H^1(C_s,\cO_{C_s})$ satisfies
    \begin{align*}
        \exp\rho_{\cC/\cS}(y_1,y_2) = [\cO_{C_s}(y_1-y_2)]\in\Pic^0(C_s).   
    \end{align*}
    Moreover, $\rho_{\cC/\cS}(y_1,y_2)$ vanishes whenever $y_1 = y_2$.
\end{lemma}
\begin{proof}
    Given a point $x\in\cC^\circ$, we will define a holomorphic function $\rho_x$ on a neighborhood $V_x\subset\cC^\circ\times_\cS\cC^\circ$ of $(x,x)$ such that $\rho_x$ vanishes on $\Delta_{\cC^\circ}\cap V_x$ and we have
    \begin{align*}
        \exp\rho_x(y_1,y_2) = [\cO_{C_s}(y_1-y_2)]\in\text{Pic}^0(C_s)
    \end{align*}
    for any $(y_1,y_2)\in V_x$ with image $s\in\cS$. By shrinking $V_x$, we can then arrange for $\rho_x$ to map into the open subset $U\subset\bH^*_{\cC/\cS}$ from Lemma \ref{injectivity-of-exp}. Repeating this at each point of $\cC^\circ$, we obtain functions which automatically agree on overlaps by Lemma \ref{injectivity-of-exp} and patch together to define $\rho_{\cC/\cS}$ in a neighborhood of $\Delta_{\cC^\circ}$.

    We now construct $\rho_x$. Choose local holomorphic coordinates $(z_1,\ldots,z_n,w)$ on $\cC^\circ$ centred at $x$ so that $\pi:\cC\to\cS$ is given by $(z_1,\ldots,z_n,w)\mapsto(z_1,\ldots,z_n)$.
    Assume that the image of the coordinate chart $(z_1,\ldots,z_n,w)$ is the polydisc of polyradius $(1,\ldots,1,1)$ centred at $0\in\bC^{n+1}$. Consider the open subsets 
    \begin{align*}
        W := \{|w|<1\}\subset\cC^\circ\quad\text{and}\quad W' := \cC\setminus\{|w|\le\textstyle\frac12\}\subset\cC.    
    \end{align*}
    Then, for every $s\in\cS$ lying in the domain of the coordinate functions $(z_1,\ldots,z_n)$, the sets $W_s := W\cap C_s$ and $W'_s := W'\cap C_s$ constitute an open cover of $C_s$.

    Define $V_x$ to consist of those points $(y_1,y_2)\in\cC^\circ\times_\cS\cC^\circ$ for which $y_1$ and $y_2$ lie in our coordinate neighborhood and their respective $w$-coordinates $\xi_1$ and $\xi_2$ satisfy $|\xi_1|<\frac12$ and $|\xi_2|<\frac12$. Now, consider any $(y_1,y_2)\in V_x$ and let its image in $\cS$ be $s$. With respect to the open cover $\{W_s,W'_s\}$ of $C_s$, the line bundle $\cO_{C_s}(y_1-y_2)$ is represented by the transition function 
    \begin{align*}
        W_s\cap W'_s\to\bC^\times,\quad
        \textstyle w\mapsto\frac{w-\xi_1}{w-\xi_2}.
    \end{align*}
    By considering winding numbers around the origin in $\bC$, we find that there is a well-defined holomorphic function $L(w,\xi_1,\xi_2)$, defined for $\frac12<|w|<1$, $|\xi_1|<\frac12$ and $|\xi_2|<\frac12$, so that $L(w,\xi_1,\xi_2) = 0$ whenever $\xi_1 = \xi_2$ and we have
    \begin{align*}
        \exp (L(w,\xi_1,\xi_2)) = \textstyle\frac{w-\xi_1}{w-\xi_2}.
    \end{align*}
    Define $\rho_x(y_1,y_2)\in H^1(C_s,\cO_{C_s})$ to be the class of the $\cO_{C_s}$-valued \v{C}ech $1$-cocyle for the cover $\{W_s,W'_s\}$ given by the function $L(\cdot,\xi_1,\xi_2)$ on $W_s\cap W'_s$. By construction $\rho_x$ vanishes when $y_1 = y_2$. To construct a Dolbeault representative for $\rho_x(y_1,y_2)$, choose a smooth function $\chi:\bR\to[0,1]$ satisfying 
    \begin{align*}
        \chi(t) = \begin{cases}
            1 & \text{if }t\le\textstyle\frac12+\epsilon\\
            0 & \text{if }t\ge 1-\epsilon
        \end{cases}
    \end{align*}
    for some $0<\epsilon<\frac14$.
    Then, the $(0,1)$-form 
    \begin{align*}
        \alpha_x(y_1,y_2) = \delbar\chi(|\cdot|)L(\cdot,\xi_1,\xi_2)    
    \end{align*}
    on $C_s$ represents $\rho_x(y_1,y_2)$. Indeed, it has $\delbar$-primitives $-(1-\chi(|\cdot|))L(\cdot,\xi_1,\xi_2)$ on $W_s$ and $\chi(|\cdot|)L(\cdot,\xi_1,\xi_2)$ on $W_s'$ such that their difference on $W_s\cap W'_s$ is the holomorphic function $L(\cdot,\xi_1,\xi_2)$.
    
    We will use the Serre duality pairing to show that $\rho_x$ is holomorphic. Choose local holomorphic sections $\gamma_1,\ldots,\gamma_g$ trivializing $\pi_*\omega_{\cC/\cS} = \bH_{\cC/\cS}$ near $s$, and write $\gamma_j = f_j(z_1,\ldots,z_n,w)\,dw$ in local coordinates. Since the expression
    \begin{align*}
        \int_{C_s}\alpha_x(y_1,y_2)\wedge\gamma_j(s) = \int f_j(z_1,\ldots,z_n,w)L(w,\xi_1,\xi_2)\textstyle\frac{\partial}{\partial\overline w}\chi(|w|)\,d\overline w\wedge dw
    \end{align*}
    is clearly holomorphic in the coordinates $(z_1,\ldots,z_n,\xi_1,\xi_2)$ on $\cC^\circ\times_\cS\cC^\circ$ for all $1\le j\le g$, we conclude that $\rho_x$ is holomorphic.
\end{proof}
\section{Refined rel--$C^\infty$ representability}\label{rel-smooth-addendum}

We work in the setting established in \textsection\ref{subsec:rel-smooth-review}. Recall that we fixed a complex manifold $\fB$ and an analytic family of prestable genus $g$ curves $\pi:\fC\to\fB$ on it, with the fibre of $\pi$ over any $s\in\fB$ denoted by $C_s$. We also fixed an almost complex manifold $(X,J)$ and a homology class $A\in H_2(X,\bZ)$. 

Following \cite[\textsection 4]{Swa21}, we summarize the key steps in the construction of a local rel--$C^\infty$ chart representing the functor $\fM^\text{reg}(\pi,X)_A$ from Definition \ref{rel-smooth-moduli-functors} near a point $(s\in\fB,u:C_s\to X)$, where $u$ is a $J$-holomorphic map in class $A$ such that linearized Cauchy--Riemann operator $D(\delbar_J)_u$ is surjective. The construction resembles the usual gluing analysis (e.g., as explained in \cite[Appendix B]{P16}), except for our invocation of the polyfold implicit function theorem. We emphasize only the parts of the construction which are relevant here.

From this, we extract a refinement of the representability statement, which is absent from \cite{Swa21}, although Lemma 6.3 therein has a similar spirit.

\subsection{Setup, following \cite[\textsection4.1--4.2]{Swa21}}\label{gluing-setup}

Write $C:=C_s$ for simplicity and write $j_0$ for the almost complex structure on $C$. Let $\nu_1,\ldots,\nu_d$ be an enumeration of the nodes of $C$ and let $C^\circ$ be the complement of the nodes in $C$. For $1\le l\le d$, choose holomorphic coordinates $z_l,z_l'$, defined on discs on either side of $\nu_j$. Note that $z_l,z_l'$ define coordinates $(s_l,t_l),(s_l',t_l')\in \bR_+\times S^1$ via $s_l+\text{i}t_l = -\log z_l$ and $s_l' + \text{i}t_l' = -\log z_l'$ on the cylindrical ends of $C^\circ$. We fix a metric on $X$ which is the flat metric with respect to chosen local coordinates near the image of each node under $u$. This allows us to define exponential maps and parallel transport.

Following \cite[Lemma 4.1]{Swa21}, we reduce to the case when $\fC\to\fB$ is a versal deformation of the curve $C$ near $s\in\fB$ and has the following standard form. First, we have a product decomposition $\fB = \bG\times V$, where $\bG\subset\bC^d$ is an open neighborhood of $0$ and $V$ is a finite dimensional vector space parametrizing a smooth family of almost complex structures $\{j(v)\}_{v\in V}$ on $C$ with $j(0) = j_0$ and $j(v)\equiv j_0$ on the coordinate neighborhoods of the nodes. Second, the point $s$ is given by $(0,0)$ with respect to $\fB = \bG\times V$ and the fibre over any point $(\alpha,v)$, with $\alpha = (\alpha_1,\ldots,\alpha_d)\in \bG$ and $v\in V$, is given by equipping $C$ with the almost complex structure $j(v)$ and replacing the neighborhood $\{z_lz_l'=0\}$ of each node $\nu_l$ by the plumbing $\{z_lz_l'=\alpha_l\}$. When $\alpha_l = \exp(-(R_l + \text{i}\theta_l))\ne 0$, this amounts to truncating the cylindrical ends of $C^\circ$ on either side of $\nu_j$ to $(0,R_l)\times S^1$ and identifying them using $s_l+s_l'=R_l$ and $t_l+t_l'=\theta_l$. We will denote the fibre of $\fC\to\fB$ over $(\alpha,v)$ by $(C_\alpha,j_\alpha(v))$.

Fix a $C^\infty$ function $\beta:\bR\to[0,1]$ with $\beta(s)\equiv 1$ for $s\le-1$, $\beta(s)+\beta(-s)\equiv 1$ for all $s$ and $-1\le\beta'(s)<0$ for $|s|<1$. Using the functions $\beta_R(s) = \beta(s-R/2)$ for $R>0$, we can `pre-glue' the map $u:C\to X$ to a map $\oplus_\alpha u:C_\alpha\to X$ for any small value of the gluing parameter $\alpha\in \bG$. Explicitly, if $\alpha = (\alpha_1,\ldots,\alpha_d)$ and $1\le l\le d$ is such that $\alpha_l = \exp(-(R_l+i\theta_l))\ne 0$, then we define
\begin{align*}
    (\oplus_\alpha u)(s_l,t_l) = \beta_{R_l}(s_l)\cdot u(s_l,t_l) + \beta_{R_l}(s_l')\cdot u(s_l',t_l')
\end{align*}
on the region $\{z_lz_l'=\alpha_l\}$, using the flat local coordinates on $X$ near $u(\nu_l)$. The map $\oplus_\alpha u$ is defined to be the same as $u$ everywhere else on $C_\alpha$. More generally, given a section $\xi$ of $u^*T_X$ on $C$, we may use the same formula to `pre-glue' it to a section $\oplus_\alpha\xi$ of $u_\alpha^*T_X$ on $C_\alpha$. In particular, $\oplus_\alpha\xi$ coincides with $\xi$ away from the cylindrical ends which get plumbed. Note that $\xi$ cannot be recovered just from the data of $\alpha$ and its `pre-gluing' $\oplus_\alpha\xi$, because we multiply $\xi$ by (translates) of the cutoff function $\beta$. For fixed $\alpha$, among all possible $\xi$ which give rise to the same $\oplus_\alpha\xi$, there is a unique choice whose `anti-gluing' $\ominus_\alpha\xi$ vanishes. The precise formula \cite[Equation (4.2.7)]{Swa21} for `anti-gluing' is not relevant for us here.

\begin{remark}
    Informally, the upshot of this discussion is the following: any pseudo-holomorphic map from $(C_\alpha,j_\alpha(v))$ to $(X,J)$, which is sufficiently close to $\oplus_\alpha u$, can be written as $\exp_{\oplus_\alpha u}(\oplus_\alpha\xi)$ for a unique $\xi$ subject to the condition $\ominus_\alpha\xi = 0$. This helps motivate the functional analytic setup in \textsection\ref{gluing-ift} below.
\end{remark}

In the above discussion, given $\alpha = (\alpha_1,\ldots,\alpha_d)\in \bG$, we plumbed $C$ according to the equations $\{z_lz'_l=\alpha_l\}$. This is known in the polyfold literature as the `logarithmic gluing profile'. An alternative that is better suited to applying the sc-implicit function theorem is the `exponential gluing profile', where we instead plumb according to the equations $\{z_lz_l' = \varphi_{\exp}(\alpha_l)\}$, where
\begin{align*}
    \varphi_{\exp}(\alpha_l) = \exp(-2\pi (e^{1/|\alpha_l|}-e))\cdot\textstyle\frac{\alpha_l}{|\alpha_l|}
\end{align*}
for $\alpha_l\ne 0$ and $\varphi_{\exp}(0) = 0$. With this in mind, define $\bG_{\exp}$ to be the $C^\infty$ manifold which has the same underlying topological space as $\bG$ but whose $C^\infty$ structure is induced from the continuous embedding $(\alpha_1,\ldots,\alpha_d)\mapsto(\varphi_{\exp}(\alpha_1),\ldots,\varphi_{\exp}(\alpha_d))$ into $\bC^d$. Note that the identity map $\bG_{\exp}\to \bG$ is $C^\infty$ but its inverse is not $C^\infty$.

\subsection{Chart construction, following \cite[\textsection4.3--4.4]{Swa21}}\label{gluing-ift}

Define the weighted Sobolev spaces of sections
\begin{align*}
    H^{k,\delta}:=W^{k,2,\delta}(C,u^*T_X)\quad\text{and}\quad H^{k-1,\delta}:=W^{k-1,2,\delta}(\tilde C,\Omega^{0,1}_{\tilde C}\otimes\tilde u^*T_X)
\end{align*}
as in \cite[\textsection B.4]{P16}, for some integer $k\ge 10$ and real number $0<\delta<1$, with $\tilde C$ being the normalization of $C$. Following \cite{HWZ17}, fix an increasing sequence $(\delta_m)_{m\ge 0}$ in $(0,1)$ with $\delta_0 = \delta$ and view these as sc-Hilbert spaces with scales given by $H^{k+m,\delta_m}$ and $H^{k+m-1,\delta_m}$ for integers $m\ge 0$. Following \cite[Lemma 4.10]{Swa21}, fix an sc-complement $E\subset H^{k,\delta}$ of the finite dimensional subspace
\begin{align*}
    K:=\ker(D(\delbar_J)_u:H^{k,\delta}\to H^{k-1,\delta}).
\end{align*}
Being an sc-Hilbert space in its own right, $E$ has scales $E_m = E\cap H^{k+m,\delta_m}$ for each $m\ge 0$. We have $K\subset H^{k+m,\delta_m}$ for all $m\ge 0$ by elliptic regularity.

Applying the sc-implicit function theorem \cite[Theorem 4.6]{HWZ-ImpFuncThms} to the sc-Fredholm operator $\delbar_J$ yields the following. First, for small $(\alpha,v,\kappa)$ in $\bG\times V\times K$, there is a unique small element $\xi_{\alpha,v,\kappa}\in E$ such that $\ominus_\alpha(\xi_{\alpha,v,\kappa}) = 0$ and the map
\begin{align*}
    \exp_{\oplus_\alpha u}\left(\oplus_\alpha(\kappa+\xi_{\alpha,v,\kappa})\right):(C_\alpha,j_\alpha(v))\to (X,J)
\end{align*}
is pseudo-holomorphic and has surjective linearization. Second, the map
\begin{align}
    \label{gluing-map-1} \bG_{\exp}\times V\times K&\to E\subset H^{k,\delta}\\
    \label{gluing-map-2}(\alpha,v,\kappa)&\mapsto\xi_{\alpha,v,\kappa}
\end{align}
is sc-smooth. It is crucial that we use $\bG_{\exp}$ here, rather than $\bG$.

Note that the first assertion can be obtained via usual gluing analysis (e.g., as in \cite[Appendix B]{P16}). It is the second assertion is where the full strength of the sc-implicit function theorem is actually used. 

At this point, using the sc-smoothness of \eqref{gluing-map-1}--\eqref{gluing-map-2}, one can show that 
$$\bG\times V\times K\to \bG\times V$$ 
provides a rel--$C^\infty$ representing object for $\fM^\text{reg}(\pi,X)_A$ near $(s\in\fB,u:C_s\to X)$. See the proof of \cite[Theorem 4.16]{Swa21} for details. Note that we are able to change $\bG_{\exp}$ back to $\bG$ here since the rel--$C^\infty$ structure forgets the smooth structure in the $\fB$-direction. 

This concludes our summary of the key steps in the proof of Theorem \ref{rel-smooth-representability}.

\subsection{Refinement}\label{gluing-refinement}

We use the above discussion to refine Theorem \ref{rel-smooth-representability}.

\begin{lemma}\label{gluing-map-smooth}
    The assignment from \eqref{gluing-map-2} defines a $C^\infty$ map
    \begin{align*}
        \bG_{\exp}\times V\times K&\to H^{k+m,\delta_m}\\
        (\alpha,v,\kappa)&\mapsto\xi_{\alpha,v,\kappa}
    \end{align*}
    of Banach manifolds for all integers $m\ge 0$.
\end{lemma}
\begin{proof}
    We already know that the map $\bG_{\exp}\times V\times K\to H^{k,\delta}$ induced by the assignment $(\alpha,v,\kappa)\mapsto\xi_{\alpha,v,\kappa}$ is sc-smooth. As $\bG_{\exp}\times V\times K$ is finite dimensional, all its scales coincide, i.e., $(\bG_{\exp}\times V\times K)_m = \bG_{\exp}\times V\times K$ for all $m\ge 0$. The desired result is now a formal consequence of \cite[Proposition 2.3]{HWZ10}.
\end{proof}

In terms of the notation established above, let $C^\text{int}\subset C$ be the complement of the closure of the cylindrical ends in $C^\circ$. Note that we have canonical embeddings $C^\text{int}\subset C_\alpha$ for all small values of $\alpha\in \bG$.

\begin{lemma}\label{eval-smooth}
    The following is a $C^\infty$ map of manifolds.
    \begin{align*}
        \normalfont \bG_{\exp}\times V\times K\times C^\text{int}&\to X \\
        (\alpha,v,\kappa,w)&\mapsto\exp_{u(w)}(\kappa(w) + \xi_{\alpha,v,\kappa}(w)).
    \end{align*}
\end{lemma}
\begin{proof}
    Fix an integer $m\ge 0$.    
    Lemma \ref{gluing-map-smooth} shows that $(\alpha,v,\kappa)\mapsto \kappa + \xi_{\alpha,v,\kappa}$ is $C^\infty$ as a map $\bG_{\exp}\times V\times K\to H^{k+m,\delta_m}$. The Sobolev embedding theorem gives a continuous restriction map $H^{k+m,\delta_m}\to C^m(C^\text{int},u^*T_X)$. Using the smoothness of the exponential map on $X$ and working locally on $C^\text{int}$ to trivialize $u^*T_X$, we are left to consider evaluation map
    \begin{align*}
        C^m(\bD,\bR)\times\bD\to\bR,\quad
        (f,w)\mapsto f(w)
    \end{align*}
    where $\bD\subset\bC$ is a disc, which is evidently $C^m$. This shows that the map in the statement of the lemma is a composition of $C^m$ maps and is therefore $C^m$. As $m\ge 0$ was arbitrary, we get the desired result.
\end{proof}

We can now state a refinement of Theorem \ref{rel-smooth-representability}. Denote the complement of the nodes in the fibres of $\pi$ by $\fC^\circ\subset\fC$ and note that $\pi$ restricts to a (holomorphic) submersion $\pi^\circ:\fC^\circ\to\fB$. In particular, we may naturally regard $\fC^\circ/\fB$ as a rel--$C^\infty$ manifold. If we take any commutative diagram
\begin{center}
\begin{tikzcd}
    \fC \arrow[d,"\pi"] & \arrow[l] \fC_T \arrow[d] & \arrow[l] \fC_Z \arrow[r,"f"] \arrow[d] & X \\
    \fB & \arrow[l,"\varphi"] T & \arrow[l] Z &
\end{tikzcd}
\end{center}
belonging to $\fM^\text{reg}(\pi,X)_A(Z/T)$ for some rel--$C^\infty$ manifold $Z/T$ (as in Definition \ref{rel-smooth-moduli-functors}), then restricting it over $\fC^\circ\subset\fC$ yields a commutative diagram
\begin{center}
\begin{tikzcd}
    \fC^\circ \arrow[d,"\pi^\circ"] & \arrow[l] \fC_T^\circ \arrow[d] & \arrow[l] \fC_Z^\circ \arrow[r,"f^\circ"] \arrow[d] & X \\
    \fB & \arrow[l,"\varphi"] T & \arrow[l] Z &
\end{tikzcd}
\end{center}
where $\fC^\circ_T/T$ and $\fC^\circ_Z/Z$ are the pullbacks of $\pi^\circ$ and the maps $f^\circ$ is the restriction of $f$. Note that $\fC^\circ_T/T$ is naturally a rel--$C^\infty$ manifold by pulling back the rel--$C^\infty$ structure of $\fC^\circ/\fB$. Thus, from the rel--$C^\infty$ structures on $Z/T$ and $\fC^\circ_T/T$, we obtain a natural rel--$C^\infty$ structure on their fibre product $\fC_Z^\circ/T$. 

\begin{proposition}[Refinement of Theorem \ref{rel-smooth-representability}]\label{extra-smoothness}
    For any rel--$C^\infty$ manifold $Z/T$ and any diagram in $\normalfont\fM^\text{reg}(\pi,X)_A(Z/T)$ as above, $f$ restricts to a rel--$C^\infty$ map
    \begin{align*}
        f^\circ:\fC^\circ_Z/T\to X/\normalfont\text{pt}.
    \end{align*}
\end{proposition}
\begin{proof}
    We reduce to the universal case, i.e., $Z/T = \fM^\text{reg}(\pi,X)_A/\fB$ and $f$ is the universal pseudo-holomorphic map over it. This is possible because a diagram over a general $Z/T$ is obtained from the universal case by a rel--$C^\infty$ pullback. 
    
    Given a point $(s,u)\in\fM^\text{reg}(\pi,X)_A$ and a point $\zeta\in C_s\cap\fC^\circ$, we must show that $f^\circ$ is rel--$C^\infty$ near $(s,u,\zeta)$. Clearly, this can be checked after passing to a local rel--$C^\infty$ chart for $\fM^\text{reg}(\pi,X)_A$  near $(s,u)$. 
    
    Construct a chart near $(s,u)$, as in \textsection\ref{gluing-setup}--\ref{gluing-ift}, ensuring that the cylindrical ends on $C_s$ are chosen to be disjoint from a neighborhood of $\zeta$ in $C_s$. In terms of the notation introduced above Lemma \ref{eval-smooth}, this means we have $\zeta\in C^\text{int}$. The fact that $f^\circ$ is rel--$C^\infty$ near $(s,u,\zeta)$ now follows from Lemma \ref{eval-smooth}.
\end{proof}

\bibliographystyle{amsalpha}
\bibliography{HGGW}

\providecommand{\bysame}{\leavevmode\hbox to3em{\hrulefill}\thinspace}
\providecommand{\MR}{\relax\ifhmode\unskip\space\fi MR }
\providecommand{\MRhref}[2]{%
  \href{http://www.ams.org/mathscinet-getitem?mr=#1}{#2}
}
\providecommand{\href}[2]{#2}
\begin{thebibliography}{HWZ21}

\bibitem[ACG11]{ACG-moduli}
E.~Arbarello, M.~Cornalba, and P.~A. Griffiths, \emph{Geometry of algebraic curves. {V}olume {II}}, Grundlehren der Mathematischen Wissenschaften [Fundamental Principles of Mathematical Sciences], vol. 268, Springer, Heidelberg, 2011, With a contribution by Joseph Daniel Harris. \MR{2807457}

\bibitem[AMS21]{AMS21}
M.~Abouzaid, M.~McLean, and I.~Smith, \emph{Complex cobordism, {H}amiltonian loops and global {K}uranishi charts}, arXiv:2110.14320 (2021).

\bibitem[AMS24]{AMS23}
\bysame, \emph{{G}romov-{W}itten invariants in complex and {M}orava-local {K}-theories}, arXiv:2307.01883, to appear in \textit{Geometric and Functional Analysis} (2024).

\bibitem[Beh99]{Beh99}
K.~Behrend, \emph{The product formula for {G}romov-{W}itten invariants}, J. Algebraic Geom. \textbf{8} (1999), no.~3, 529--541. \MR{1689355}

\bibitem[BM96]{behrend-manin}
K.~Behrend and Yu. Manin, \emph{Stacks of stable maps and {G}romov-{W}itten invariants}, Duke Math. J. \textbf{85} (1996), no.~1, 1--60. \MR{1412436}

\bibitem[Bre97]{Bre12}
G.~E. Bredon, \emph{Sheaf theory}, second ed., Graduate Texts in Mathematics, vol. 170, Springer-Verlag, New York, 1997. \MR{1481706}

\bibitem[BX22]{BX22}
S.~Bai and G.~Xu, \emph{An integral {E}uler cycle in normally complex orbifolds and {Z}-valued {G}romov-{W}itten type invariants}, arXiv:2201.02688 (2022).

\bibitem[CM07]{CM}
K.~Cieliebak and K.~Mohnke, \emph{Symplectic hypersurfaces and transversality in {G}romov-{W}itten theory}, J. Symplectic Geom. \textbf{5} (2007), no.~3, 281--356. \MR{2399678}

\bibitem[FO99]{FO99}
K.~Fukaya and K.~Ono, \emph{Arnold conjecture and {G}romov-{W}itten invariant}, Topology \textbf{38} (1999), no.~5, 933--1048. \MR{1688434}

\bibitem[FP97]{FP97}
W.~Fulton and R.~Pandharipande, \emph{Notes on stable maps and quantum cohomology}, Algebraic geometry---{S}anta {C}ruz 1995, Proc. Sympos. Pure Math., vol.~62, Amer. Math. Soc., Providence, RI, 1997, pp.~45--96. \MR{1492534}

\bibitem[GH94]{Griffiths-Harris}
P.~Griffiths and J.~Harris, \emph{Principles of algebraic geometry}, Wiley Classics Library, John Wiley \& Sons, Inc., New York, 1994, Reprint of the 1978 original. \MR{1288523}

\bibitem[Gro66]{EGA4part3}
A.~Grothendieck, \emph{\'{E}l\'{e}ments de g\'{e}om\'{e}trie alg\'{e}brique. {IV}. \'{E}tude locale des sch\'{e}mas et des morphismes de sch\'{e}mas. {III}}, Inst. Hautes \'{E}tudes Sci. Publ. Math. (1966), no.~28, 255. \MR{217086}

\bibitem[Har77]{Har77}
R.~Hartshorne, \emph{Algebraic geometry}, Graduate Texts in Mathematics, No. 52, Springer-Verlag, New York-Heidelberg, 1977. \MR{0463157}

\bibitem[Hir23]{Hir23}
A.~Hirschi, \emph{Properties of {G}romov-{W}itten invariants defined via global {K}uranishi charts}, arXiv:2312.03625 (2023).

\bibitem[HW24]{HW24}
A.~Hirschi and L.~Wang, \emph{On {D}onaldson's 4-6 question}, arXiv:2309.07041 (2024).

\bibitem[HWZ09]{HWZ-ImpFuncThms}
H.~Hofer, K.~Wysocki, and E.~Zehnder, \emph{A general {F}redholm theory {II}: Implicit function theorems}, Geometric and Functional Analysis \textbf{19} (2009), no.~1, 206--293. \MR{2507223}

\bibitem[HWZ10]{HWZ10}
\bysame, \emph{sc-smoothness, retractions and new models for smooth spaces}, Discrete Contin. Dyn. Syst. \textbf{28} (2010), no.~2, 665--788. \MR{2644764}

\bibitem[HWZ17]{HWZ17}
\bysame, \emph{Applications of polyfold theory {I}: {T}he polyfolds of {G}romov-{W}itten theory}, Mem. Amer. Math. Soc. \textbf{248} (2017), no.~1179, v+218. \MR{3683060}

\bibitem[HWZ21]{HWZ21}
\bysame, \emph{Polyfold and {F}redholm theory}, Ergebnisse der Mathematik und ihrer Grenzgebiete. 3. Folge. A Series of Modern Surveys in Mathematics [Results in Mathematics and Related Areas. 3rd Series. A Series of Modern Surveys in Mathematics], vol.~72, Springer, Cham, [2021] \copyright 2021. \MR{4298268}

\bibitem[Hyv12]{Hy12}
C.~Hyvrier, \emph{A product formula for {G}romov-{W}itten invariants}, J. Symplectic Geom. \textbf{10} (2012), no.~2, 247--324. \MR{2926997}

\bibitem[IP19]{IP19}
E.-N. Ionel and T.~H. Parker, \emph{Thin compactifications and relative fundamental classes}, J. Symplectic Geom. \textbf{17} (2019), no.~3, 703--752. \MR{4022212}

\bibitem[KM94]{KM94}
M.~Kontsevich and Yu. Manin, \emph{Gromov-{W}itten classes, quantum cohomology, and enumerative geometry}, Comm. Math. Phys. \textbf{164} (1994), no.~3, 525--562. \MR{1291244}

\bibitem[KM96]{KM96}
\bysame, \emph{Quantum cohomology of a product}, Invent. Math. \textbf{124} (1996), no.~1-3, 313--339, With an appendix by R. Kaufmann. \MR{1369420}

\bibitem[LT98]{LT98}
J.~Li and G.~Tian, \emph{Virtual moduli cycles and {G}romov-{W}itten invariants of general symplectic manifolds}, Topics in symplectic {$4$}-manifolds ({I}rvine, {CA}, 1996), First Int. Press Lect. Ser., I, Int. Press, Cambridge, MA, 1998, pp.~47--83. \MR{1635695}

\bibitem[MS12]{MS12}
D.~McDuff and D.~Salamon, \emph{{$J$}-holomorphic curves and symplectic topology}, second ed., American Mathematical Society Colloquium Publications, vol.~52, American Mathematical Society, Providence, RI, 2012. \MR{2954391}

\bibitem[MW17]{MW17b}
D.~McDuff and K.~Wehrheim, \emph{Smooth {K}uranishi atlases with isotropy}, Geom. Topol. \textbf{21} (2017), no.~5, 2725--2809. \MR{3687107}

\bibitem[Pal61]{Pal61}
R.~Palais, \emph{On the existence of slices for actions of non-compact lie groups}, Annals of Mathematics \textbf{73} (1961), 295.

\bibitem[Par16]{P16}
J.~Pardon, \emph{An algebraic approach to virtual fundamental cycles on moduli spaces of pseudo-holomorphic curves}, Geom. Topol. \textbf{20} (2016), no.~2, 779--1034. \MR{3493097}

\bibitem[RRS08]{RRS}
J.~W. Robbin, Y.~Ruan, and D.~A. Salamon, \emph{The moduli space of regular stable maps}, Math. Z. \textbf{259} (2008), no.~3, 525--574. \MR{2395126}

\bibitem[RT95]{RT95}
Y.~Ruan and G.~Tian, \emph{A mathematical theory of quantum cohomology}, J. Differential Geom. \textbf{42} (1995), no.~2, 259--367. \MR{1366548}

\bibitem[RT97]{RT97}
\bysame, \emph{Higher genus symplectic invariants and sigma models coupled with gravity}, Invent. Math. \textbf{130} (1997), no.~3, 455--516. \MR{1483992}

\bibitem[Sie98]{siebert-98}
B.~Siebert, \emph{{G}romov-{W}itten invariants of general symplectic manifolds}, arXiv:dg-ga/9608005 (1998).

\bibitem[SPa24]{stacks-project}
The {S}tacks~{P}roject authors, \emph{The {S}tacks {P}roject}, \url{https://stacks.math.columbia.edu}, 2024.

\bibitem[Swa21]{Swa21}
M.~Swaminathan, \emph{Rel-{$C^\infty$} structures on {G}romov-{W}itten moduli spaces}, J. Symplectic Geom. \textbf{19} (2021), no.~2, 413--473. \MR{4325409}

\end{thebibliography}

\Addresses

\end{document}